\documentclass[a4paper,10pt]{article}

\usepackage[english]{babel}

\usepackage{lastpage}
\usepackage{amsmath,amsfonts,amssymb,amsthm,latexsym,dsfont,pictexwd,
  color,bbm,mathtools,mathrsfs}

\usepackage{geometry}
\RequirePackage{bera} 

\usepackage[utf8]{inputenc}
\usepackage[T1]{fontenc}
\usepackage[normalem]{ulem}

\newtheoremstyle{ejpecpbodyit}
{3pt}
{3pt}
{\itshape}
{}
{\bfseries\sffamily}
{.}
{ }
{}
\theoremstyle{ejpecpbodyit}
\newtheorem{theoremOwn}{Theorem}%
\newtheorem{proposition}{Proposition}[section]%
\newtheorem{lemma}[proposition]{Lemma}%
\newtheorem{definition}[proposition]{Definition}%
\newtheorem{notation}[proposition]{Notation}%

\newtheoremstyle{ejpecpbodyrm}
{3pt}
{3pt}
{}
{}
{\bfseries\sffamily}
{.}
{ }
{}

\theoremstyle{ejpecpbodyrm}%
\newtheorem{example}[proposition]{Example}%
\newtheorem{remark}[proposition]{Remark}%

\newtheoremstyle{step}{3pt}{0pt}{}{}{\bf}{}{.5em}{}
\theoremstyle{step} 

\makeatletter
\renewcommand{\section}{\@startsection%
  {section}
  {1}
  {0em}
  {\baselineskip}
  {0.5\baselineskip}
  {\normalfont\large\bfseries}}
\renewcommand{\subsection}{\@startsection%
  {subsection}
  {2}
  {0em}
  {\baselineskip}
  {0.25\baselineskip}
  {\normalfont\bfseries}
}
\makeatother

\DeclareMathAlphabet{\mathpzc}{OT1}{pzc}{m}{it}

\newcommand{\suml}{\sum\limits}

\newcommand{\ind}[1]{\mathbbm{1}_{\{#1\}}} 

\numberwithin{equation}{section}

\newcommand{\nul}{
  \parbox[b]{1cm}{
    \beginpicture
    \setcoordinatesystem units <.5cm,.5cm>
    \setplotarea x from 0 to 1.5, y from 0 to 0.5
    \put{$\bullet$} [cC] at 0.5 0.2
    \endpicture}}

\newcommand{\eins}{
  \parbox[b]{1cm}{
    \beginpicture
    \setcoordinatesystem units <.5cm,.5cm>
    \setplotarea x from 0 to 1.5, y from 0 to 0.5
    \circulararc -180 degrees from 0 0 center at 0.5 0
    \endpicture}}

\newcommand{\zwei}{
  \parbox[b]{1cm}{
    \beginpicture
    \setcoordinatesystem units <.5cm,.5cm>
    \setplotarea x from 0 to 1.5, y from 0 to 0.5
    \circulararc -180 degrees from 0 0 center at 0.5 0
    \ellipticalarc axes ratio 1:1.5 -180 degrees from 0 0 center at 0.5 0
    \endpicture}}

\newcommand{\drei}{
  \parbox[b]{1cm}{
    \beginpicture
    \setcoordinatesystem units <.5cm,.5cm>
    \setplotarea x from 0 to 2.5, y from 0 to 0.5
    \circulararc -180 degrees from 0 0 center at 0.5 0
    \circulararc -180 degrees from 1 0 center at 1.5 0
    \endpicture}}

\newcommand{\vier}{
  \parbox[b]{1cm}{
    \beginpicture
    \setcoordinatesystem units <.5cm,.5cm>
    \setplotarea x from 0 to 3, y from 0 to 0.5
    \circulararc -180 degrees from 0 0 center at 0.5 0
    \circulararc -180 degrees from 1.5 0 center at 2 0
    \endpicture}}

\newcommand{\fuenf}{
  \parbox[b]{1cm}{
    \beginpicture
    \setcoordinatesystem units <.5cm,.5cm>
    \setplotarea x from 0 to 1.5, y from 0 to 1
    \circulararc -180 degrees from 0 0 center at 0.5 0
    \ellipticalarc axes ratio 1:1.5 -180 degrees from 0 0 center at 0.5 0
    \ellipticalarc axes ratio 1:2 -180 degrees from 0 0 center at 0.5 0
    \endpicture}}

\newcommand{\sechs}{
  \parbox[b]{1cm}{
    \beginpicture
    \setcoordinatesystem units <.5cm,.5cm>
    \setplotarea x from 0 to 2.5, y from 0 to 0.5
    \circulararc -180 degrees from 0 0 center at 0.5 0
    \ellipticalarc axes ratio 1:1.5 -180 degrees from 0 0 center at 0.5 0
    \circulararc -180 degrees from 1 0 center at 1.5 0
    \endpicture}}

\newcommand{\sieben}{
  \parbox[b]{1cm}{
    \beginpicture
    \setcoordinatesystem units <.5cm,.5cm>
    \setplotarea x from 0 to 2.5, y from 0 to 1
    \circulararc -180 degrees from 0 0 center at 0.5 0
    \circulararc -180 degrees from 1 0 center at 1.5 0
    \ellipticalarc axes ratio 1.5:1 -180 degrees from 0 0 center at 1 0
    \endpicture}}

\newcommand{\acht}{
  \parbox[b]{1cm}{
    \beginpicture
    \setcoordinatesystem units <.5cm,.5cm>
    \setplotarea x from 0 to 3, y from 0 to 0.5
    \circulararc -180 degrees from 0 0 center at 0.5 0
    \circulararc -180 degrees from 1.5 0 center at 2 0
    \ellipticalarc axes ratio 1:1.5 -180 degrees from 0 0 center at 0.5 0
    \endpicture}}

\newcommand{\neun}{
  \parbox[b]{1cm}{
    \beginpicture
    \setcoordinatesystem units <.5cm,.5cm>
    \setplotarea x from 0 to 3.5, y from 0 to 1
    \circulararc -180 degrees from 0 0 center at 0.5 0
    \circulararc -180 degrees from 1 0 center at 1.5 0
    \ellipticalarc axes ratio 1.5:1 -180 degrees from 1 0 center at 2 0
    \endpicture}}

\newcommand{\einsnull}{
  \parbox[b]{1cm}{
    \beginpicture
    \setcoordinatesystem units <.5cm,.5cm>
    \setplotarea x from 0 to 3.5, y from 0 to 0.5
    \circulararc -180 degrees from 0 0 center at 0.5 0
    \circulararc -180 degrees from 1 0 center at 1.5 0
    \circulararc -180 degrees from 2 0 center at 2.5 0
    \endpicture}}

\newcommand{\einseins}{
  \parbox[b]{1cm}{
    \beginpicture
    \setcoordinatesystem units <.5cm,.5cm>
    \setplotarea x from 0 to 4, y from 0 to 0.5
    \circulararc -180 degrees from 0 0 center at 0.5 0
    \circulararc -180 degrees from 1 0 center at 1.5 0
    \circulararc -180 degrees from 2.5 0 center at 3 0
    \endpicture}}

\newcommand{\einszwei}{
  \parbox[b]{1cm}{
    \beginpicture
    \setcoordinatesystem units <.5cm,.5cm>
    \setplotarea x from 0 to 4.5, y from 0 to 0.5
    \circulararc -180 degrees from 0 0 center at 0.5 0
    \circulararc -180 degrees from 1.5 0 center at 2 0
    \circulararc -180 degrees from 3 0 center at 3.5 0
    \endpicture}}

\newcommand{\einsdrei}{
  \parbox[b]{1cm}{
    \beginpicture
    \setcoordinatesystem units <.5cm,.5cm>
    \setplotarea x from 0 to 1.5, y from 0 to 1.5
    \circulararc -180 degrees from 0 0 center at 0.5 0
    \ellipticalarc axes ratio 1:1.5 -180 degrees from 0 0 center at 0.5 0
    \ellipticalarc axes ratio 1:2 -180 degrees from 0 0 center at 0.5 0
    \ellipticalarc axes ratio 1:2.5 -180 degrees from 0 0 center at 0.5 0
    \endpicture}}

\newcommand{\einsvier}{
  \parbox[b]{1cm}{
    \beginpicture
    \setcoordinatesystem units <.5cm,.5cm>
    \setplotarea x from 0 to 2.5, y from 0 to 1
    \circulararc -180 degrees from 0 0 center at 0.5 0
    \ellipticalarc axes ratio 1:1.5 -180 degrees from 0 0 center at 0.5 0
    \ellipticalarc axes ratio 1:2 -180 degrees from 0 0 center at 0.5 0
    \circulararc -180 degrees from 1 0 center at 1.5 0
    \endpicture}}

\newcommand{\einsfuenf}{
  \parbox[b]{1cm}{
    \beginpicture
    \setcoordinatesystem units <.5cm,.5cm>
    \setplotarea x from 0 to 2.5, y from 0 to .5
    \circulararc -180 degrees from 0 0 center at 0.5 0
    \ellipticalarc axes ratio 1:1.5 -180 degrees from 0 0 center at 0.5 0
    \circulararc -180 degrees from 1 0 center at 1.5 0
    \ellipticalarc axes ratio 1:1.5 -180 degrees from 1 0 center at 1.5 0
    \endpicture}}

\newcommand{\einssechs}{
  \parbox[b]{1cm}{
    \beginpicture
    \setcoordinatesystem units <.5cm,.5cm>
    \setplotarea x from 0 to 2.5, y from 0 to 0.5
    \circulararc -180 degrees from 0 0 center at 0.5 0
    \ellipticalarc axes ratio 1:1.5 -180 degrees from 0 0 center at 0.5 0
    \circulararc -180 degrees from 1 0 center at 1.5 0
    \circulararc -180 degrees from 0 0 center at 1 0
    \endpicture}}

\newcommand{\einssieben}{
  \parbox[b]{1cm}{
    \beginpicture
    \setcoordinatesystem units <.5cm,.5cm>
    \setplotarea x from 0 to 3.5, y from 0 to 0.5
    \circulararc -180 degrees from 0 0 center at 0.5 0
    \ellipticalarc axes ratio 1:1.5 -180 degrees from 0 0 center at 0.5 0
    \circulararc -180 degrees from 1 0 center at 1.5 0
    \circulararc -180 degrees from 2 0 center at 2.5 0
    \endpicture}}

\newcommand{\einsacht}{
  \parbox[b]{1cm}{
    \beginpicture
    \setcoordinatesystem units <.5cm,.5cm>
    \setplotarea x from 0 to 3.5, y from 0 to 0.5
    \circulararc -180 degrees from 0 0 center at 0.5 0
    \circulararc -180 degrees from 1 0 center at 1.5 0
    \ellipticalarc axes ratio 1:1.5 -180 degrees from 1 0 center at 1.5 0
    \circulararc -180 degrees from 2 0 center at 2.5 0
    \endpicture}}

\newcommand{\einsneun}{
  \parbox[b]{1cm}{
    \beginpicture
    \setcoordinatesystem units <.5cm,.5cm>
    \setplotarea x from 0 to 3.5, y from 0 to 1
    \circulararc -180 degrees from 0 0 center at 0.5 0
    \circulararc -180 degrees from 1 0 center at 1.5 0
    \circulararc -180 degrees from 2 0 center at 2.5 0
    \ellipticalarc axes ratio 1.5:1 -180 degrees from 0 0 center at 1 0
    \endpicture}}

\newcommand{\zweinull}{
  \parbox[b]{1cm}{
    \beginpicture
    \setcoordinatesystem units <.5cm,.5cm>
    \setplotarea x from 0 to 3.5, y from 0 to 1
    \circulararc -180 degrees from 0 0 center at 0.5 0
    \circulararc -180 degrees from 1 0 center at 1.5 0
    \circulararc -180 degrees from 2 0 center at 2.5 0
    \ellipticalarc axes ratio 2:1 -180 degrees from 0 0 center at 1.5 0
    \endpicture}}

\newcommand{\zweieins}{
  \parbox[b]{1cm}{
    \beginpicture
    \setcoordinatesystem units <.5cm,.5cm>
    \setplotarea x from 0 to 3.5, y from 0 to 1
    \circulararc -180 degrees from 0 0 center at 0.5 0
    \ellipticalarc axes ratio 1:1.5 -180 degrees from 0 0 center at 0.5 0
    \circulararc -180 degrees from 1 0 center at 1.5 0
    \ellipticalarc axes ratio 1.5:1 -180 degrees from 1 0 center at 2 0
    \endpicture}}

\newcommand{\zweizwei}{
  \parbox[b]{1cm}{
    \beginpicture
    \setcoordinatesystem units <.5cm,.5cm>
    \setplotarea x from 0 to 3, y from 0 to 1
    \circulararc -180 degrees from 0 0 center at 0.5 0
    \ellipticalarc axes ratio 1:1.5 -180 degrees from 0 0 center at 0.5 0
    \circulararc -180 degrees from 1.5 0 center at 2 0
    \ellipticalarc axes ratio 1:1.5 -180 degrees from 1.5 0 center at 2 0
    \endpicture}}

\newcommand{\zweidrei}{
  \parbox[b]{1cm}{
    \beginpicture
    \setcoordinatesystem units <.5cm,.5cm>
    \setplotarea x from 0 to 3, y from 0 to 1
    \circulararc -180 degrees from 0 0 center at 0.5 0
    \ellipticalarc axes ratio 1:1.5 -180 degrees from 0 0 center at 0.5 0
    \ellipticalarc axes ratio 1:2 -180 degrees from 0 0 center at 0.5 0
    \circulararc -180 degrees from 1.5 0 center at 2 0
    \endpicture}}

\newcommand{\zweivier}{
  \parbox[b]{1cm}{
    \beginpicture
    \setcoordinatesystem units <.5cm,.5cm>
    \setplotarea x from 0 to 4.5, y from 0 to 1.5
    \circulararc -180 degrees from 0 0 center at 0.5 0
    \circulararc -180 degrees from 1 0 center at 1.5 0
    \circulararc -180 degrees from 2 0 center at 2.5 0
    \circulararc -180 degrees from 3 0 center at 3.5 0
    \endpicture}}

\newcommand{\zweifuenf}{
  \parbox[b]{1cm}{
    \beginpicture
    \setcoordinatesystem units <.5cm,.5cm>
    \setplotarea x from 0 to 4.5, y from 0 to 1.5
    \circulararc -180 degrees from 0 0 center at 0.5 0
    \circulararc -180 degrees from 1 0 center at 1.5 0
    \circulararc -180 degrees from 2 0 center at 2.5 0
    \ellipticalarc axes ratio 2:1 -180 degrees from 1 0 center at 2.5 0
    \endpicture}}

\newcommand{\zweisechs}{
  \parbox[b]{1cm}{
    \beginpicture
    \setcoordinatesystem units <.5cm,.5cm>
    \setplotarea x from 0 to 4.5, y from 0 to 1
    \circulararc -180 degrees from 0 0 center at 0.5 0
    \circulararc -180 degrees from 1 0 center at 1.5 0
    \ellipticalarc axes ratio 1.5:1 -180 degrees from 1 0 center at 2 0
    \ellipticalarc axes ratio 2:1 -180 degrees from 1 0 center at 2.5 0
    \endpicture}}

\newcommand{\zweisieben}{
  \parbox[b]{1cm}{
    \beginpicture
    \setcoordinatesystem units <.5cm,.5cm>
    \setplotarea x from 0 to 4, y from 0 to 1
    \circulararc -180 degrees from 0 0 center at 0.5 0
    \ellipticalarc axes ratio 1:1.5 -180 degrees from 0 0 center at 0.5 0
    \circulararc -180 degrees from 1.5 0 center at 2 0
    \circulararc -180 degrees from 2.5 0 center at 3 0
    \endpicture}}

\newcommand{\zweiacht}{
  \parbox[b]{1cm}{
    \beginpicture
    \setcoordinatesystem units <.5cm,.5cm>
    \setplotarea x from 0 to 4, y from 0 to 1
    \circulararc -180 degrees from 0 0 center at 0.5 0
    \ellipticalarc axes ratio 1:1.5 -180 degrees from 0 0 center at 0.5 0
    \circulararc -180 degrees from 1 0 center at 1.5 0
    \circulararc -180 degrees from 2.5 0 center at 3 0
    \endpicture}}

\newcommand{\zweineun}{
  \parbox[b]{1cm}{
    \beginpicture
    \setcoordinatesystem units <.5cm,.5cm>
    \setplotarea x from 0 to 4, y from 0 to 1
    \circulararc -180 degrees from 0 0 center at 0.5 0
    \circulararc -180 degrees from 1 0 center at 1.5 0
    \ellipticalarc axes ratio 1.5:1 -180 degrees from 0 0 center at 1 0
    \circulararc -180 degrees from 2.5 0 center at 3 0
    \endpicture}}

\newcommand{\dreinull}{
  \parbox[b]{1cm}{
    \beginpicture
    \setcoordinatesystem units <.5cm,.5cm>
    \setplotarea x from 0 to 5, y from 0 to 1.5
    \circulararc -180 degrees from 0 0 center at 0.5 0
    \circulararc -180 degrees from 1 0 center at 1.5 0
    \circulararc -180 degrees from 2.5 0 center at 3 0
    \circulararc -180 degrees from 3.5 0 center at 4 0
    \endpicture}}

\newcommand{\dreieins}{
  \parbox[b]{1cm}{
    \beginpicture
    \setcoordinatesystem units <.5cm,.5cm>
    \setplotarea x from 0 to 5, y from 0 to 1
    \circulararc -180 degrees from 0 0 center at 0.5 0
    \ellipticalarc axes ratio 1:1.5 -180 degrees from 0 0 center at 0.5 0
    \circulararc -180 degrees from 1.5 0 center at 2 0
    \circulararc -180 degrees from 3 0 center at 3.5 0
    \endpicture}}

\newcommand{\dreizwei}{
  \parbox[b]{1cm}{
    \beginpicture
    \setcoordinatesystem units <.5cm,.5cm>
    \setplotarea x from 0 to 5, y from 0 to 1
    \circulararc -180 degrees from 0 0 center at 0.5 0
    \circulararc -180 degrees from 1 0 center at 1.5 0
    \ellipticalarc axes ratio 1.5:1 -180 degrees from 1 0 center at 2 0
    \circulararc -180 degrees from 3.5 0 center at 4 0
    \endpicture}}

\newcommand{\dreidrei}{
  \parbox[b]{1cm}{
    \beginpicture
    \setcoordinatesystem units <.5cm,.5cm>
    \setplotarea x from 0 to 5, y from 0 to 1.5
    \circulararc -180 degrees from 0 0 center at 0.5 0
    \circulararc -180 degrees from 1 0 center at 1.5 0
    \circulararc -180 degrees from 2 0 center at 2.5 0
    \circulararc -180 degrees from 3.5 0 center at 4 0
    \endpicture}}

\newcommand{\dreivier}{
  \parbox[b]{1cm}{
    \beginpicture
    \setcoordinatesystem units <.5cm,.5cm>
    \setplotarea x from 0 to 5.5, y from 0 to 1.5
    \circulararc -180 degrees from 0 0 center at 0.5 0
    \circulararc -180 degrees from 1 0 center at 1.5 0
    \circulararc -180 degrees from 2.5 0 center at 3 0
    \circulararc -180 degrees from 4 0 center at 4.5 0
    \endpicture}}

\newcommand{\dreifuenf}{
  \parbox[b]{1cm}{
    \beginpicture
    \setcoordinatesystem units <.5cm,.5cm>
    \setplotarea x from 0 to 6, y from 0 to 1.5
    \circulararc -180 degrees from 0 0 center at 0.5 0
    \circulararc -180 degrees from 1.5 0 center at 2 0
    \circulararc -180 degrees from 3 0 center at 3.5 0
    \circulararc -180 degrees from 4.5 0 center at 5 0
    \endpicture}}

\newcommand{\Rand}[1]{\marginpar{#1}}
\renewcommand{\Rand}[1]{}\newcommand{\be}[1]{\Rand{\vspace{0,6cm}\tt #1}\begin{equation}\label{#1}}

  \newcommand{\ee}{\end{equation}}

\newcommand{\CC}{\mathcal{C}}
\newcommand{\CX}{\mathcal{X}}

\newcommand{\CM}{\mathcal{M}}
\newcommand{\CB}{\mathcal{B}}

\newcommand{\N}{\mathbb{N}}

\newcommand{\R}{\mathbb{R}}

\newcommand{\ve}{\varepsilon}

\makeatletter
\renewcommand{\@fnsymbol}[1]{\ensuremath{%
    \ifcase#1\or 1\or 2\or 3\or
    \mathsection\or \mathparagraph\or \|\or 1\or
    2\or 3 \else\@ctrerr\fi}}
\makeatother

\begin{document}
\title{\Large Path-properties of the tree-valued Fleming--Viot
  process} \author{Andrej Depperschmidt\thanks{Abteilung f\"ur
    Mathematische Stochastik, Albert-Ludwigs University of Freiburg,
    Eckerstr. 1, D - 79104 Freiburg, Germany, e-mail:
    depperschmidt@stochastik.uni-freiburg.de}, Andreas
  Greven\thanks{Department Mathematik, Universit\"at
    Erlangen-N\"urnberg, Cauerstr. 11, D-91058 Erlangen, Germany,
    e-mail: greven@mi.uni-erlangen.de}\; and Peter
  Pfaffelhuber\thanks{Abteilung f\"ur Mathematische Stochastik,
    Albert-Ludwigs University of Freiburg, Eckerstra\ss e 1, D - 79104
    Freiburg, Germany, e-mail: p.p@stochastik.uni-freiburg.de}}

\thispagestyle{empty}
\date{\today}


\maketitle

\begin{abstract}
We consider the tree-valued Fleming--Viot process,
  $(X_t)_{t\geq 0}$, with mutation and selection as studied in
  Depperschmidt, Greven and Pfaffelhuber (2012). This process models
  the stochastic evolution of the genealogies and (allelic) types
  under resampling, mutation and selection in the population currently
  alive in the limit of infinitely large populations. Genealogies and
  types are described by (isometry classes of) marked metric measure
  spaces. The long-time limit of the neutral tree-valued Fleming--Viot
  dynamics is an equilibrium given via the marked metric measure space
  associated with the Kingman coalescent.

  In the present paper we pursue two closely linked goals. First, we
  show that two well-known properties of the neutral Fleming--Viot
  genealogies at fixed time $t$ arising from the properties of the
  dual, namely the Kingman coalescent, hold for the whole path. These
  properties are related to the geometry of the family tree close to
  its leaves. In particular we consider the number and the size of
  subfamilies whose individuals are not further than $\ve$ apart in
  the limit $\ve\to 0$. Second, we answer two open questions about the
  sample paths of the tree-valued Fleming--Viot process. We show that
  for all $t>0$ almost surely the marked metric measure space $X_t$
  has no atoms and admits a mark function. The latter property means
  that all individuals in the tree-valued Fleming--Viot process can
  uniquely be assigned a type. All main results are proven for the
  neutral case and then carried over to selective cases via Girsanov's
  formula giving absolute continuity.\end{abstract}



\section{Introduction and  background}
A frequently used model for stochastically evolving multitype
populations is the Fleming--Viot diffusion. In the neutral case the
corresponding genealogy at a fixed time $t$ is described by the
Kingman coalescent which was introduced some~30 years ago as the
random genealogy relating the individuals of a population of large
constant size in equilibrium (\cite{Kingman1982a,Kingman1982b}).

As a genealogical tree with infinitely many leaves, the Kingman
coalescent exhibits some distinct geometric properties. In particular,
S.~Evans studied the random tree as a metric space, where the distance
of two leaves is given by the time to their most recent common
ancestor (\cite{Ev00}; see also the fine-properties of the metric
space derived in \cite{MR2534485}). Using this picture, the Kingman
coalescent close to its leaves has a nice shape: roughly speaking, in
the limit $\ve\to 0$, approximately $2/\ve$ balls of radius $\ve$ are
needed to cover the whole tree; see Section~4.2 in D.~Aldous' review
article (\cite{Ald99}). Equivalently, there are $2/\ve$ families whose
individuals have a common ancestor not further than $\varepsilon$ in
the past. Moreover, these $2/\varepsilon$ families have sizes of order
$\varepsilon$. More precisely, the size of a typical family is
exponential with parameter $2/\varepsilon$ (see eq~(35)
in\cite{Ald99}), and the empirical distribution of the family sizes
converges to this exponential distribution. However, these results
have been proved only for the genealogy of a population at a fixed
time.

In a series of papers of the authors, in part with A.~Winter,
\cite{GPWmetric09,DGP11,GPWmp,DGP12} the Kingman coalescent was
extended to a tree-valued process $(X_t)_{t\geq 0}$, where
$X_t$ gives the genealogy of an evolving population at time
$t$. The resulting process, the tree-valued Fleming--Viot process, is
connected to the Fleming--Viot measure-valued diffusion, which
describes the evolution of type-frequencies in a large (i.e.\
infinite) population of constant size. In the simplest case of neutral
evolution all individuals have the same chance to produce viable
offspring, i.e., the frequency of offspring of any subset of
individuals is a martingale. However, biologically most interesting is
the \emph{selective case} where the evolutionary success of an
individual depends on its (allelic) type and where also mutation
(i.e.\ random changes in types) may occur. This case including
mutation and selection was studied in \cite{DGP12}.

We note that rather than studying the full-tree valued process in the
infinite population limit, it is possible to obtain limits of its
functionals directly as well. For the  neutral tree-valued
Fleming-Viot process, this has been done for the height
\cite{PfaffelhuberWakolbinger2006, DelamsEtAl2010} and the length
\cite{MR2851692}. In addition, functionals of other tree-valued
processes have been studied, e.g.\ for the height of the tree in
branching processes \cite{EvansRalph2010} and for the height and
length of a population with the Bolthausen-Sznitman coalescent as
long-time limit \cite{MR2988406}.

\medskip \textbf{Goals:} The construction of the tree-valued
Fleming--Viot process allows one to ask if the above mentioned
properties of the geometry of the Kingman coalescent trees are almost
sure path properties of the tree-valued Fleming-Viot
process. Furthermore, while we gave a construction of the tree-valued
Fleming--Viot process under neutrality in~\cite{GPWmp} and under
mutation and selection in~\cite{DGP12}, some questions about path
behavior remained open. We will carry over some (not all) of the
geometric properties of the fixed random trees to the evolving paths
of trees in Theorems~\ref{T1}~--~\ref{T4} of this work.

In the next section, we explain in detail how we model
\emph{genealogical trees}. In order to formulate open questions let us
briefly mention here that we use a \emph{marked metric measure space
  (mmm-space)}, that is, a triple $(U,r,\mu)$ where $(U,r)$ is a
complete metric space describing genealogical distances between
individuals and $\mu$ is a probability measure on the Borel-$\sigma$
algebra of $U\times A$, where $A$ is the set of possible (allelic)
types. In particular, the tree-valued Fleming--Viot process
$(X_t)_{t\geq 0}$ takes values in the space of (continuous) paths in
the space of mmm-spaces.

To state two open questions from earlier work (see Remark~3.11 in
\cite{DGP12}), let $X_t = (U_t, r_t, \mu_t)$ be the state of
the tree-valued Fleming--Viot process at time $t\geq 0$.  First, we
ask if the measure~$\mu_t$ has atoms for some $t>0$. To understand
what this means, recall that the state of the measure-valued
Fleming--Viot process is purely atomic for all~$t>0$, almost
surely. However, in the tree-valued case, existence of an atom in the
measure $\mu_t \in \mathcal M_1(U_t\times A)$ implies that there
exists a set of positive $\mu_t$-mass such that individuals belonging
to this set have zero genealogical distance to each other. As we will
see in Theorem~\ref{T5}, this is not possible, and the tree-valued
Fleming--Viot process is non-atomic for all $t>0$, almost
surely. Second, we ask if every individual in $U_t$ can uniquely be
assigned a type which is of course the case for the Moran model, but
does not automatically carry over to the (infinite population)
diffusion limit. This is the case iff the support of $\mu_t$ is given
by $\{(u,\kappa_t(u)): u\in U_t\}$ for a function $\kappa_t: U_t\to
A$. In Theorem~\ref{T6}, we will see that this is indeed the case and
every individual can be assigned a type for all $t>0$, almost surely.

\medskip

\textbf{Methods:} Since the tree-valued Fleming--Viot process was
constructed using a well-posed martingale problem, we will frequently
use martingale techniques in our proofs. These allow us to study the
\emph{sample Laplace-transform} for the distance of two points of the
tree as a semi-martingale. In addition, population models have
specific features that will also be useful. For example all
individuals have unique ancestors even though not all individuals have
descendants and if an individual has a descendant, she might as well
have many. This simple structure can be used for finite population
models (e.g.\ the Moran model) or the tree-valued Fleming--Viot
process, since this infinite model arises as a large-population limit
from finite Moran models (for the neutral case see Theorem~2 of
\cite{GPWmp} and for the selective case Theorem~3 of \cite{DGP12}) to
derive properties of the family structure.

An important point of the proofs is that we can transfer properties
from the neutral case since for most forms of selection (which are
determined by the interacting fitness functions, which gives the
dependence of the offspring distribution depends on the allelic type),
the resulting process is \emph{absolutely continuous} to the
\emph{neutral case} (which comes with no dependency between allelic
type and offspring distribution) via a \emph{Girsanov transform}.

\medskip

\textbf{Outline:} The paper is organized as follows: In
Section~\ref{s.tvFV} we recall the definition of the state space of
the tree-valued Fleming--Viot process, its construction by a
well-posed martingale problem and some of its properties. In
Section~\ref{S:res}, we give our main results. Theorem~\ref{T1} states
that the law of large numbers for the number of ancestors of Kingman's
coalescent holds along the whole path of the tree-valued Fleming--Viot
process. Moreover, we discover a Brownian motion within the
tree-valued Fleming--Viot process based on the fluctuations of the
number of ancestors; see Theorem~\ref{T2}. Another law of large
numbers is obtained for a statistic concerning the family sizes and we
make a big step towards this result in Theorem~\ref{T3}. Another
Brownian motion is discovered within the tree-valued Fleming--Viot
process based on family sizes in Theorem~\ref{T4}. Finally we show the
non-atomicity along the path in Theorem~\ref{T5} and obtain existence
of a mark function in Theorem~\ref{T6}.

In Section~\ref{s.proofPcover} we prove Theorem~\ref{T1} and after some
preparatory moment computations in Section~\ref{S:prep}, we give in the
subsequent sections the remaining proofs of the main results. We note that
various proofs have been carried out using {\sc Mathematica} and can be
reproduced by the reader via the accompanying {\sc Mathematica}-file.

\section{The tree-valued Fleming--Viot process}
\label{s.tvFV}
In this section, we recall the tree-valued Fleming--Viot process given
as the unique solution of a \emph{martingale problem} on the space of
\emph{marked metric measure spaces}. The material presented here is a
condensed version of results from~\cite{GPWmetric09, DGP11, GPWmp}
and~\cite{DGP12}. We only recall notions needed to follow our
arguments in the present paper. Let us fix some notation first.

\begin{notation} \mbox{}\\%
  For \label{not:aux} a Polish space $E$ the set of all bounded
  measurable functions is denoted by $\CB(E)$, its subset containing
  the bounded and continuous functions by $\CC_b(E)$, the set of
  c\`adl\`ag function $I\subseteq \mathbb R\to E$ by $\mathcal D_E(I)$
  (which is equipped with the Skorohod topology) and the subset of
  continuous functions by $\mathcal C_E(I)$. The set of probability
  measures on (the Borel $\sigma$-algebra of) $E$ is denoted by
  $\CM_1(E)$ and $\Rightarrow$ denotes either weak convergence of
  probability measures or convergence in distribution of random
  variables. If $\phi: E \to E'$ for some Polish space $E'$ then the
  image measure of $\mu \in \CM_1(E)$ under $\phi$ is denoted by
  $\phi_\ast \mu$. For functions $\lambda\mapsto a_\lambda$ and
  $\lambda\mapsto b_\lambda$, we write $a_\lambda\lesssim b_\lambda$
  if there is $C>0$ such that $a_\lambda\leq Cb_\lambda$ uniformly for
  all $\lambda$.  Furthermore for $\lambda_0 \in \R\cup \{\pm
  \infty\}$ we write $a_\lambda \stackrel{\lambda \to
    \lambda_0}{\approx \,} b_\lambda$ if $a_\lambda$ and $b_\lambda$
  are asymptotically equivalent as $\lambda\to\lambda_0$, i.e.\ if
  $a_\lambda/b_\lambda \to 1$ as $\lambda \to \lambda_0$. For product
  spaces $E_1 \times E_2 \times \dots$ we denote the projection
  operators by $\pi_{E_1}, \pi_{E_2}, \dots$. When there is no chance
  of ambiguity we use the shorter notation $\pi_1, \pi_2,\dots$.
\end{notation}

\subsection{The state space: genealogies as marked metric measure spaces}

At any time $t \ge 0$ the state of the neutral tree-valued
Fleming-Viot process without types is a genealogical tree describing
the ancestral relations among individuals alive at time $t$. Such
trees can be encoded by ultrametric spaces and vice versa where the
distance of two individuals is given by the time back to their most
recent common ancestor. Adding selection and mutation to the process
requires that we not only keep track of the genealogical distances
between individuals but also of the type of each individual. This
leads to the concept of marked metric measure spaces which we recall
here. For more details and interpretation of the state space we refer
to Section~2.3 in \cite{DGP12} and to Remark~\ref{rem:int-st-sp}
below.

Throughout, we fix a \emph{compact metric space $A$} which we refer to
as the \emph{(allelic) type space}. An \emph{$A$-marked (ultra-)metric
  measure space}, abbreviated as $A$-mmm space or just mmm-space in
the following, is a triple $(U,r,\mu)$, where $(U,r)$ is an
ultra-metric space and $\mu\in\mathcal M_1(U\times A)$ is a
probability measure on $U\times A$.

The state space of the tree-valued Fleming--Viot process is
\begin{align}\label{eq:UA}
  \mathbb U_A \coloneqq \bigl\{\overline{(U,r,\mu)}: (U,r,\mu) \text{ is $A$-mmm
    space}\bigr\},
\end{align}
where $\overline{(U,r,\mu)}$ is the equivalence class of the $A$-mmm
space $(U,r,\mu)$, and two mmm-spaces $(U_1,r_1,\mu_1)$ and
$(U_2,r_2,\mu_2)$ are called \emph{equivalent} if there exists an
isometry (here $\text{supp}$ of a measure denotes its support)
\begin{align}
  \label{eq:UA2}
  \varphi: \text{supp} \bigl(\pi_{U_1}{}_\ast \mu_1\bigr) \to
  \text{supp}\bigl(\pi_{U_2}{}_\ast \mu_2\bigr)
\end{align}
with $(\varphi,id)_\ast \mu_1 = \mu_2$. The subspace of compact mmm-spaces
\begin{align}
  \label{eq:31k}
  \mathbb U_{A,c} & \coloneqq  \bigl\{\overline{(U,r,\mu)} \in \mathbb U_A:
  (U,r) \text{ compact} \bigr\} \subsetneq \mathbb U_A
\end{align}
will play an important role.

\begin{remark}[Interpretation of equivalent marked metric measure
  spaces]\mbox{}\\%
  1.\ In \label{rem:int-st-sp} our presentation, only
  \emph{ultra}-metric spaces $(U,r)$ will appear. The reason is that
  we only consider stochastic processes whose state at time $t$
  describes the genealogy of the population alive at time $t$, which
  makes $r$ an ultra-metric.

  \medskip

  \noindent 2.\ There are several reasons why we consider equivalence
  classes of marked metric spaces instead of the marked metric spaces
  themselves. The most important is that we view a genealogical tree as a metric
  space on its set of leaves. Since in population genetic models the individuals
  are regarded as exchangeable (at least among individuals carrying the same
  allelic type), reordering of leaves does not change (in this view) the tree.
\end{remark}

In order to construct a stochastic process with c\`adl\`ag paths and
state space $\mathbb U_A$, we have to introduce a topology. To this end,
we need to introduce test functions with domain $\mathbb U_A$.

\begin{definition}[Polynomials]\mbox{}\\
  We set $\mathbb R^{\binom{\mathbb N}{2}} \coloneqq  \{\underline{\underline
    r}:=(r_{ij})_{1\leq i<j}: r_{ij}\in\mathbb R\}$. A function $\Phi: \mathbb
  U_A \to \mathbb R$ is a \emph{polynomial},
  if there is a measurable function $\phi: \mathbb R^{\binom{\mathbb
      N}{2}} \times A^{\mathbb N} \to\mathbb R$ depending only on
  finitely many coordinates such that
  \begin{align}\label{eq:Phi}
    \begin{split}
      \Phi\bigl(\overline{(U,r,\mu)}\bigr) & \coloneqq
      \Phi^\phi\bigl(\overline{(U,r,\mu)}\bigr) \\ & \coloneqq \langle
      \mu^{\otimes \N}, \phi\rangle \coloneqq \int \mu^{\otimes \N }(d\underline
      u, d\underline a) \phi\bigl(r(u_i,u_j)_{1\leq i<j}, (a_i)_{i\geq 1}\bigr),
    \end{split}
  \end{align}
  where $\mu^{\otimes \N}$ is the infinite product measure, i.e.\ the
  law of a sequence sampled independently with sampling measure $\mu$.
\end{definition}
Let us remark that  functions of the form \eqref{eq:Phi} are actually
monomials.  However, products and sums of such monomials are again
monomials, and hence we may in fact speak of polynomials; cf.\ the
example below.

\begin{remark}[Interpretation of polynomials]\mbox{}\\
  Assume that $\phi$ only depends on the first $\binom n 2$ coordinates in
  $r(u_i,u_j)_{1\leq i<j}$ and the first $n$ in $(a_i)_{i\geq 1}$. Then, we view
  a function of the form~\eqref{eq:Phi} as taking a sample of size $n$ according
  to $\mu$ from the population, observing the value under $\phi$ of this sample
  and then taking the $\mu$-sample mean over the population.
\end{remark}

\begin{example}[Some functions of the form~\eqref{eq:Phi}]\mbox{}\\
  Some \label{rem:intPhi} functions of the form~\eqref{eq:Phi} will
  appear frequently in this paper, for example $\underline{\underline
    r} \mapsto \phi(\underline{\underline r}):=
  \psi_\lambda^{12}(\underline{\underline r}) \coloneqq e^{-\lambda
    r_{12}}$,
  \begin{align}\label{eq:int}
    \Psi_\lambda^{12}\bigl(\overline{(U,r,\mu)}\bigr) \coloneqq \langle
    \mu^{\otimes \N  }, \psi_\lambda^{12}\rangle = \int
    (\pi_1{}_\ast\mu)^{\otimes 2}(du_1, du_2) e^{-\lambda r(u_1,u_2)}.
  \end{align}
  This function arises from sampling two leaves, $u_1$ and $u_2$, from
  the genealogy $(U,r)$ according to $\pi_1{}_\ast \mu$ and averaging
  over the test function $e^{-\lambda r(u_1,u_2)}$ of this
  sample. Then \ $(\Psi^{12})^2$ is again of the form~\eqref{eq:Phi}
  and
  \begin{align}
    \label{eq:31l}
    (\Psi^{12}_\lambda)^2\bigl(\overline{(U,r,\mu)}\bigr) = \int
    (\pi_1{}_\ast\mu)^{\otimes 4} (du_1,\dots, du_4) e^{-\lambda (r(u_1,u_2) +
      r(u_3,u_4))}.
  \end{align}
  Another function that will be used and which also depends on types is
  given by
  \begin{align}\label{eq:int2}
    \widehat\Psi_\lambda^{12}\bigl(\overline{(U,r,\mu)}\bigr) \coloneqq \int
    \mu^{\otimes 2} (du_1, du_2, da_1, da_2) \ind{a_1=a_2} e^{-\lambda
      r(u_1,u_2)}.
  \end{align}
  In this function $u_1$ and $u_2$ contribute to the integral only if their
  types, $a_1$ and $a_2$ agree.
\end{example}

Since we use polynomials as the domain of the generator for the
tree-valued Fleming--Viot process, we need to restrict this class to smooth
functions.
\begin{definition}[Smooth polynomials]\mbox{}\\
  We denote by
  \begin{align}
    \label{eq:31m}
    \Pi^1 \coloneqq \Bigl\{\Phi^\phi \text{ as in }~\eqref{eq:Phi}:
    \phi \text{ bounded, measurable and for all $\underline a \in
      A^\N$, } \phi(\cdot, \underline a)
    \in\CC_b^1\bigl(\R^{\binom{\mathbb N}{2}}\bigr) \Bigr\}
  \end{align}
  the set of \emph{smooth (in the first coordinate) polynomials}.
  Furthermore we denote by $\Pi^1_n$ the subset of $\Pi^1$ consisting
  of all $\Phi^\phi$ for which $\phi(\underline{\underline r},
  \underline a)$ depends at most on the first $\binom n 2$ coordinates
  of $\underline{\underline r}$ and the first $n$ of $\underline a$
  and hence have degree at most $n$.
\end{definition}

\begin{definition}[Marked Gromov-weak topology]\mbox{}\\
  The \emph{marked Gromov-weak topology} on $\mathbb U_A$ is the
  coarsest topology such that all $\Phi^\phi \in \Pi^1$ with (in
    both variables) continuous $\phi$ are continuous.
\end{definition}

\noindent
The following is from Theorems~2 and~5 in \cite{DGP11}:

\begin{proposition}[Some topological facts about $\mathbb U$]\mbox{}\\%
  The following properties hold:
  \begin{enumerate}
  \item The space $\mathbb U_A$ equipped with the marked Gromov-weak
    topology is Polish.
  \item The set $\Pi^1$ is a convergence determining algebra of
    functions, i.e.\ for random $\mathbb U_A$-valued variables $X,
    X_1, X_2,\dots$ we have
    \begin{equation}\label{ag1}
      X_n\xRightarrow{n\to\infty} X \quad \mbox{ iff } \quad  \mathbf
      E[\Phi(X_n)]
      \xrightarrow{n\to\infty} \mathbf E[\Phi(X)] \; \mbox{ for all
      }\;  \Phi\in\Pi^1.
    \end{equation}
  \end{enumerate}
\end{proposition}

\subsection{Construction of the tree-valued FV-process}
The tree-valued Fleming--Viot process will be defined via a well-posed
martingale problem. Let us briefly recall the concept of a martingale
problem.

\begin{definition}[Martingale problem]\mbox{}\\
  Let \label{D:01} $E$ be a Polish space, $\mathbf P_0 \in\mathcal
  M_1(E)$, $\mathcal F \subseteq {\mathcal B}(E)$ and $\Omega$ a
  linear operator on ${\mathcal B}(E)$ with domain $\mathcal F$. The
  law $\mathbf{P}$ of an $E$-valued stochastic process $\mathcal
  X=(X_t)_{t\geq 0}$ is a \emph{solution of the $(\mathbf
    P_0,\Omega,\mathcal F)$-martingale problem} if $X_0$ has
  distribution $\mathbf P_0$, $\mathcal X$ has paths in the space
  ${\mathcal D}_E([0,\infty))$, almost surely, and for all
  $F\in\mathcal F$,
  \begin{equation}\label{13def}
    \Big(F(X_t)-\int_0^t \Omega F(X_s) ds \Big)_{t\geq 0}
  \end{equation}
  is a $\mathbf{P}$-martingale with respect to the canonical filtration.
  The $(\mathbf P_0,\Omega,\mathcal F)$-martingale problem is said to
  be \emph{well-posed} if there is a unique solution $\mathbf{P}$.
\end{definition}

Let us first specify the \emph{generator} of the tree-valued
Fleming--Viot process. It is a \emph{linear operator with domain
  $\Pi^1$}, given by
\begin{align}\label{dgpGir2}
  \Omega \coloneqq \Omega^{\mathrm{grow}}+\Omega^{\mathrm{res}} +
  \Omega^{\mathrm{mut}} + \Omega^{\mathrm{sel}}.
\end{align}
Here, for $\Phi = \Phi^\phi\in\Pi^1_n$, the linear operators
$\Omega^{\mathrm{grow}}, \Omega^{\mathrm{res}}, \Omega^{\mathrm{mut}},
\Omega^{\mathrm{sel}}$ are defined as follows:
\begin{enumerate}
\item We define the \emph{growth operator} by
  \begin{equation}\label{eq:omega1}
    \Omega^{\mathrm{grow}}\Phi\bigl(\overline{(U,r,\mu)}\bigr)
    \coloneqq \big\langle\mu^{\otimes \N}, \langle \nabla_{\underline{\underline
        r}}\phi,1\rangle \big \rangle,
  \end{equation}
  with
  \begin{align}
    \label{eq:27}
    \langle \nabla_{\underline{\underline r}}\phi, 1\rangle \coloneqq
    \sum_{1 \leq i<j} \frac{\partial \phi}{\partial
      r_{ij}}(\underline{\underline r}, \underline u).
  \end{align}
\item We define the \emph{resampling operator} by
  \begin{align}
    \label{eq:omega2}
    \Omega^{\mathrm{res}}\Phi\bigl(\overline{(U,r,\mu)}\bigr) \coloneqq
    \frac{1}{2}\sum_{k,\ell=1}^n \langle \mu^{\otimes \N}, \phi\circ\theta_{k,\ell}
    - \phi\rangle,
  \end{align}
  with $\theta_{k,\ell}(\underline{\underline r}, \underline a) =
  (\underline{ \underline{\tilde r}}, \underline{\tilde a})$, where
  \begin{equation}
    \label{pp11b}
    \underline{\underline{\tilde r}}_{ij} \coloneqq
    \begin{cases}
      r_{ij}, & \mbox{ if }i,j\neq \ell, \\
      r_{i\wedge k, i\vee k} , & \mbox{ if }j=\ell, \\
      r_{j\wedge k, j\vee k}, & \mbox{ if }i=\ell,
    \end{cases} \quad \text{and} \quad
    \underline{\tilde a}_i \coloneqq
    \begin{cases}
      a_i, & i\neq \ell,\\
      a_k, & i=\ell.
    \end{cases}
  \end{equation}
\item For the \emph{mutation operator}, let $\vartheta\geq 0$ and
  $\beta(\cdot,\cdot)$ be a Markov transition kernel from $A$ to the \emph{Borel
    sets } of $A$ and set
  \begin{align}
    \label{eq:omega3}
    \Omega^{ \mathrm{mut}}\Phi\bigl(\,\overline{(U,r,\mu)}\,\bigr)
    \coloneqq \vartheta \cdot \sum_{k=1}^n \langle \mu^{\otimes \N},
    B_k\phi\rangle,
  \end{align}
  where
  \begin{equation}
    \label{eq:33}
    \begin{aligned}
      B_k\phi & \coloneqq \beta_k\phi - \phi, \\
      (\beta_k\phi)(\underline{\underline r}, \underline a) & \coloneqq \int
      \phi(\underline{\underline r}, \underline a_k^b) \beta(a_k, db),\\
      \underline a_k^b & \coloneqq (a_1,\dots,a_{k-1},b,a_{k+1},\dots).
    \end{aligned}
  \end{equation}
  We say that mutation has a parent-independent component if
  $\beta( \cdot , \cdot)$ is of the form
  \begin{align}
    \label{eq:782}
    \beta(u,dv) = z \bar\beta(dv) + (1-z)\widetilde\beta (u,dv)
  \end{align}
  for some $z\in (0,1]$, $\bar\beta\in\mathcal M_1(A)$ and a
  probability transition kernel $\widetilde\beta$ on $A$.
\item For \emph{selection}, consider the \emph{fitness function}
  \begin{align}\label{dgpGir4}
    \chi': A \times A \times \mathbbm R_+ \to [0,1]
  \end{align}
  with $\chi'(a,b,r) = \chi'(b,a,r)$ for all $a,b\in A, r\in\mathbbm
  R_+$. We require that $\chi'\in\overline{\mathcal
    C}^{0,0,1}(A\times A \times\mathbbm R_+)$, i.e.\ $\chi'$ is
  continuous and continuously differentiable with respect to its
  third coordinate. Then, with $\alpha\geq 0$ (the selection
  intensity) and
  \begin{align}\label{dgpGir5}
    \chi'_{k,\ell} ( \underline{\underline r},
    \underline a) = \chi'(a_k, a_{\ell}, r_{k\wedge
      \ell, k\vee \ell})
  \end{align}
  we set
  \begin{align}\label{eq:pp15}
    \Omega^{ \mathrm{sel}}\Phi\bigl(\overline{(U,r,\mu)}\bigr) & \coloneqq
    \alpha \cdot \sum_{k=1}^n \langle \mu^{\otimes \N}, \phi \cdot \chi'_{k,n+1}
    - \phi\cdot \chi'_{n+1,n+2} \rangle.
  \end{align}

  If $\chi'(a,b,r)$ does not depend on $r$, and if there is $\chi: A \to
  [0,1]$ such that
  \begin{align}\label{eq:208b}
    \chi'(a,b,r)=\chi(a)+\chi(b),
  \end{align}
  we say that selection is additive and conclude that with
  \begin{align}
    \label{eq:pp15c}
    \chi_k(\underline{\underline r}, \underline a) = \chi(a_k),
  \end{align}
  we have
  \begin{align}\label{eq:pp15b}
    \Omega^{ \mathrm{sel}}\Phi\bigl(\overline{(U,r,\mu)}\bigr) &:= \alpha \cdot
    \sum_{k=1}^n \langle \mu^{\otimes \N}, \phi \cdot \chi_{k} -
    \phi\cdot \chi_{n+1} \rangle.
  \end{align}
\end{enumerate}

\begin{remark}[Interpretation of generator terms]\mbox{}\\%
  The growth, resampling, mutation and selection generator terms are interpreted
  as follows:
  \begin{enumerate}
  \item \emph{Growth:} The distance of any pair of individuals is
    given by the time to the most recent common ancestor (MRCA). When
    time passes this distance grows at speed~1. Note that in
    \cite{GPWmp} and \cite{DGP12} the corresponding distance was twice
    the time to MRCA. The reason for this change were some
    simplifications of the terms in the computations that we will see
    later.
  \item \emph{Resampling:} The term $\langle \mu^{\otimes \N},
    \phi\circ\theta_{k,\ell} - \phi\rangle$ describes the action of an
    event where an offspring of individual $k$ replaces individual
    $\ell$ in the sample corresponding to the polynomial $\Phi^\phi$.
    This term is analogous to the measure-valued case; see e.g.\
    (3.21) in \cite{EthierKurtz1993}, but acts on both, the genealogy
    and the types.
  \item \emph{Mutation:} It is important to note that mutation only
    affects types, but not genealogical distances. Hence, the mutation
    operator agrees with the measure-valued case; see e.g.\~(3.16) in
    \cite{EthierKurtz1993}. Note that here we consider only jump
    operators $B$.
  \item \emph{Selection:} This term is best understood when
    considering a finite population. Consider for simplicity the case
    of additive selection (i.e.\ \eqref{eq:208b} holds) in particular
    covering haploid models. Here, the offspring of an individual of
    type $a$ replaces some randomly chosen individual at rate $\alpha
    \chi(a)$ due to selection. In the large population limit, we only
    consider a sample of $n$ individuals and this sample changes only
    if some offspring of an individual outside the sample, e.g.\ the
    $(n+1)$st individual by exchangeability, replaces an individual
    within the sample, the $k$th say, due to selection. After this
    selection event, the fitness of the $k$th individuals is $\chi(a)$
    which is also seen from the generator term. In the case of
    selection acting on diploids, the situation is similar, but one
    has to build diploids from haploids first and then apply the
    fitness function.
  \end{enumerate}
\end{remark}

In \cite{GPWmp,DGP12} the tree-valued Fleming--Viot processes were
constructed via well-posed martingale problems. The following
proposition summarizes Theorems~1, 2 and~4 from \cite{DGP12}.

\begin{proposition}[Tree-valued Fleming--Viot process]\mbox{}\\
  For \label{P:main} $\mathbf P_0\in\mathcal M_1(\mathbb U_{A})$
  the $(\mathbf P_0, \Omega, \Pi^1)$-martingale problem is well-posed.
  Its solution $\CX = (X_t)_{t\geq 0}$ defines a Feller semigroup,
  i.e.\ $X_0\mapsto \mathbf E[f(X_t)|X_0]$ is continuous for all
  $f\in\mathcal C_b(\mathbb U_{A})$, and hence, $\mathcal X$ is a
  strong Markov process.

  Furthermore, the process $\CX$ satisfies the following properties:
  \begin{enumerate}
  \item $\mathbf P(t\mapsto X_t \text{ is continuous}) = 1$.
  \item $\mathbf P(X_t \in \mathbb U_{A,c} \text{ for all }t>0) = 1$.
  \item Let $\Phi = \Phi^\phi\in\Pi^1_n$ such that $\phi$ is symmetric
    under permutations. Then, the quadratic variation of the
    semi-martingale $\Phi(\mathcal X) \coloneqq (\Phi(X_t))_{t\geq 0}$
    is given by
    \begin{align}\label{eq:qv1}
      [\Phi(\mathcal X)]_t & = \int_0^t \big\langle \mu_s, \big(
      \rho_s
      - \langle \mu_s, \rho_s\rangle\big)^2\big\rangle\, ds, \\
      \label{eq:qv2}
      \rho_s(u_1) & \coloneqq \int \mu_s^{\otimes \N} (d(u_2,u_3,\dots ))
      \phi((r_s(u_i, u_j))_{1\leq i<j}).
    \end{align}
  \item Let $\mathbf P_\alpha$ be the distribution of $\mathcal X$
    with selection intensity $\alpha$. Then, for all $\alpha,
    \alpha'\geq 0$, the laws $\mathbf P_{\alpha}$ and $\mathbf
    P_{\alpha'}$ are absolutely continuous with respect to each other.
  \item If either (i) $\alpha=0$ and the process with generator
    $\Omega^{\textnormal{mut}}$ has a unique equilibrium or (ii) $\alpha\geq
    0$ and mutation has a parent-independent component, then the
    process $\mathcal X$ is ergodic. That is, there is an $\mathbb
    U_{A,c}$-valued random variable $X_\infty$, depending on the model
    parameters but not the initial law, such that
    $X_t\xRightarrow{t\to\infty} X_\infty$.
  \end{enumerate}
\end{proposition}

\begin{definition}[Tree-valued Fleming--Viot process and marked Kingman
  measure tree]\mbox{}%
  Using \label{def:treev} the same notation as in
  Proposition~\ref{P:main}, we call the process $\mathcal X$ the
  \emph{tree-valued Fleming--Viot process} and in the case $\alpha=0$
  its ergodic limit
  \begin{equation}\label{ag2}
    X_\infty =  \overline{(U_\infty, r_\infty, \mu_\infty)}
  \end{equation}
  is called \emph{Kingman marked measure tree}.
\end{definition}

\begin{remark}[The Kingman measure tree]\mbox{}\\
  The random variable $X_\infty$ arises from the marked ultrametric measure
  space which is associated with the partition-valued entrance law of the
  Kingman coalescent \cite{GPWmetric09}.
\end{remark}

\begin{example}[The quadratic variation of
  $(\Psi^{12}_\lambda(X_t))_{t\geq 0}$]\mbox{}\\
  In \label{rem:qv} some of the proofs, we will need to compute the quadratic
  variation of $\Phi(\mathcal X) \coloneqq (\Phi(X_t))_{t\geq 0}$ for specific
  $\Phi\in\Pi^1$ via (\ref{eq:qv1}) explicitly. For $\Psi^{12}_\lambda$ as in
  Example~\ref{rem:intPhi}, we have by~\eqref{eq:qv1}
  \begin{equation}\label{eq:qv3}
    \begin{aligned}
      [\Psi^{12}_\lambda(\mathcal X)]_t & = \int_0^t
      \bigl(\Psi^{12,23}_\lambda(X_s) - \Psi^{12,34}_\lambda(X_s)\bigr)\, ds,
    \end{aligned}
  \end{equation}
  with (cf. Definition~\ref{def:PsiPhi})
  \begin{equation}\label{eq:qv4}
    \begin{aligned}
      \Psi^{ij,kl}_\lambda\bigl(\overline{(U,r,\mu)}\bigr) & = \big\langle
      \mu^{\otimes \N} , e^{-\lambda (r(u_i,u_j) + r(u_k,u_l))}\big\rangle, \qquad
      i,j,k,l=1,2,\dots
    \end{aligned}
  \end{equation}
\end{example}

\section{Results}
\label{S:res}
Our main goal is to establish almost sure properties of the paths of
the tree-valued Fleming--Viot process, beyond continuity of paths and
the property that the states are compact marked metric measure space
for every $t>0$, almost surely. We start by studying the geometry of
the marked metric measure tree at time $t$ of the tree-valued
Fleming--Viot process. First we recall in Section~\ref{ss.geoprop}
some well-known facts concerning the geometry of the Kingman
coalescent and then extend them in Section~\ref{sss.pptv} to the
tree-valued Fleming--Viot process. In Section~\ref{ss:respath} we take
advantage of our results and techniques and state some further path
properties of the tree-valued Fleming--Viot process answering two open
questions.

\subsection{Geometric properties of the Kingman coalescent near the
  leaves}
\label{ss.geoprop}
We focus on the Kingman marked measure tree $X_\infty$ introduced in
Proposition~\ref{P:main}.5, but for most assertions in this subsection
we can ignore the marks (i.e.\ think of $A$ consisting of only one
element). Since the introduction of the partition-valued Kingman
coalescent in \cite{Kingman1982a}, this random tree has been studied
extensively for instance in \cite{Ald99} and \cite{Ev00} -- see
  also \cite{MR2534485}. In our present formalism (using metric
measure spaces), $X_\infty$ appeared first in \cite{GPWmetric09}. In
this section, we mostly reformulate known results, but also add a new
one in Proposition~\ref{P.fluc}.

The Kingman measure tree, $X_\infty$, arises from the partition-valued
Kingman coalescent, but can also be realized as a discrete graph tree
using the following construction (see also Figure~\ref{fig:not}). Let
$S_2, S_3,\dots$ be independent exponentially distributed random
variables with parameter~1. Starting with two lines from the root the
tree stays with these two lines for time $S_2/\binom 2 2$. At time
$S_2$ one of the two lines chosen at random splits in two, such that
three lines are present. In general after the jump from $k-1$ to $k$
lines the tree stays with that $k$ lines for a period of time
$S_k/\binom k 2$ and then one of the $k$ lines chosen at random
splits, such that there are $k+1$ lines. The total tree height is thus
$T_1$, where $ T_n \coloneqq
S_{n+1}/\binom{n+1}2+S_{n+2}/\binom{n+2}2+\cdots$, i.e.\ $T_n$ is the
time it takes the coalescent to go from $n$ to infinitely many lines.
The time of the root is called the time of the most recent common
ancestor (MRCA) and $T$ is the present time of the population. In
order to derive the Kingman \emph{marked} metric measure tree,
consider the uniform distribution on the branches and construct a
tree-indexed Markov process, by using a collection of independent
mutation processes as follows. Start with an equilibrium value of the
mutation processes at the root up to the next splitting time where we
continue with two independent mutation processes both starting from
the type in the vertex, etc.\ Running from the root to the leaves and
letting time approach $T$ we finally obtain $X_\infty$.

\begin{figure}
  \centering
  \includegraphics[width=0.85\textwidth]{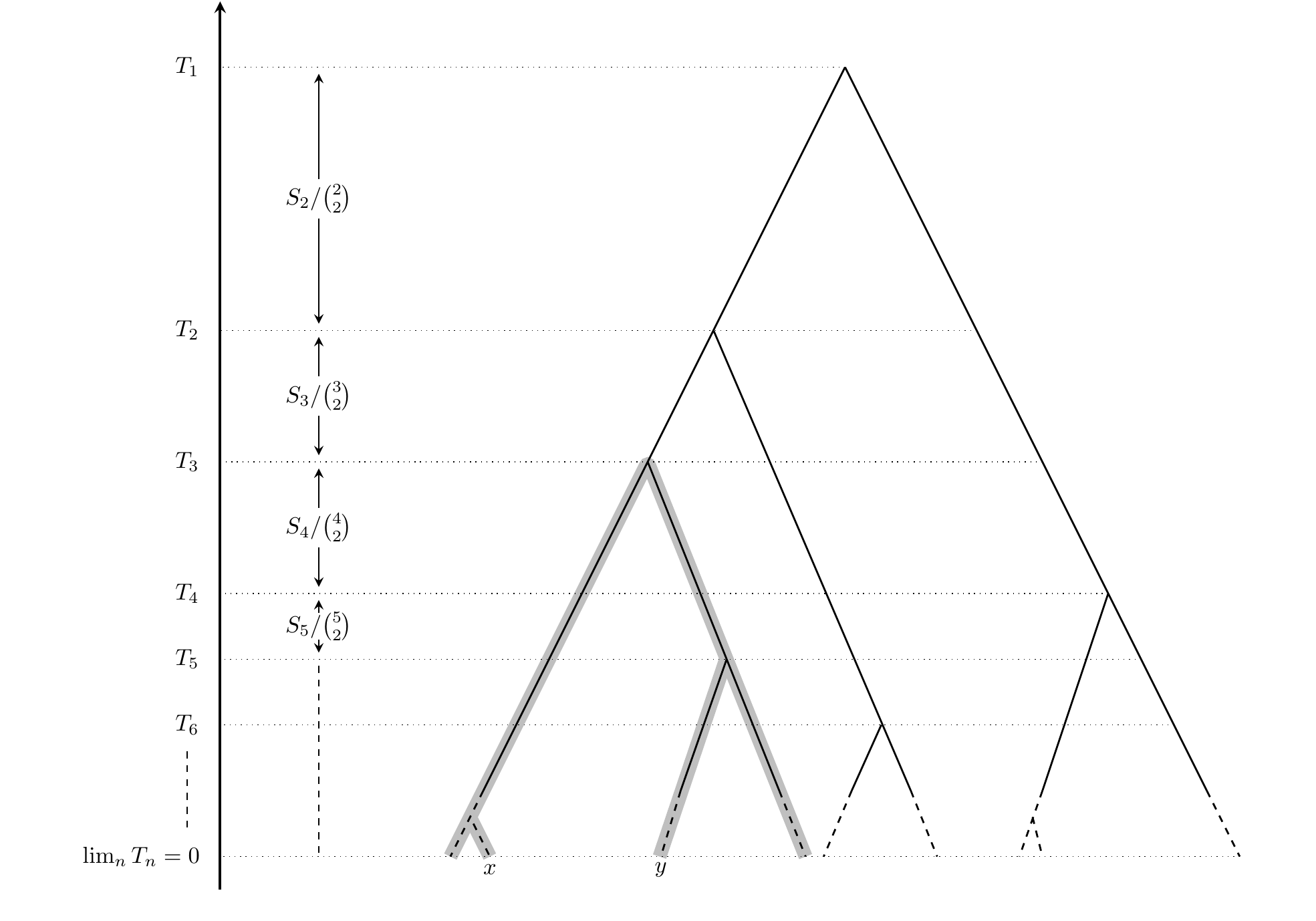}
  \caption{A construction of the Kingman measure tree
    $(X_\infty,r_\infty,\mu_\infty)$ without marks. In the ``dashed
    region'' the tree comes down from infinitely many lines at the
    treetop (time $0$) to six lines at time $T_6$. We have
    $r_\infty(x,y) = T_3$. The thick grey sub-tree is the closed and
    open ball of radius $T_3$ around $x$ and around $y$. The balls
    coincide because $r_\infty$ is an ultrametric.}
  \label{fig:not}
\end{figure}

At time $\varepsilon$ (counted from the top of the tree, for
$\varepsilon<T_1$), a random number $N_\varepsilon$ of lines are
present. Equivalently, $N_\varepsilon$ is the minimal number of
$\varepsilon$-balls needed to cover (the leaves of) the
random tree $X_\infty$. It is a well-known fact using de Finetti's
Theorem that the frequency of the family descending from every of the
$N_\varepsilon$ lines can be defined for all $\varepsilon>0$. In
addition, these frequencies are distributed as the spacings between
$N_\varepsilon$ on $[0,1]$ uniformly distributed random variables
\cite{Pyke1965}.

>From these considerations several results on the geometry of
$X_\infty$ near the leaves can be derived. We briefly recall and
extend some of them and reprove them later in our setting. Roughly we
will show that there are $2/\ve\pm \mathcal
O(1/\sqrt{\varepsilon})$-many families in which the genealogical
distance between the individuals is at most
$\ve$. Furthermore, each of the families has mass of
order $\ve$, as $\ve \to 0$. More precisely, the distribution of (by
$\ve$ rescaled) family sizes is exponential with rate~2.

We split the above picture in two parts. First we study the number of
families and then their size in both cases looking at a LLN and then
at a CLT. We begin with a law of large numbers and a central limit
theorem for $N_\varepsilon$ (see (35) in \cite{Ald99}). Our proofs are
given in Sections~\ref{ss.proofPcover} and~\ref{s.proofPcover2}.

\begin{proposition}[LLN for the number of balls to cover
  $X_\infty$]\mbox{}\\
  Let \label{P:cover} $X_\infty = \overline{(U_\infty, r_\infty,
    \mu_\infty)}$ be the Kingman measure tree. Moreover, let
  $N_\varepsilon$ be the (minimal) number of $\varepsilon$-balls
  needed to cover $(U_\infty, r_\infty)$. Then
  \begin{align}\label{eq:cover1}
    \varepsilon N_\varepsilon \xrightarrow{\varepsilon\to 0}2, \; \text{ almost
      surely.}
  \end{align}
\end{proposition}

\begin{proposition}[CLT for the number of balls to cover
  $X_\infty$]\mbox{}\\
  With \label{P:cover2} the same notation as in
  Proposition~\ref{P:cover} and $Z\sim N(0,1)$,
  \begin{align}\label{eq:Neps3}
    \frac{N_\varepsilon -
      2/\varepsilon}{\sqrt{2/(3\varepsilon)}}\xRightarrow{\varepsilon\to
      0} Z \intertext{and} \label{eq:Neps4}\mathbb
    E\Big[\Big(\frac{N_\varepsilon -
      2/\varepsilon}{\sqrt{2/(3\varepsilon)}}\Big)^2\Big]
    \xrightarrow{\varepsilon\to 0} 1,&& \mathbb
    E\Big[\Big(\frac{N_\varepsilon -
      2/\varepsilon}{\sqrt{2/(3\varepsilon)}}\Big)^4\Big]
    \xrightarrow{\varepsilon\to 0} 1.
  \end{align}
\end{proposition}

\vspace{3ex}

We now come to the \emph{family structure} of $X_\infty =
\overline{(U_\infty,r_\infty,\mu_\infty)}$ close to the leaves. For
$\varepsilon>0$, we define
$B_{\varepsilon}(1),\dots,B_{\varepsilon}(N_\varepsilon) \subseteq
U_\infty$ as the disjoint balls of radius $\varepsilon$ that cover
$U_\infty$ and the corresponding frequencies by
\begin{equation}\label{ag3}
  F_i(\varepsilon) \coloneqq \mu_\infty(B_{\varepsilon}(i)), \;
  i=1,\dots,N_\varepsilon.
\end{equation}
Recall that in an ultrametric space two balls of the same radius are
either equal or disjoint (see also Figure~\ref{fig:not}). Therefore,
the vectors $(F_1(\varepsilon), \dots, F_{N_\varepsilon}(\varepsilon
))$ above are defined in a unique way. It can be viewed as the
frequency vector of a sequence of exchangeable random variables and we
can ask for the law of the empirical distribution of the scaled masses
in the limit $\ve \to 0$, where the underlying sequence, even if
scaled, becomes i.i.d.\ and we should get the scaled law of a single
scaled $F_i$. It turns out, a first step (cf.\ Remark~\ref{R.refine})
is to see that the following law of large numbers holds, the proof of
which (together with the proof of Lemma~\ref{L.reform}) appears in
Section~\ref{S:proofT3}.

\begin{proposition}[Asymptotics of ball masses near the
  leaves]\mbox{}\\
  For \label{P:smallballs} $F_i(\varepsilon)$ as above,
  \begin{align}\label{eq:smallballs}
    \frac{1}{\varepsilon} \sum_{i=1}^{N_\varepsilon}
    F_i(\varepsilon)^2 \xrightarrow{\varepsilon\to 0} 1
  \end{align}
  almost surely.
\end{proposition}

\vspace{2ex}

\sloppy The classical proof of Proposition~\ref{P:smallballs} uses the
fact that the random vector $\bigl(F_1(\varepsilon),\dots,
F_{N_\varepsilon}(\varepsilon)\bigr)$ has the same distribution as the
vector of spacings between $N_\varepsilon$ random variables uniformly
distributed on $[0,1]$. This vector in turn has the same distribution
as $\big(Y_1/(\sum_{i=1}^{N_\varepsilon} Y_i), \dots,
Y_{N_\varepsilon}/(\sum_{i=1}^{N_\varepsilon} Y_i)\big)$, where
$Y_1,Y_2,\dots$ are i.i.d.\ Exp$(1)$ random variables. Then, using a
moment computation,~\eqref{eq:smallballs} can be proved. For details
we refer to Section~2 in \cite{Ev00}. We will use a different route
for which we need the following auxiliary Tauberian result.

\begin{lemma}[Reformulation]\mbox{}\\
  The assertions \label{L.reform}
  \begin{align}\label{eq:Fit}
    & \frac{1}{\ve} \sum_{i=1}^{N_\ve}
    F_i(\ve)^2 \xrightarrow{\ve\to 0}1, \; \text{ a.s. }\\
    \label{eq:convPsi12}
    & (\lambda+1)\Psi^{12}_\lambda(X_\infty)
    \xrightarrow{\lambda\to\infty} 1\; \text{ a.s. }
  \end{align}
  with $\Psi^{12}_\lambda$ from Example~\ref{rem:intPhi} are
  equivalent. Moreover, the equivalence remains true if we replace
  $\varepsilon$ by $\varepsilon_n\downarrow 0$, $\lambda$ by
  $\lambda_n\uparrow \infty$ with $\varepsilon_n \lambda_n = 1$ and
  let $n\to\infty$.
\end{lemma}

\begin{remark}[Refinements of Proposition~\ref{P:smallballs}]\mbox{}\\
  Actually, \cite{Ald99} contains \label{R.refine} refinements of
  Proposition~\ref{P:smallballs}.
  \begin{enumerate}
  \item In equation~(35) of \cite{Ald99} it is claimed that
    (correcting a typo in Aldous' equation)
    \begin{align}\label{eq:refine1}
      \sup_{0\leq x < \infty}
      \Big|\frac{\varepsilon}{2}\sum_{i=1}^{N_\varepsilon} \ind{
        F_i(\varepsilon) < \varepsilon x} -
      (1-e^{-2x})\Big|\xrightarrow{\varepsilon\to 0} 0\; \text{ a.s.}
    \end{align} This means that the Kingman coalescent at distance
    $\ve$ from the tree top consists of approximately $2/\ve$
    families, and the size of a randomly sampled family has an
    exponentially distributed size with expectation $\ve/2$, in
    particular the rescaled empirical measure of the family
    sizes converges weakly to the exponential distribution with mean
    $2$, denoted by $\text{Exp}(1/2)$.

    In order to show this assertion using moments of
    $(F_i(\varepsilon))_{i=1,\dots,N_\varepsilon}$, it is necessary
    and sufficient that for $k=1,2,\dots$
    \begin{equation}\label{ag4}
      \frac{1}{\ve^{k-1}} \suml^{N_\ve}_{i=1} (F_i (\ve))^k \xrightarrow{\ve \to
        0} 2^{-(k-1)} k! \quad \text{a.s.}
    \end{equation}
    The sufficiency follows since the moment problem for the
    exponential distribution is well posed, while for the necessity,
    we assume that~\eqref{eq:refine1} holds, and then one concludes
    (recall the notation $\approx$ from Remark~\ref{not:aux})
    \begin{equation}\label{eq:refine2}
      \begin{aligned}
        \frac{1}{\varepsilon}\sum_{i=1}^{N_\varepsilon}
        F_i(\varepsilon)^2 & = \frac 2\ve \sum_{i=1}^{N_\varepsilon}
        \int_0^{ F_i(\varepsilon)} x\, dx = \frac 2\ve
        \sum_{i=1}^{N_\varepsilon} \int_0^\infty \ind{\ve x \leq
          F_i(\varepsilon)} x\, dx \\ & \stackrel{\varepsilon\to
          0}\approx \int_0^\infty 4x e^{-2x}\, dx = 1
      \end{aligned}
    \end{equation}
    as well as, for $k\geq 2$,
    \begin{equation}\label{eq:refine3}
      \begin{aligned}
        \frac{1}{\varepsilon^{k-1}}\sum_{i=1}^{N_\varepsilon}
        F_i(\varepsilon)^k & = \frac k{\ve^{k-1}}
        \sum_{i=1}^{N_\varepsilon} \int_0^{ F_i(\varepsilon)}
        x^{k-1}\,dx = k\ve \sum_{i=1}^{N_\varepsilon} \int_0^\infty
        \ind{\ve x \leq F_i(\varepsilon)} x^{k-1}\,dx \\ &
        \stackrel{\varepsilon\to 0}\approx \int_0^\infty 2 kx^{k-1}
        e^{-2x}\, dx = 2^{-(k-1)}k!.
      \end{aligned}
    \end{equation}
  \item The statement \eqref{eq:refine1} raises the issue to determine
    the fluctuations in that LLN, i.e.\ to derive a CLT. Here, (36) in
    \cite{Ald99} states that
    \begin{equation}\label{ag6}
      \sqrt{\frac{2}{\varepsilon}} \Big(\frac\varepsilon 2
      \sum_{i=1}^{N_\varepsilon}
      \ind{ F_i(\varepsilon)<\varepsilon x} -
      (1-e^{-2x})\Big)_{x\geq 0} \xRightarrow{\varepsilon\to 0}
      \big(B^0_{1-e^{-2x}} + \tfrac{1}{\sqrt 6}(1-e^{-2x})Z\big)_{x\geq 0},
    \end{equation}
    where $(B^0_t)_{0\leq t\leq 1}$ is a Brownian bridge.  Another,
    for us more suitable formulation is to consider the sum multiplied
    by $N^{-1}_\ve$ instead of $\ve/{2}$, so that $Z$ disappears on
    the right hand side. In this case one would consider the
    fluctuations of the empirical measure of masses of the
    $B(\ve)$-balls that cover the Kingman coalescent
    tree.
  \end{enumerate}
\end{remark}

We have so far investigated the behavior near the treetop looking at
the family sizes with respect to fixed degree $\ve$ of kinship for
$\ve \to 0$. This picture can be refined by obtaining fluctuation
results in~\eqref{eq:smallballs} (or~\eqref{eq:convPsi12}). We obtain
a partial result by considering a degree of kinship $\ve/t$ for $t$
varying in $\R_+$ and letting $\ve \to 0$. This gives a \emph{profile}
of the \emph{family sizes} of varying degrees of kinship and their
correlation structure close to the leaves, if we view the scaling
limit as a function of $t>0$. This profile should be the deterministic
flow of distributions $\{\text{Exp}(t/2): t >0\}$ which are the limits
of
\begin{equation}\label{ag1a}
   \Big(\frac{1}{N_{\ve/t}}\Big(\suml^{ N_{\ve/t}}_{i=1}
    \delta_{\ve^{-1} F_i(\ve /t)}\Big) \Big)_{t\geq 0}
\end{equation}
as $\varepsilon\to 0$. Again we consider the Laplace transform given
through $\Psi^{12}_\lambda$ and obtain the following fluctuation
result -- proved in Section~\ref{ss:proofT3}.

\begin{proposition}[Fluctuations of scaled small masses in small
  balls]\mbox{}\\
  Let \label{P.fluc} $X_\infty$ be the Kingman measure tree. Define
  the process $\mathcal Z^\lambda \coloneqq (Z^\lambda_t)_{t\geq 0}$
  by
  \begin{align}\label{eq:Psi12a}
    Z^\lambda_t \coloneqq \sqrt{\lambda}\big((\lambda t + 1)
    \Psi_{\lambda t}^{12}(X_\infty)-1\big).
  \end{align}
  Then every sequence $(Z^{\lambda_n})_{n\geq 0}$ with
  $\lambda_n\to\infty$ has a convergent subsequence
  $(\lambda_n')_{n\geq 0}$ with
  \begin{align}\label{eq:Psi12b}
    \mathcal Z^{\lambda_n'} \xRightarrow{n \to \infty} \mathcal Z,
  \end{align}
  for some process $\mathcal Z \coloneqq (Z_t)_{t>0}$ with continuous paths.
  Furthermore all limit points satisfy
  \begin{equation}
    \begin{aligned}
      \mathbf E[Z_t] & = 0,\\
      \mathbf {Var}[Z_t] & = \frac 2 t,\\
      \mathbf{Cov}(Z_s, Z_t) & = \frac{4st}{(s+t)^3},\\ \mathbf
      E[Z_t^3] & = 0,\\ \mathbf E[Z_t^4] & = \frac{3}{4t^2}.
    \end{aligned}
  \end{equation}
\end{proposition}

\begin{remark}[Is $\mathcal Z$ Gaussian?]\mbox{}\\
  We conjecture that there is a unique limit process $\mathcal Z$ in
  Proposition~\ref{P.fluc}. Moreover, we note that $\mathbf {Var}[Z_t]$
  and $\mathbf E[Z^4_t]$ are in the relation if $Z_t \sim N(0,
  1/2t)$, which raises the question whether $\mathcal{Z}$ is a
  Gaussian process.
\end{remark}

\subsection{Path properties: the tree-valued Fleming--Viot process
  near the leaves}
\label{sss.pptv}
Although the Kingman measure tree, $X_\infty$, only arises as the
long-time limit of the neutral tree-valued Fleming--Viot process,
$\mathcal X = (X_t)_{t\geq 0}$, near the leaves, $X_t$ (for $t>0$) and
$X_\infty$ have similar geometry. The reason is that the structure
near the leaves of $X_\infty$ or $X_t$ only depends on resampling
events in the (very) recent past. Hence, we expect that the properties
of $X_\infty$ from Propositions~\ref{P:cover} and~\ref{P:smallballs}
hold along the paths of $\mathcal X$. This will be shown in
Theorem~\ref{T1} and Theorem~\ref{T3}, respectively.  Furthermore we
conjecture (but don't have a proof) that the more ambitious
refinements described in Remark~\ref{R.refine} (see~\eqref{ag4}) also
hold along the paths.  In addition, in the stationary regime
$X_0\stackrel d = X_\infty$, Theorems~\ref{T2} and \ref{T4} give two
results on convergence to a Brownian motion along the tree-valued
Fleming--Viot process.

The following theorem is proved in Section~\ref{ss.proofT1}.

\begin{theoremOwn}[Uniform convergence of $\varepsilon N_\varepsilon$
  along paths]\mbox{}\\
  Let \label{T1} $\mathcal X = (X_t)_{t\geq 0}$ with $X_t =
  \overline{(U_t,r_t,\mu_t)}$ be the tree-valued Fleming--Viot process
  (started in some $X_0\in\mathbb U_{A}$) and selection coefficient
  $\alpha\geq 0$. Moreover, let $N_\varepsilon^t$ be the number of
  $\varepsilon$-balls needed to cover $(U_t,r_t)$. Then,
  \begin{align}\label{eq:314}
    \mathbf P\bigl(\lim_{\varepsilon\to 0} \varepsilon N_\varepsilon^t = 2
    \text{ for all }t>0\bigr) = 1.
  \end{align}
\end{theoremOwn}

While the fluctuations in Proposition~\ref{P:cover2} are dealing with
a fixed-time genealogy, we can view the fluctuations of the path
$(\varepsilon N_\varepsilon^t)_{t\geq 0}$ arising in Theorem~\ref{T1}
in the limit $\ve \to 0$.  This program is now carried out along the
tree-valued Fleming--Viot process.

In order to obtain a meaningful limit object, we consider time
integrals. It is important to understand that the part of the time-$t$
tree $X_t$ which is at most $\varepsilon$ apart from the treetop is
independent of $\mathcal F_s \coloneqq \sigma({X_r}: 0\leq r\leq s)$
as long as $s\leq t-\varepsilon$. The following is proved in
Section~\ref{ss.proofT2}.

\begin{theoremOwn}[A Brownian motion in the tree-valued Fleming--Viot
  process]\mbox{}\\
  Let \label{T2} $\mathcal X = (X_t)_{t\geq 0}$ with $X_t =
  \overline{(U_t, r_t, \mu_t)}$ be the neutral tree-valued
  Fleming--Viot process (i.e.\ $\alpha=0$) started in equilibrium,
  $X_0\stackrel d = X_\infty$, and $\mathcal B_\varepsilon =
  (B_\varepsilon(t))_{t\geq 0}$ given by
  \begin{align}
    \label{eq:344b}
    B_\varepsilon(t) := \sqrt \frac 32\int_0^t \Big(N_\varepsilon^s -
    \mathbf E[N_\varepsilon^\infty]\Big) ds.
  \end{align}
  Then,
  \begin{align}
    \label{eq:344c}
    \mathcal B_\varepsilon \xRightarrow{\varepsilon\to 0} \mathcal B,
  \end{align}
  where $\mathcal B = (B_t)_{t\geq 0}$ is a Brownian motion started in
  $B_0=0$.
\end{theoremOwn}

\begin{remark}[Expectation of $N_\varepsilon^\infty$]\mbox{}\\
  In \eqref{eq:344b} one would rather like to replace $\mathbf
  E[N_\ve^\infty]$ by $2/\ve$ to measure the fluctuations around the
  limit profile, i.e.\ to consider $\widetilde{\mathcal B}_\varepsilon
  \coloneqq (\widetilde B_\varepsilon(t))_{t\geq 0}$ defined by
  \begin{align}
    \label{eq:8454}
    \widetilde B_\varepsilon (t):= \sqrt \frac 32\int_0^t
    \Big(N_\varepsilon^s - \frac 2 \varepsilon\Big) ds,
  \end{align}
  instead of $\mathcal B_\varepsilon$. We will see that
  $\widetilde{\mathcal B}_\varepsilon$ converges as $\varepsilon\to 0$
  to a Brownian motion $\widetilde {\mathcal B}$ with \emph{drift},
  but unfortunately we cannot identify the latter. Indeed, from
  Proposition~\ref{P:cover2}, in particular using boundedness of
  second moments, we see that, approximately, $\mathbf
  E[N_\varepsilon^\infty] \approx \frac 2 \varepsilon$ in the sense
  that $\varepsilon \cdot \mathbf E[N_\varepsilon^\infty]
  \xrightarrow{\varepsilon\to 0} 2$.  However, this only implies
  $\mathbf E[N_\varepsilon^\infty]=2/\varepsilon + o(1/\varepsilon)$
  and the error term can be large. In order to sharpen this expansion
  to $\mathbf E[N_\varepsilon^\infty] = 2/\ve + \mathcal O(1)$, we use
  results from~\cite{Tavare1984}. His Section~5.4 (with $\theta=0$ and
  $i=\infty$) yields
  \begin{equation}
    \label{eq:8452}
    \begin{aligned}
      \mathbf E[N_\varepsilon^\infty\cdots (N_\varepsilon^\infty-j+1)]
      & = \sum_{k=1}^\infty \rho_k(\varepsilon)(2k-1)
      \frac{(k-1)\cdots (k-j+1)\cdot k\cdots(k+j-2)}{(j-1)!},
    \end{aligned}
  \end{equation}
  with $\rho_k(\varepsilon) = \exp(-k(k-1)\varepsilon/2)$. From this,
  writing $\delta := \sqrt\varepsilon$ we also see that
  \begin{equation}
    \label{eq:8453}
    \begin{aligned}
      \mathbf E[N_{\varepsilon}^\infty] & =
      \frac{2}{\delta^2}\sum_{x\in\delta\mathbb N}
      \exp(-x(x-\delta)/2) (x-\delta/2)\delta \\ & =
      \frac{2}{\delta^2}\sum_{x\in\delta\mathbb N} \exp(-x^2/2)(1 +
      x\delta/2 + O(\delta^2))(x-\delta/2)\delta \\
      & = \frac{2}{\delta^2} \int_0^\infty xe^{-x^2/2} dx +
      \frac{1}{\delta} \int_0^x (x^2-1)e^{-x^2/2}dx
      + \mathcal O(1) \\
      & = \frac{2}{\varepsilon} + \mathcal O(1)
    \end{aligned}
  \end{equation}
  as $\varepsilon\to 0$. This, together with Theorem~\ref{T2}, implies
  that $\widetilde{\mathcal B}_\varepsilon$ is of the form
  \begin{align}\label{ad900}
    \widetilde B_\varepsilon (t) = B_\varepsilon (t) + \mathcal O(1) t
    \quad \text{ as } \varepsilon \to 0,
  \end{align}
  that is, $\widetilde{\mathcal B}$ is a Brownian motion with drift.
\end{remark}

Now we come to a generalization of Proposition~\ref{P:smallballs} to
the tree-valued Fleming--Viot process. Together with
Lemma~\ref{L.reform}, we obtain the following result on the Laplace
transform of two randomly sampled points. The proof is based on
martingale arguments which will also be useful in the proof of
Theorem~\ref{T6}. Theorem~\ref{T3} is proved in
Section~\ref{ss:proofT3}.

\begin{theoremOwn}[Small ball probabilities]\mbox{}\\
  Let \label{T3} $\mathcal X = (X_t)_{t\geq 0}$ with $X_t =
  \overline{(U_t,r_t,\mu_t)}$ be the tree-valued Fleming--Viot process
  with selection coefficient $\alpha \geq 0$, started in some
  $X_0\in\mathbb U_{A}$, and let $\Psi^{12}_\lambda$ be as in
  Remark~\ref{rem:intPhi}. Then
  \begin{align}\label{eq:small1}
    \lim_{\lambda\to\infty} \mathbf P\Bigl(\sup_{\varepsilon\leq t\leq T}
    |(\lambda +1)\Psi^{12}_\lambda(X_t) - 1| > \varepsilon\Bigr) = 0 \; \text{ for
      all } T<\infty, \ve>0.
  \end{align}

\end{theoremOwn}

\begin{remark}[Convergence in probability versus almost sure
  convergence]\mbox{}\\
  Denote by $ F_1^t(\ve),\dots, F_{N_\ve^t}^t(\ve)$ the sizes of the
  $N_\ve^t$ balls of radius $\ve$ needed to cover $(U_t,r_t)$. If we
  could show that $\mathbf P(\lim_{\lambda\to\infty}(\lambda
  +1)\Psi^{12}_\lambda(X_t)\xrightarrow{\lambda\to\infty}1 \text{ for
    all }t>0)=1$, we could use Lemma~\ref{L.reform} in order to see
  that
  \begin{equation}\label{ag10}
    \mathbf P\Big(\frac{1}{\ve}\sum_{i=1}^{N_\ve^t} (F_i^t(\varepsilon))^2
    \xrightarrow{\ve\to 0}1 \text{ for all }t>0\Big)=1.
  \end{equation}
  However, our proof of Theorem~\ref{T3} is based on a computation
  involving the evolution of fourth moments of $\Psi^{12}_\lambda$ in
  order to show tightness of $\{((\lambda
  +1)\Psi^{12}_\lambda(X_t))_{t\geq \varepsilon}: \lambda>0\}$.  Based
  on these computations, we can only claim convergence in probability
  rather than almost sure convergence.
\end{remark}

\begin{remark}[Possible refinement of Theorem~\ref{T3}]\mbox{}\\
  As an ultimate goal one would want to prove that (compare
  with~\eqref{eq:refine1})
  \begin{align}\label{eq:refine1a}
    \sup_{0\leq t\leq T}\sup_{0\leq x < \infty} \Big|\frac \varepsilon
    2\sum_{i=1}^{N^t_\varepsilon} \ind{ F_i^t(\varepsilon) <
      \varepsilon x} - (1-e^{-2x})\Big|\xrightarrow{\varepsilon\to 0}
    0\; \text{ a.s.}
  \end{align}
  This would mean that the assertion that roughly the tree consists of
  $2/\varepsilon$ families of mean $\varepsilon/2$ exponentially
  distributed sizes holds at all times. Using our conclusions from
  Remark~\ref{R.refine}, this goal can be achieved if we show
  that~\eqref{ag4} holds for $k=1,2,\dots$ uniformly at all
  times. (While the case $k=1$ is trivial, note that a combination of
  Theorem~\ref{T3} and Lemma~\ref{L.reform} gives~\eqref{ag4} for
  $k=2$.) In principle, the technique of our proof of
  Proposition~\ref{P:smallballs} can be extended in order to obtain
  \eqref{ag4} for a given but arbitrary $k$ which would require
  controlling higher order moments of $\Psi^{12}_\lambda$.  If we
  could do this for general $k$ then we would obtain a proof of
  \eqref{eq:refine1}. But since we are using \textsc{Mathematica} for
  these calculations the problem remains open.
\end{remark}

Again, we can formulate a result on fluctuations. Integrating over
time (to get a process rather than white noise) the quantity
$(\lambda+1)\Psi^{12}(X_t)-1$, which appears in Theorem~\ref{T3}, and
using the right scaling, we again obtain a Brownian motion as the weak
limit. The following result is proved in Section~\ref{ss.proofT4}.

\begin{theoremOwn}[Another Brownian motion in the tree-valued
  Fleming--Viot process]\mbox{}\\
  Let \label{T4} $\mathcal X = (X_t)_{t\geq 0}$ with $X_t =
  \overline{(U_t, r_t, \mu_t)}$ be the neutral tree-valued
  Fleming--Viot process (i.e.\ $\alpha=0$) started in equilibrium,
  i.e., $X_0\stackrel d = X_\infty$ and let $\mathcal W_\lambda =
  (W_\lambda(t))_{t\geq 0}$ be given by
  \begin{align}
    \label{eq:344}
    W_\lambda(t) & \coloneqq \lambda \int_0^t
    ((\lambda+1)\Psi^{12}_\lambda(X_s) - 1) ds,
  \end{align}
  with $\Psi^{12}_\lambda$ as in Example~\ref{rem:intPhi}. Then,
  \begin{align}
    \mathcal W_\lambda \xRightarrow{\lambda\to\infty} \mathcal W,
  \end{align}
  where $\mathcal W = (W_t)_{t\geq 0}$ is a Brownian motion started in
  $W_0=0$.
\end{theoremOwn}

\begin{remark}[A heuristic argument]\mbox{}\\
  Assume that $\lambda$ is large. Then,
  $(\lambda+1)\Psi^{12}_\lambda(X_s)-1$ depends approximately only on
  resampling events which happened within an interval
  $[s-C/\lambda,s]$ for some large $C$. In particular, on different
  time intervals (which are at least of order $1/\lambda$ apart), the
  increments of $\mathcal W_\lambda$ are approximately independent.
  Thus, it is reasonable to expect that the limiting process is a
  local martingale. In fact, using some stochastic calculus we can
  show that the limiting process is continuous (i.e.\ the family
  $\{\mathcal W_\lambda: \lambda>0\}$ is tight in the space $\mathcal
  C_{\R }([0,\infty))$) and the limiting object of
  $(W_\lambda^2(t)-t)_{t\geq 0}$ is a local martingale as well. By
  L\'evy's characterization of Brownian motion, $\mathcal W$ must be a
  Brownian motion.
\end{remark}

\subsection{Path properties: non-atomicity and mark functions}
\label{ss:respath}
Using the calculus developed for the statements in
Section~\ref{sss.pptv} we obtain two further properties of the states
of the tree-valued Fleming--Viot process $\mathcal X = (X_t)_{t\geq
  0}$, $X_t = \overline{(U_t,r_t,\mu_t)}$, namely that the states are
{\em atom-free} and admit a {\em mark function}. More precisely,
Theorem~\ref{T5} says that at no time it is possible to sample two
individuals with distribution $\mu_t$ with distance zero; cf.\
Remark~\ref{rem:intT5} below. Furthermore Theorem~\ref{T6} says that
we can assign marks to all individuals in the sense that $\mu_t$ has
the form $\mu_t(du, da) = (\pi_{U_t})_\ast \mu_t(du)
\delta_{\kappa_t(u)}(da)$ for some measurable function $\kappa_t:
U_t\to A$. These two theorems are proved in Section~\ref{S:proofT5}.

\begin{theoremOwn}[$X_t$ never has an atom]\mbox{}\\
  Let \label{T5} $\mathcal X = (X_t)_{t\geq 0}$ with $X_t =
  \overline{(U_t, r_t, \mu_t)}$ be the tree-valued Fleming--Viot
  process. Then,
  \begin{equation}\label{pp1}
    \mathbf P(\mu_t \text{ has no atoms for all }t>0)=1.
  \end{equation}
\end{theoremOwn}

\begin{remark}[Interpretation and idea of the proof] \mbox{}
  \label{rem:intT5}
  \begin{enumerate}
  \item At first glance the fact that $\mu_t$ is non-atomic for all $t
    >0$ might seem to contradict the fact that the measure-valued
    Fleming--Viot diffusion is purely atomic for every $t >0$.
    However, both properties are of different kind and the probability
    measures in question are different objects: $\mu_t$ is a sampling
    measure and the state of the measure-valued Fleming--Viot
    diffusion is a probability measure on the type space. The above
    theorem implies that randomly sampled individuals from the
    tree-valued Fleming--Viot process have distance of order $1$,
    whereas genealogically the atomicity of the measure-valued
    Fleming--Viot diffusion expresses the fact that at every time
    $t>0$ one can cover the state with a finite number of balls with
    radius $t$.
  \item The proof is based on a simple observation: for a measure
    $\mu\in\mathcal M_1(E)$,
    \begin{equation}
      \label{eq:noatom1}
      \begin{aligned}
        \text{$\mu$ has no atom} \quad& \iff\quad \int \mu^{\otimes
          2}(du, dv) \ind{r(u,v)=0} = 0 \quad\\ & \iff \quad
        \lim_{\lambda\to\infty} \int \mu^{\otimes 2}(du, dv)
        e^{-\lambda r(u,v)} = 0.
      \end{aligned}
    \end{equation}
    Hence, the proof of \eqref{pp1} is based on a detailed analysis of
    the Laplace transform of the distance of two points, independently
    sampled with distribution $\mu_t$.
  \end{enumerate}
\end{remark}

\medskip The next goal is to establish that at any time there is a
\emph{mark function}. Briefly, the state $\overline{(U,r,\mu)}$ of a
tree-valued population dynamics admits a mark function $\kappa$ iff
every individual $u\in U$ can be assigned a (unique) type
$\kappa(u)\in A$. This situation occurs in particular in finite
population models, e.g.\ in the Moran model. The question for the
tree-valued Fleming--Viot model is whether types in the finite Moran
model can change at a fast enough scale so that an individual can have
several types in the large population limit. Such a situation can
occur, if the cloud of very close relatives (as measured in the metric
$r$) is not close in location (as measured in the type space $A$).

\begin{definition}[Mark function]\mbox{}\\
  We \label{def:markfct} say that $\overline{(U,r,\mu)} \in \mathbb U_A$ admits
  a mark function if there is a measurable function $\kappa: U\to A$ such that
  for a random pair $({\mathfrak{U}}, {\mathfrak{A}})$ with values in $U \times
  A$ and distribution $\mu$
  \begin{equation}\label{pp12}
    \kappa(\mathfrak{U}) =    {\mathfrak{A}} \quad \mu \text{-almost surely.}
  \end{equation}
  Equivalently, $\overline{(U,r,\mu)} \in \mathbb U_A$ admits a mark
  function if there is  $\kappa: U \to A$ and  $\nu\in\mathcal
  M_1(U)$ with
  \begin{equation}\label{ag.mf2}
    \mu(du, da) = \nu (du) \otimes \delta_{\kappa (u)} (da).
  \end{equation}
  We set
  \begin{align}
    \label{eq:UAmark}
    \mathbb U_A^{\text{mark}} \coloneqq \bigl\{\overline{(U,r,\mu)}\in\mathbb U_A:
    (U,r,\mu) \text{ admits a mark function}\bigr\}.
  \end{align}
\end{definition}

\begin{remark}[mmm-spaces admitting a mark function are
  well-defined]\mbox{}\\
  Let us note that admitting a mark function is a property of an
  equivalence class. Assume $\overline{(U,r,\mu)} =
  \overline{(U',r',\mu')} \in \mathbb U_A$ (with an isometry $\varphi:
  U'\to U$ as in~\eqref{eq:UA2}), where $\overline{(U,r,\mu)}$ admits
  a mark function $\kappa: U\to A$, i.e.\ \eqref{ag.mf2} holds. Then,
  clearly for $\kappa' :=\kappa \circ \varphi$ we have
  \begin{align}
    \label{eq:markequ}
    \mu'(du, da) = (\varphi,\text{id})_\ast \mu(du, da) =
    (\varphi,\text{id})_\ast \nu(du) \otimes \delta_{\kappa(u)}(da) =
    \varphi_\ast \nu(du) \otimes \delta_{\kappa(\varphi(u))}(da).
  \end{align}
  In other words, $(U',r',\mu')$ admits the mark function $\kappa' =
  \kappa\circ\varphi$.
\end{remark}

\begin{theoremOwn}[$X_t$ admits a mark function for all $t$]\mbox{}\\
  Let \label{T6} $\mathcal X = (X_t)_{t\geq 0}$,  $X_t = \overline{(U_t, r_t,
    \mu_t)}$ be the tree-valued Fleming--Viot-dynamics. Then,
  \begin{equation}\label{pp13}
    \mathbf P\bigl( X_t \in \mathbb U_A^{\text{mark}}\text{ for all
    }t>0\bigr)=1.
  \end{equation}
\end{theoremOwn}

\begin{remark}[Mark functions and the lookdown process]\mbox{}\\
  For a series of exchangeable population models it is possible to
  construct the state of an infinite population via the lookdown
  construction \cite{DonnellyKurtz1996,DonnellyKurtz1999}. This
  construction immediately allows us to define a mark function on a
  countable number of individuals specifying their types at all times,
  which suggests that \eqref{pp13} should hold. However, the metric
  space read off from the lookdown process is not complete, and the
  mark function is not continuous. It seems possible to extend the
  definition of the mark function to the completion of the
  corresponding metric space by defining a (right-continuous)
  mark-function on the tree from root to the leaves. However, we do
  not pursue this direction here. Instead, our proof of
  Theorem~\ref{T6} in Section~\ref{S:proofT6} uses again martingale
  arguments and moment computations.
\end{remark}
\noindent

\subsection{Strategy of proofs}
\label{ss.stratproof} The proofs of our results are of two types. On
the one hand, the proofs of Propositions~\ref{P:cover} and
\ref{P:cover2}, Theorems~\ref{T1} and~\ref{T2} use as the basic tools
the fine properties of coalescent times in Kingman's coalescent. This
means they are carried out without specific martingale properties of
the tree-valued Fleming--Viot process. On the other hand,
Propositions~\ref{P:smallballs} and~\ref{P.fluc},
Theorems~\ref{T3},~\ref{T4},~\ref{T5} and~\ref{T6} are proved by
calculating expectations (moments) of polynomials, which is possible
by using the martingale problem for the tree-valued Fleming--Viot
process. The polynomials we have to consider here (see also
Remark~\ref{rem:intPhi}) are either $\Psi^{12}_\lambda$ or
$\widehat\Psi^{12}_\lambda$, i.e.\ polynomials based on the test
functions $\varphi(\underline{\underline r},\underline a) = \exp
(-\lambda r_{12})$ or $\varphi(\underline{\underline r},\underline a)
= \exp (-\lambda r_{12})\ind{a_1=a_2}$ and products, powers and linear
combinations thereof. For the calculations of the moments of this type
we develop some methodology which we explain in Section~\ref{S:prep}.

Propositions~\ref{P:cover} and~\ref{P:cover2}, Theorems~\ref{T1}
and~\ref{T2} are proved in Section~\ref{s.proofPcover} while
Propositions~\ref{P:smallballs} and~\ref{P.fluc},
Theorems~\ref{T3} and~\ref{T4} are proved in
Section~\ref{S:proofT3}. The latter
results are then used to prove Theorems~\ref{T5} and~\ref{T6}
in Section~\ref{S:proofT5}.

\section{Proof of Propositions~\ref{P:cover} and~\ref{P:cover2} and of
  Theorems~\ref{T1} and~\ref{T2}}
\label{s.proofPcover}

\subsection{Preparation: times in the Kingman coalescent}
\label{ss.prepKingco}
Recall from Section~\ref{ss.geoprop} that $T_n = S_{n+1}/\binom{n+1} 2
+ S_{n+2}/\binom{n+2} 2 + \cdots$ is the time the Kingman coalescent
needs to go down to $n$ lines, where $S_2, S_3,\dots $ are i.i.d.\
exponential random variables with rate~1. Before we begin, we prove
some simple results on the times $T_n$.

\begin{lemma}[Moments and exponential moments of $T_n$]\mbox{}\\
  Let \label{l:momentsTn} $T_n$ be the time the Kingman coalescent
  needs to go from infinitely many to $n$ lines. Then,
  \begin{equation}
    \label{eq:785}
    \begin{aligned}
      \mathbf E[T_n] & = \frac{2}{n},\\
      \mathbf E\bigl[(T_n-2/n)^2\bigr] & = \frac{4}{3n^3}(1 + \mathcal O(1/n)),\\
      \mathbf E\bigl[(T_n-2/n)^3\bigr] & = \frac{16}{5n^5}(1 + \mathcal O(1/n)),\\
      \mathbf E\bigl[(T_n-2/n)^4\bigr] & = \frac{16}{9n^6}(1 + \mathcal O(1/n)),\\
      \mathbf E\bigl[(T_n-2/n)^6\bigr] & = \frac{64}{27n^9}(1 + \mathcal O(1/n)),\\
      \mathbf E\bigl[(T_n-2/n)^8\bigr] & = \frac{4^4}{3^4n^{12}}(1 +
      \mathcal
      O(1/n)),\\
      \mathbf E[e^{-\lambda T_n}] & \lesssim e^{-\tfrac 43 (\tfrac
        \lambda n \wedge \tfrac{\sqrt{\lambda}}2)}, \quad \lambda\geq
      0.
    \end{aligned}
  \end{equation}
\end{lemma}

\begin{proof}
  Recall that $\mathbf E[(S_i-1)^k] = k!\sum_{i=0}^{k} (-1)^i/i!$. We
  start by writing
  \begin{equation}
    \label{eq:567}
    \begin{aligned}
      \mathbf E[T_n] & = \sum_{i=n+1}^\infty \frac{\mathbf
        E[S_i]}{\binom i 2} = \sum_{i=n+1}^\infty \frac{2}{i(i-1)} =
      2\sum_{i=n+1}^\infty \frac{1}{i-1} - \frac{1}{i} = \frac 2 n.
    \end{aligned}
  \end{equation}
  Next,
  \begin{equation}
    \label{eq:568}
    \begin{aligned}
      \mathbf E[(T_n-2/n)^2] & = \mathbf{Var}[T_n] = \sum_{i=n+1}^\infty
      \frac{4\mathbf {Var}[S_i]}{i^2(i-1)^2} = 4\sum_{i=n+1}^\infty
      \frac{1}{i^2(i-1)^2} \\ & = 4 \int_n^\infty \frac{1}{x^4} dx +
      \mathcal O(1/n^4) = \frac{4}{3n^3}(1+\mathcal O(1/n)).
    \end{aligned}
  \end{equation}
  For third moments, using $\mathbf E[(S_i-1)^3] =2$
  \begin{equation}
    \label{eq:569}
    \begin{aligned}
      \mathbf E[(T_n-2/n)^3] & = \mathbf
      E\Big[\Big(\sum_{i=n+1}^\infty \frac{2}{i(i-1)}(S_i -
      1)\Big)^3\Big] \\ & = \sum_{i=n+1}^\infty
      \frac{2^3}{i^3(i-1)^3}\mathbf E[(S_i-1)^3] \\ &  =
      \frac{16}{5n^5}(1+\mathcal O(1/n)),\\
    \end{aligned}
  \end{equation}
  For fourth moments,
  \begin{equation}
    \label{eq:570}
    \begin{aligned}
      \mathbf E[(T_n-2/n)^4] & = \mathbf
      E\Big[\Big(\sum_{i=n+1}^\infty \frac{2}{i(i-1)}(S_i -
      1)\Big)^4\Big] \\ & = \sum_{i=n+1}^\infty
      \frac{2^4}{i^4(i-1)^4}\mathbf E[(S_i-1)^4] \\ & \qquad +
      \sum_{i,j=n+1 \atop i\neq j}^\infty \frac{4}{i^2(i-1)^2}
      \frac{4}{j^2(j-1)^2} \mathbf E[(S_i-1)^2]\cdot \mathbf
      E[(S_j-1)^2] \\ & = \Big(\sum_{i=n+1}^\infty
      \frac{4}{i^2(i-1)^2}\mathbf E[(S_i-1)^2]\Big)^2 + \mathcal
      O(1/n^7) \\ & = \frac{16}{9n^6}(1+\mathcal O(1/n))
    \end{aligned}
  \end{equation}
  For sixth moments,
  \begin{equation}
    \label{eq:1678}
    \begin{aligned}
      \mathbf E[(T_n-2/n)^6] & = \mathbf
      E\Big[\Big(\sum_{i=n+1}^\infty \frac{2}{i(i-1)}(S_i -
      1)\Big)^6\Big] \\ & = \Big(\sum_{i=n+1}^\infty
      \frac{4}{i^2(i-1)^2}\mathbf E[(S_i-1)^2]\Big)^3 + \mathcal
      O(1/n^{10}) \\ & = \frac{64}{27n^9}(1+\mathcal O(1/n)).
    \end{aligned}
  \end{equation}
  With analogous calculations, the results for the 8th moment follows.

  Finally for the exponential moments, we compute for any $\lambda\geq 0$
  \begin{equation}
    \label{eq:1679}
    \begin{aligned}
      \mathbf E[e^{-\lambda T_n}] & = \prod_{i=n+1}^\infty
      \frac{\binom i 2}{\binom i 2 + \lambda} =
      \exp\Big(\sum_{i=n+1}^\infty \log\Big(1 - \frac{\lambda}{\binom
        i 2 + \lambda}\Big)\Big) \\ & \leq \exp\Big(
      -\sum_{i=n+1}^\infty \frac{\lambda}{\binom i 2 + \lambda}\Big)
      \leq \exp\Big( -\sum_{i = (n+1)\vee \lceil \sqrt{4\lambda} +1
        \rceil}^\infty \frac{\lambda}{\binom i 2 + \lambda}\Big) \\ &
      \leq \exp\Big( -\sum_{i = (n+1)\vee \lceil \sqrt{4\lambda} +
        1\rceil}^\infty \frac{\tfrac 23 \lambda}{\binom i 2}\Big) =
      \exp\Big(-\frac 43 \cdot\frac\lambda{n\vee \big\lceil
        \sqrt{4\lambda}\big\rceil }\Big) \lesssim e^{-\tfrac 43
        (\tfrac \lambda n \wedge \tfrac{\sqrt{\lambda}}2)}.
    \end{aligned}
  \end{equation}
\end{proof}

\subsection{Proof of Proposition~\ref{P:cover}}
\label{ss.proofPcover}
Let $T_n$ be as in the last subsection and recall $N_\varepsilon$ from
Proposition~\ref{P:cover}. Then,~\eqref{eq:cover1} is equivalent to
\begin{align}
  \label{eq:l12}
  n T_n\xrightarrow{n\to\infty} 2.
\end{align}
In order to see this, note that $N_{T_n}=n$ by definition of $T_n$ and
\begin{align}
  \label{eq:equi1}
  \{T_n\geq \varepsilon\} = \{N_\varepsilon\geq n\}.
\end{align}
Since $T_n\downarrow 0$ as $n\to\infty$ (and $N_\ve \uparrow \infty$
as $\varepsilon\to 0$), the equivalence of~\eqref{eq:cover1}
and~\eqref{eq:l12} follows.

For \eqref{eq:l12}, it suffices to note that
\begin{equation}
  \begin{aligned}
    \mathbf P(|nT_n-2|>\varepsilon) & \leq \frac{\mathbf E[(T_n -
      2/n)^4]}{(\varepsilon/n)^4} \leq \frac{16}{9\varepsilon^4
      n^2}(1+\mathcal O(1/n)).
  \end{aligned}
\end{equation}
by Lemma~\ref{l:momentsTn}. Since the right hand side is summable,
$\limsup_{n\to\infty}|nT_n-2|\leq \varepsilon$ almost surely for all
$\varepsilon>0$. In other words, $nT_n\to 2$ almost surely.

\subsection{Proof of Proposition~\ref{P:cover2}}
\label{s.proofPcover2}
By the Lindeberg--Feller central limit theorem, we see from the moment
computations of Lemma~\ref{l:momentsTn} that
\begin{align}\label{eq:Neps1}
  \frac{T_n - 2/n}{\sqrt{4/(3n^3)}} \xRightarrow{n\to\infty} Z.
\end{align}
Recalling~\eqref{eq:equi1}, we set
\begin{align}\label{eq:1781}
  a_\varepsilon(x) := \Big\lfloor \frac{2}{\varepsilon} +
  x\sqrt{2/(3\varepsilon)}\Big\rfloor = \frac{2}{\varepsilon} +
  x\sqrt{2/(3\varepsilon)} + \mathcal O(1),
\end{align}
such that for every $x\in\mathbb R$:
\begin{equation}
  \label{eq:1682}
  \begin{aligned}
    \mathbf P\Big(\frac{N_\varepsilon -
      2/\varepsilon}{\sqrt{2/(3\varepsilon)}}> x\Big) & = \mathbf
    P(T_{a_\varepsilon(x)}>\varepsilon) \\ & = \mathbf
    P\Big(\frac{T_{a_\varepsilon(x)}-2/a_\varepsilon(x)}{\sqrt{4/(3a_\varepsilon(x)^3)}}
    > \underbrace{\frac{\varepsilon
        -2/a_\varepsilon(x)}{\sqrt{4/(3a_\varepsilon(x)^3)}}}_{=
      x(1+\mathcal O(\sqrt\varepsilon))}\Big) \\ & \xrightarrow{\ve
      \to 0} \mathbf P(Z>x),
  \end{aligned}
\end{equation}
which implies~\eqref{eq:Neps3}.

However, we also need to show convergence of  moments up to fourth order. We
write
\begin{align}\label{eq:int}
  \mathbf E\Big[\Big(\frac{N_\varepsilon -
    2/\varepsilon}{\sqrt{2/(3\varepsilon)}}\Big)^2\Big] & =
  2\int_0^\infty x\mathbf P\Big(\Big|\frac{N_\varepsilon -
    2/\varepsilon}{\sqrt{2/(3\varepsilon)}}\Big|>x\Big) dx,\\
  \label{eq:int10}
  \mathbf E\Big[\Big(\frac{N_\varepsilon -
    2/\varepsilon}{\sqrt{2/(3\varepsilon)}}\Big)^4\Big] & =
  4\int_0^\infty x^3\mathbf P\Big(\Big|\frac{N_\varepsilon -
    2/\varepsilon}{\sqrt{2/(3\varepsilon)}}\Big|>x\Big) dx.
\end{align}
To estimate the right hand side of \eqref{eq:int10} we first show that for a
suitably chosen $\delta >0$ and $\sqrt{6-\delta} \le c<\sqrt{6}$ we have
\begin{align}\label{eq:921}
  \int_{c/\sqrt\varepsilon}^\infty x^3  \mathbb P\Big(\frac{N_\varepsilon -
    2/\varepsilon}{\sqrt{2/(3\varepsilon)}}<-x\Big) =
  \int_{c/\sqrt\varepsilon}^{\sqrt{6/\ve}} x^3  \mathbb
  P\Big(\frac{N_\varepsilon - 2/\varepsilon}{\sqrt{2/(3\varepsilon)}}<-x\Big)
  \xrightarrow{\varepsilon\to 0} 0,
\end{align}
where the equality follows because for  $x \geq \sqrt{6/\ve}$ the integrand is
identically $0$. Using the exponential Chebyshev inequality, we obtain for all
$\lambda_{y,\varepsilon}\geq 0$
\begin{align}
  \label{eq:1683}
  \begin{split}
    \int_{c/\sqrt\varepsilon}^{\sqrt{6/\ve}} x^3 & \mathbb
    P\biggl(\frac{N_\varepsilon -
      2/\varepsilon}{\sqrt{2/(3\varepsilon)}}<-x\biggr)\, dx =
    \int_{c/\sqrt\varepsilon}^{\sqrt{6/\varepsilon}} x^3 \mathbb
    P\Big(N_\varepsilon < \frac 2\varepsilon - x
    \sqrt{2/(3\varepsilon)}\Big) \,dx \\ & = \frac 94
    \int_{c\sqrt{2/3}}^{2} \frac{y^3}{\varepsilon^2} \mathbb
    P\Big(N_\varepsilon \leq \lfloor \frac{2-y}{\varepsilon}
    \rfloor\Big)\, dy = \frac 94 \int_{c\sqrt{2/3}}^{2}
    \frac{y^3}{\varepsilon^2} \mathbb P\Big(T_{\lfloor
      \frac{2-y}{\varepsilon}\rfloor}\leq \varepsilon\Big)\, dy \\ &
    \leq \frac 94 \int_{c\sqrt{2/3}}^{2} \frac{y^3}{\varepsilon^2}
    e^{\lambda_{y,\varepsilon} \varepsilon}\mathbb
    E\Big[e^{-\lambda_{y,\varepsilon}T_{\lfloor
        \frac{2-y}{\varepsilon} \rfloor}}\Big]\,dy, \\ \intertext{now
      taking the lower bound for $c$, setting $\delta'=2\delta/3$ and
      using \eqref{eq:1679} we get} & \leq \frac 94
    \int_{\sqrt{4-\delta'}}^{2} \frac{y^3}{\varepsilon^2}
    e^{\lambda_{y,\varepsilon} \varepsilon}\mathbb
    E\Big[e^{-\lambda_{y,\varepsilon}T_{\lfloor
        \frac{2-y}{\varepsilon} \rfloor}}\Big]\, dy \\ & \lesssim \frac
    94 \int_{\sqrt{4-\delta'}}^{2}
    \frac{y^3}{\varepsilon^2}\exp\Big\{\lambda_{y,\varepsilon}
    \varepsilon - \frac43 \Big(\frac{\lambda_{y,\varepsilon}
      \varepsilon}{2-y} \wedge
    \frac{\sqrt{\lambda_{y,\varepsilon}}}{2}\Big) \, \Big\} \,dy.
  \end{split}
\end{align}
Now choose $\lambda_{y,\varepsilon} =
\frac{(2-y+1/2)^2}{\varepsilon^2}$ and let $\delta'= 0.39$ be the
solution of $\sqrt{4-\delta'}=1.9$. For $y \in (1.9,2]$ we have
\begin{align}
  \label{eq:ad1}
  \frac{\lambda_{y,\varepsilon} \varepsilon}{2-y} \wedge
  \frac{\sqrt{\lambda_{y,\varepsilon}}}{2} =
  \frac{(2-y+1/2)^2}{\varepsilon(2-y)} \wedge
  \frac{2-y+1/2}{2\varepsilon} = \frac{2-y+1/2}{2\varepsilon}.
\end{align}
Thus,
\begin{align}
  \label{eq:ad2}
  \lambda_{y,\varepsilon} \varepsilon - \frac43
  \Big(\frac{\lambda_{y,\varepsilon} \varepsilon}{2-y} \wedge
  \frac{\sqrt{\lambda_{y,\varepsilon}}}{2}\Big) = -
  \frac1{\varepsilon}\Big(-\frac{55}{12} + \frac{13}{3}y-y^2 \Big).
\end{align}
It is easy to see that on the interval $[1.9,2]$ the function $y
\mapsto (-\frac{55}{12} + \frac{13}{3}y-y^2)$ is bounded below by
$a=0.04$ (its value in $1.9$). It follows
\begin{equation}
  \label{eq:1684}
  \begin{aligned}
    \int_{c/\sqrt\varepsilon}^\infty x^3 & \mathbb
    P\Big(\frac{N_\varepsilon -
      2/\varepsilon}{\sqrt{2/(3\varepsilon)}}<-x\Big) \, dx \leq c'
    \frac 1{\varepsilon^2} e^{-a/\varepsilon}
    \xrightarrow{\varepsilon\to 0} 0
  \end{aligned}
\end{equation}
for a suitable constant $c'>0$, and hence we have shown~\eqref{eq:921}.

In order to show convergence of fourth (and second) moments,
using~\eqref{eq:int10}, since $\mathbf P\Big(\Big|\frac{N_\varepsilon
  - 2/\varepsilon}{\sqrt{2/(3\varepsilon)}}\Big|>x\Big)
\xrightarrow{\varepsilon\to 0} \mathbf P(|Z|>x)$ pointwise, we need to
show that there is an integrable function dominating
\begin{align}
  \label{eq:Nepsproof2}
  x^3 \left( \mathbf P\left(\frac{N_\varepsilon -
        2/\varepsilon}{\sqrt{2/(3\varepsilon)}}>x\right) + \mathbf
    P\left(\frac{N_\varepsilon -
        2/\varepsilon}{\sqrt{2/(3\varepsilon)}}<-x\right)\ind{x\leq
      c/\sqrt{\varepsilon}}\right),
\end{align}
for some $c>0$ and $x\geq 0$. For this, we get, by the Markov
inequality and~\eqref{eq:1678}
\begin{equation}
  \label{eq:1684}
  \begin{aligned}
    \mathbf P\biggl(\Big| & \frac{N_\varepsilon -
      2/\varepsilon}{\sqrt{2/(3\varepsilon)}}\Big|>x \biggr)\\ & = \mathbf
    P\Big(T_{a_\varepsilon(x)} - 2/a_\varepsilon(x) > \varepsilon -
    2/a_\varepsilon(x)\Big) + \mathbf P\Big(T_{a_\varepsilon(-x)} -
    2/a_\varepsilon(-x) \leq \varepsilon - 2/a_\varepsilon(-x)\Big) \\ & \leq
    \frac{\mathbf E[(T_{a_\varepsilon(x)} - 2/a_\varepsilon(x))^6]}{(\varepsilon
      - 2/a_\varepsilon(x))^6} + \frac{\mathbf E[(T_{a_\varepsilon(-x)} -
      2/a_\varepsilon(-x))^6]}{(\varepsilon - 2/a_\varepsilon(-x))^6} \\
    & = \frac{64(1+\mathcal O(1/a_\varepsilon(x))}{27
      a_\varepsilon(x)^9(x\sqrt{\varepsilon^3/6}(1+\mathcal
      O(\sqrt{\varepsilon})))^6} + \frac{64(1+\mathcal O(1/a_\varepsilon(x))}{27
      a_\varepsilon(-x)^6(-x\sqrt{\varepsilon^3/6}(1+\mathcal
      O(\sqrt{\varepsilon})))^6} \\ & = \frac{2(1+\mathcal
      O(1/a_\varepsilon(x))}{x^6(1+\mathcal O(\sqrt\varepsilon))} \leq
    \frac{2}{x^6(1+\mathcal O(\varepsilon))} + \mathcal
    O\Big(\frac{\varepsilon}{x^6(1+x\sqrt{\varepsilon/6})}\Big),
  \end{aligned}
\end{equation}
since
\begin{align}\label{eq:1686}
  \frac{1}{a_\varepsilon(x)} = \frac{1}{\frac{2}{\varepsilon} +
    x\sqrt{2/(3\varepsilon)} + \mathcal O(1)} = \frac{\varepsilon}{2}
  \frac{1}{1+x\sqrt{\varepsilon/6} + \mathcal O(\varepsilon)}.
\end{align}
Since the area $x<0$ is restricted to $x\leq c/\sqrt\varepsilon$
in~\eqref{eq:Nepsproof2}, the $\mathcal O(\cdot)$-term on the right hand
side of~\eqref{eq:1684} does not have a pole. It is now easy to obtain
an integrable function dominating~\eqref{eq:Nepsproof2}, leading
to~\eqref{eq:Neps4}.

\subsection{Proof of  Theorem~\ref{T1}}
\label{ss.proofT1}
By Proposition~\ref{P:main}.5.\ (see Theorem 2 of \cite{DGP12} for
details) the tree-valued Fleming--Viot process with selection has a
law which is absolutely continuous with respect to the neutral
process. Therefore it suffices to consider the neutral case,
$\alpha=0$. We observe that for $\alpha=0$ the claim is not affected
by mutation. Moreover, it suffices to deal with the case $X_0
\stackrel d = X_\infty$. The reason is that~\eqref{eq:314} is
equivalent to the assertion that for all $\delta>0$ and uniformly for
all $t>\delta$ we have $\lim_{\varepsilon\to 0} \varepsilon
N_\varepsilon^t = 2$ almost surely. Then one can use the independence
of $N_\varepsilon^t$ and $X_0$ for $\varepsilon<t$.

Let
\begin{align}
  T_n^t \coloneqq \inf\{\ve>0: (U_t, r_t) \text{ can be covered by $n$
    balls of radius $\ve$}\},
\end{align}
i.e.\ $T_n^t$ is the minimal time we have to go back from time $t$
such that we have $n$ ancestral lineages. It suffices to show (see
around \eqref{eq:l12}) that $\mathbf P(nT_n^t \xrightarrow{n\to\infty}
2 \text{ for all }t>0)=1$. To prove this, we need to extend the proof
of Proposition~\ref{P:cover}. It suffices to show that $nT_n^t\to 2$
uniformly for $0\leq t\leq 1$ (if $X_0\stackrel d = X_\infty$).

First,  by Lemma~\ref{l:momentsTn}, for all $t\geq 0$ we have
\begin{align}
  \mathbf E\bigl[\bigl(T^t_n - \mathbf E[T^t_n]\bigr)^8\bigr] \lesssim
  \frac{1}{n^{12}}.
\end{align}
Therefore considering the process along a discrete grid, we have for
any $\ve>0$,
\begin{equation}
  \begin{aligned}
    \mathbf P\biggl(\sup_{k=0,\dots,n^2} |nT_n^{k/n^2}-2| > \ve\biggr)
    & \leq \sum_{k=0}^{n^2} \mathbf P\bigl(|nT_n^{k/n^2} -
    2|>\ve\bigr) \\ & \le n^2 \frac{\mathbf E\bigl[\bigl(T_n^0 -
      \mathbf E[T_n^0]\bigr)^8\bigr]}{(\ve/n)^8} \lesssim
    \frac{1}{\ve^8 n^2}.
  \end{aligned}
\end{equation}
Since the right hand side is summable, $\sup_{k=1,\dots,n^2}
|nT_n^{k/n^2}-2| \xrightarrow{n\to\infty} 0$ almost surely by the
Borel--Cantelli lemma.  Now we need this in continuous time and derive
for the difference $(nT_n^{t/n^2}-2)$ bounds from above and from
below.

First observe that $\{T_n^t > \varepsilon\} \subseteq \{T_n^s >
\varepsilon - (t-s)\}$ for $t-\varepsilon \leq s \leq t$ for
tree-valued Fleming--Viot process. (This property holds in every
population model, arising as a diffusion limit from an individual
based population where we can define ancestors, since the ancestors at
time $t-\varepsilon$ of the population at time $t$ must then be
ancestors at time $t-\varepsilon = s - \varepsilon + (t-s)$ of the
population at time $s$, for $t-\varepsilon<s<t$.) We can now write,
for $\varepsilon>0$ and $n>2/\varepsilon$:
\begin{equation}
  \begin{aligned}
    \mathbf P\left(\sup_{0\leq t\leq 1} n T_n^t > 2+\varepsilon
    \right) & \leq \mathbf P\Big(\sup_{0\leq t\leq 1} T_n^{\lfloor
      tn^2\rfloor/n^2} > \frac{2+\varepsilon}{n} - \Big(t -
    \frac{\lfloor tn^2\rfloor}{n^2}\Big)\Big) \\ & \leq \mathbf
    P\Big(\sup_{k=0,\dots,n^2} T_n^{k/n^2} > \frac{2+\varepsilon}{n} -
    \frac{1}{n^2}\Big) \\ & \leq \mathbf P\Big(\sup_{k=0,\dots,n^2}
    T_n^{k/n^2} > \frac{2+\varepsilon/2}{n}\Big) \lesssim
    \frac{1}{\varepsilon^8 n^2}.
  \end{aligned}
\end{equation}
Hence $\limsup_{n\to\infty} \sup_{0\leq t\leq 1} n T_n^t \leq 2$
almost surely, by the Borel--Cantelli lemma.

For the other direction of the inequality we use $\{T_n^t < \varepsilon\}
\subseteq \{T_n^s < \varepsilon + s-t\}$ for $s\geq t$, to get
\begin{equation}
  \begin{aligned}
    \mathbf P\left(\sup_{0\leq t\leq 1} n T_n^t < 2-\varepsilon
    \right) & \leq \mathbf P\Big(\sup_{0\leq t\leq 1} T_n^{\lceil
      tn^2\rceil/n^2} <
    \frac{2-\varepsilon}{n} + \frac{\lceil tn^2\rceil}{n^2}-t\Big) \\
    & \leq \mathbf P\Big(\sup_{k=0,\dots,n^2} T_n^{k/n^2} <
    \frac{2-\varepsilon}{n} + \frac{1}{n^2}\Big) \\ & \leq \mathbf
    P\Big(\sup_{k=0,\dots,n^2} T_n^{k/n^2} <
    \frac{2-\varepsilon/2}{n}\Big) \lesssim \frac{1}{\varepsilon^8
      n^2}.
  \end{aligned}
\end{equation}
Hence $\liminf_{n\to\infty} \sup_{0\leq t\leq 1} n T_n^t \geq 2$
almost surely.

Combining both, the estimate from above and from below we obtain the assertion
of Theorem~\ref{T1}.

\subsection{Proof of Theorem~\ref{T2}}
\label{ss.proofT2}
We proceed in the following four steps.
\begin{itemize}
\item {\bf Step 0}: Warm up; computation of the first two moments of
  $B_\varepsilon(t)- B_\varepsilon(s)$.
\item {\bf Step 1}: Computation of the first two conditional moments
  of $B_\varepsilon(t)- B_\varepsilon(s)$.
\item {\bf Step 2}: The family $(B_\varepsilon)_{\varepsilon>0}$ is
  tight in $\mathcal C_{\R}([0,\infty))$.
\item {\bf Step 3}: For $(B(t))_{t\geq 0}$ a limit point
  $(B(t))_{t\geq 0}$ as well as $(B(t)^2-t)_{t\geq 0}$ are
  martingales.
\end{itemize}

\noindent Throughout we let $(\mathcal F_t)_{t\geq 0}$ be the
canonical filtration of $\mathcal X =( X_t)_{t\geq 0}$.

\medskip

\noindent
\emph{Step 0: Computation of first two moments of $B_\varepsilon(t)
  - B_\varepsilon(s)$.} The first moment of $B_\varepsilon(t) -
B_\varepsilon(s)$ equals~0 since by assumption the tree-valued
Fleming--Viot process is in equilibrium.

For the second moment, we start by noting that (see Proposition~\ref{P:cover2})
\begin{align}
  \label{eq:imp}
  N_\varepsilon^s = \frac{2}\varepsilon +
  \sqrt{\frac{2}{3\varepsilon}}Z^s +
  o\Big(\frac{1}{\sqrt{\varepsilon}}\Big)
\end{align}
for some random variable $Z^s\sim N(0,1)$. Note that $N_\varepsilon^s$ and
$N_\delta^t$ are independent given
$(t-\delta,t] \cap (s-\varepsilon, s] = \emptyset$. (The reason is
that in this case $N_\varepsilon^s$ depends on resampling events in
the time interval $(s-\varepsilon, s]$ while $N_\delta^t$ only
depends on resampling events in $(t-\delta,t]$, and these two sets
of events are independent.) Without loss of generality, we set $s=0$
and compute the variance of $B_\varepsilon(t)$ as
\begin{equation}
  \label{eq:2789}
  \begin{aligned}
    \mathbf{Var}[B_\varepsilon(t)] & = 3 \int_0^t \int_0^s
    \mathbf{Cov}(N_\varepsilon^r, N_\varepsilon^s) dr ds = 3
    \int_0^t \int_{0\vee
      (s-\varepsilon)}^s\mathbf{Cov}(N_\varepsilon^r,
    N_\varepsilon^s) dr ds \\ & = 3 \int_0^t
    \int_{0}^{\varepsilon\wedge
      t}\mathbf{Cov}(N_\varepsilon^{s-\delta}, N_\varepsilon^s)
    d\delta ds \stackrel{\varepsilon\to 0}{\approx} 6t \cdot
    \int_0^\varepsilon \mathbf{Cov}(N_\varepsilon^{0},
    N_\varepsilon^\delta) d\delta.
  \end{aligned}
\end{equation}
In order to compute the integrand of the last expression, we decompose
$N_\varepsilon^0 = N_{\varepsilon-\delta}^{0} -
(N_{\varepsilon-\delta}^{0} - N_\varepsilon^{0})$. Now
$N_{\varepsilon-\delta}^{0} - N_\varepsilon^{0}$ is independent of
$N_\varepsilon^\delta$. The former is the number of lines the tree at
time~$0$ looses between times $\varepsilon-\delta$ and $\varepsilon$
in the past, and therefore only depends on resampling events between
times $-\varepsilon$ and $\delta-\varepsilon$. The latter only depends
on resampling events between times $\delta-\varepsilon$ and
$\delta$. Hence,
\begin{equation}
  \label{eq:2790}
  \begin{aligned}
    \mathbf{Cov}(N_\varepsilon^{0}, N_\varepsilon^\delta) & =
    \mathbf{Cov}(N_{\varepsilon-\delta}^{0}, N_\varepsilon^\delta).
  \end{aligned}
\end{equation}

Consider now the dual representation of our equilibrium by the Kingman
coalescent. Let $K_i^N$ be number of lines of a subtree, starting with
$N$ lines, in a tree starting with $\infty$ many lines, at the time
the big tree has $i$ lines. In Lemma 4 of \cite{MR2851692} it is shown
that
 \begin{equation}
  \label{eq:2791}
  \begin{aligned}
    \mathbf E[K_i^N] = \frac{iN}{N+i-1}.
  \end{aligned}
\end{equation}
Hence, for independent $Z,Z'\sim N(0,1)$,
\begin{equation}
  \label{eq:2792}
  \begin{aligned}
    \mathbf{Cov}( & N_\varepsilon^{0}, N_\varepsilon^\delta) =
    \mathbf{Cov}(N_{\varepsilon-\delta}^{0}, N_\varepsilon^\delta)
    \\
& = \mathbf{Cov}(N_{\varepsilon-\delta}^{0}, \mathbf
    E[N_\varepsilon^\delta|N_\delta^\delta,
    N_{\varepsilon-\delta}^0])
    \\
& = \mathbf{Cov}(N_{\varepsilon-\delta}^{0}, \mathbf
    E[K_{N_{\varepsilon-\delta}^{0}}^{N_\delta^\delta}|N_\delta^\delta,
    N_{\varepsilon-\delta}^0])\\
& =
    \mathbf{Cov}\Big(N_{\varepsilon-\delta}^{0},
    \frac{N_{\varepsilon-\delta}^0\cdot
      N_{\delta}^\delta}{N_{\varepsilon-\delta}^0 +
      N_{\delta}^\delta -1}\Big) \\
& \stackrel{\varepsilon\to 0}\approx
    \mathbf{Cov}\bigg(\sqrt{2/(3(\varepsilon-\delta))}Z+o(1/\sqrt{\varepsilon}),\\
& \qquad \qquad \frac{\Big(\frac{2}{\varepsilon-\delta} +
      \sqrt{2/(3(\varepsilon-\delta))}Z +
      o(\sqrt{1/\varepsilon}\,)\Big)\Big(\frac{2}{\delta} +
      \sqrt{2/(3\delta)}Z' +
      o(\sqrt{1/\varepsilon}\,)\Big)}{\frac{2}{\varepsilon-\delta} +
      \frac{2}{\delta} + \sqrt{2/(3(\varepsilon-\delta))}Z +
      \sqrt{2/(3\delta)}Z' + o(1/\sqrt{\varepsilon}\,)}\bigg) \\
& =
    \frac{2}{\varepsilon}
    \mathbf{Cov}\bigg(\sqrt{2/(3(\varepsilon-\delta))}Z+o(1/\sqrt{\varepsilon}\,),\\
& \qquad \qquad \frac{\Big(1 + \sqrt{(\varepsilon-\delta)/6}Z +
      o(\sqrt{\varepsilon}\,)\Big) \Big(1+ \sqrt{\delta/6}Z' +
      o(\sqrt{\varepsilon}\,))\Big)}{1+ \frac{\delta}{\varepsilon}
      \sqrt{(\varepsilon-\delta)/6)}Z +
      \frac{\varepsilon-\delta}{\varepsilon}\sqrt{\delta/6}Z' +
      o(\sqrt{\varepsilon}\,)}\bigg) \\
& =
    \frac{2}{\varepsilon}\mathbf{Cov}\Big(
    \sqrt{2/(3(\varepsilon-\delta))}Z+o(1/\sqrt{\varepsilon}\,),\\
&
    \qquad\qquad \qquad \Big(\Big(1 + \sqrt{(\varepsilon-\delta)/6}Z
    + \sqrt{\delta/6}Z' + o(\sqrt{\varepsilon}))\Big)  \\
&
    \qquad \qquad \qquad \qquad \qquad \cdot\Big(1 -
    \frac{\delta}{\varepsilon} \sqrt{(\varepsilon-\delta)/6)}Z -
    \frac{\varepsilon-\delta}{\varepsilon}\sqrt{\delta/6}Z' +
    o(\sqrt{\varepsilon}\,)\Big) \\
& = \frac{2}{\varepsilon}\mathbf{Cov}\Big(
    \sqrt{2/(3(\varepsilon-\delta))}Z+o(1/\sqrt{\varepsilon}\,),
    \frac{\varepsilon-\delta}{\varepsilon}
    \sqrt{(\varepsilon-\delta)/6}Z +
    \frac{\delta}{\varepsilon}\sqrt{\delta/6}Z' +
    o(\sqrt{\varepsilon}\,)\Big) \\
& = \frac{2(\varepsilon-\delta)}{3\varepsilon^2}(1 + o(1))
  \end{aligned}
\end{equation}
leading to
\begin{align}\label{eq:BepsV}
  \mathbf {Var}[B_\varepsilon(t)] &= 3t \int_0^\varepsilon
  \mathbf{Cov}(N_\varepsilon^0, N_\varepsilon^\delta)d\delta
  \stackrel{\varepsilon\to 0}\approx 3t \int_0^\varepsilon
  \frac{2\delta}{3\varepsilon^2} d\delta = t.
\end{align}

\medskip

\noindent
\emph{Step 1: Computation of first two conditional moments of
  $B_\varepsilon(t) - B_\varepsilon(s)$.}  We can compute the \emph{first
  conditional moment} as
\begin{equation}
  \label{eq:2793}
  \begin{aligned}
    \mathbf E[B_\varepsilon(t) - B_\varepsilon(s)|\mathcal F_s] & =
    \sqrt{\frac 32}\int_s^t \mathbf E[N_\varepsilon^r - \mathbf
    E[N_\varepsilon^\infty]|\mathcal F_s] dr = \sqrt{\frac
      32}\int_s^{s+\varepsilon} \mathbf E[N_\varepsilon^r - \mathbf
    E[N_\varepsilon^\infty]|\mathcal F_s] \, dr
  \end{aligned}
\end{equation}
since we started in equilibrium and $N_\varepsilon^r$ is independent
of $\mathcal F_s$ for $r>s+\varepsilon$. So, by
Proposition~\ref{P:cover2},
\begin{equation}
  \label{eq:2794}
  \begin{aligned}
    \mathbf E\big[\big(\mathbf E[B_\varepsilon(t) -
    B_\varepsilon(s)|\mathcal F_s]\big)^2\big] & \lesssim
    \varepsilon^2 \mathbf E[(N_\varepsilon^\infty - \mathbf
    E[N_\varepsilon^\infty])^2] \xrightarrow{\varepsilon\to 0} 0,
  \end{aligned}
\end{equation}
which implies
\begin{align}\label{eq:Bepsconv1}
  \mathbf E[B_\varepsilon(t)-B_\varepsilon(s)|\mathcal F_s]
  \xrightarrow{\varepsilon\to 0} 0 \qquad \text{ in }L^2.
\end{align}

For the \emph{second conditional moment}, we extend our calculation from
Step~0. Here,
\begin{equation}
  \label{eq:2795}
  \begin{aligned}
    \mathbf E[(B_\varepsilon(t)-B_\varepsilon(s))^2|\mathcal F_s] & =
    3 \int_s^t \int_s^{r_1} \mathbf{E}[(N_\varepsilon^{r_1}-\mathbf
    E[N_\varepsilon^{\infty}])(N_\varepsilon^{r_2}-\mathbf
    E[N_\varepsilon^{\infty}])|\mathcal F_s] dr_1 dr_2 \\ & =
    3\int_{s+\varepsilon}^t \int_{s+\varepsilon}^{r_1}
    \mathbf{Cov}[N_\varepsilon^{r_1}, N_\varepsilon^{r_2}] dr_1 dr_2 +
    3A_\varepsilon,
  \end{aligned}
\end{equation}
with
\begin{equation}
  \label{eq:2796}
  \begin{aligned}
    \mathbf E[|A_\varepsilon|] & = \mathbf E\Big[\Big|\int_s^t
    \int_s^{r_1\wedge
      (s+\varepsilon)}\mathbf{E}[(N_\varepsilon^{r_1}-\mathbf
    E[N_\varepsilon^{\infty}])(N_\varepsilon^{r_2}-\mathbf
    E[N_\varepsilon^{\infty}])|\mathcal F_s] dr_1 dr_2\Big|\Big] \\
    & = \mathbf E\Big[\Big|\int_s^{s+2\varepsilon} \int_s^{r_1\wedge
      (s+\varepsilon)} \mathbf{E}[|(N_\varepsilon^{r_1}-\mathbf
    E[N_\varepsilon^{\infty}])(N_\varepsilon^{r_2}-\mathbf
    E[N_\varepsilon^{\infty}])|\mathcal F_s] dr_1 dr_2\Big|\Big] \\
    & \leq \int_s^{s+2\varepsilon} \int_s^{r_1\wedge
      (s+\varepsilon)} \mathbf{E}[|(N_\varepsilon^{r_1}-\mathbf
    E[N_\varepsilon^{\infty}])(N_\varepsilon^{r_2}-\mathbf
    E[N_\varepsilon^{\infty}])|] dr_1 dr_2 \\ & \leq
    \int_s^{s+2\varepsilon} \int_s^{r_1}
    \mathbf{E}[|(N_\varepsilon^{r_1}-\mathbf
    E[N_\varepsilon^{\infty}])(N_\varepsilon^{r_2}-\mathbf
    E[N_\varepsilon^{\infty}])|] dr_1 dr_2 \\ & \leq 2\varepsilon^2
    \mathbf {Var}[N_\varepsilon^\infty] \xrightarrow{\varepsilon\to 0} 0
  \end{aligned}
\end{equation}
by Proposition~\ref{P:cover2}.  So, combining the last two displays
with~\eqref{eq:BepsV},
\begin{align}\label{eq:2795}
  \mathbf E[(B_\varepsilon(t) - B_\varepsilon(s))^2|\mathcal F_s]
  \xrightarrow{\varepsilon\to 0}t-s \qquad \text{ in }L^1.
\end{align}

\medskip

\noindent\emph{Step 2: The family of the laws of
  $(B_\varepsilon)_{\varepsilon>0}$  on $\mathcal C_{\R}([0,\infty))$ is tight for
  $\ve \to 0$.} We use the Kolmogorov--Chentsov
criterion; see e.g.\ Corollary 16.9 in \cite{Kallenberg2002}. We bound
the fourth moment of the increment in $B_\varepsilon$. We write
$a_\varepsilon := \mathbf E[N_\varepsilon^\infty]$ such that, again
making use of the independence of $N_\varepsilon^{s_2}$ and
$N_\varepsilon^{s_1}$ if $|s_2-s_1| > \varepsilon$, as well as of
Proposition~\ref{P:cover2}, we estimate for fixed $t$ and $\ve \to 0$
\begin{align}
  \label{eq:851}
  \begin{split}
    \mathbf E[ (B_\varepsilon(t &))^4] \lesssim \int_0^t \int_0^{s_1}
    \int_0^{s_2} \int_0^{s_3} \mathbf E\Bigl[ \prod\nolimits_{i=1}^4
    (N_\varepsilon^{s_i}-a_\varepsilon)\Bigr]
    \, ds_4 \, ds_3 \, ds_2 \, ds_1 \\ & \lesssim \Big(\int_0^t
    \int_{0\vee(s_1-\varepsilon)}^{s_1} \mathbf
    E[(N_\varepsilon^{s_1}-a_\varepsilon)\cdot
    (N_\varepsilon^{s_2}-a_\varepsilon)] ds_2 ds_1\Big)^2 \\ & \quad +
    \int_0^t \int_{0\vee(s_1-\varepsilon)}^{s_1}
    \int_{0\vee(s_2-\varepsilon)}^{s_2}
    \int_{0\vee(s_3-\varepsilon)}^{s_3}
    \mathbf E\Bigl[ \prod\nolimits_{i=1}^4
    (N_\varepsilon^{s_i}-a_\varepsilon)\Bigr] \, ds_4 \, ds_3 \, ds_2
    \, ds_1 \\ &
    \lesssim \big(t (t\wedge \varepsilon) \mathbf
    {Var}[N_\varepsilon^\infty]\big)^2 + t(t\wedge \varepsilon)^3 \mathbf
    E[(N_\varepsilon^\infty - a_\varepsilon)^4] \\ & \lesssim t^2
  \end{split}
\end{align}
and the tightness follows.

\medskip

\noindent\emph{Step 3: If $(B_t)_{t\geq 0}$ is a limit point, then
  $(B_t)_{t\geq 0}$ as well as $(B_t^2-t)_{t\geq 0}$ are
  martingales.} Let $\mathcal B = (B_t)_{t\geq 0}$ be a weak limit point of
$\{(B_\varepsilon(t))_{t\geq 0}: \varepsilon>0\}$, which has
continuous paths by Step~2. We know from Step~0 that $B_t$ is square
integrable which allows for the following calculations. For $0\leq
r_1< \cdots \leq r_n\leq s$ and some continuous bounded function $f:
\mathbb R^n\to\mathbb R$, by \eqref{eq:Bepsconv1}
\begin{equation}
  \label{eq:432a}
  \begin{aligned}
    \big|\mathbf E[(B_t - B_s) & \cdot f(B_{r_1},\dots,B_{r_n})]\big| =
    \lim_{\varepsilon\to\infty} \big|\mathbf E[(B_\varepsilon(t) - B_\varepsilon(s))
    \cdot f(B_\varepsilon(r_1),\dots,B_\varepsilon(r_n))]\big| \\ & \leq
    \lim_{\varepsilon\to\infty} \mathbf E\Big[ \Big|\mathbf E[B_\varepsilon(t) -
    B_\varepsilon(s)|\mathcal F_s] \Big|\cdot
    f(B_\varepsilon(r_1),\dots,B_\varepsilon(r_n))\Big] = 0.
  \end{aligned}
\end{equation}
Since $r_1,\dots,r_n$ and $f$ were arbitrary, $\mathcal B$ is a
martingale. Similarly, by~\eqref{eq:2795},
\begin{equation}
  \label{eq:432b}
  \begin{aligned}
    \big|\mathbf E[( & (B_t - B_s)^2-(t-s)) \cdot f(B_{r_1},\dots,B_{r_n})]\big|
    \\ & = \lim_{\varepsilon\to\infty} \big|\mathbf E[((B_\varepsilon(t) -
    B_\varepsilon(s))^2 - (t-s)) \cdot
    f(B_\varepsilon(r_1),\dots,B_\varepsilon(r_n))]\big| \\ & \leq
    \lim_{\varepsilon\to\infty} \mathbf E\Big[ \Big|\mathbf
    E[(B_\varepsilon(t) - B_\varepsilon(s))^2-(t-s)|\mathcal F_s]
    \Big|\cdot f(B_\varepsilon(r_1),\dots,B_\varepsilon(r_n))\Big] = 0,
  \end{aligned}
\end{equation}
which shows that $(B_t^2 - t)_{t\geq 0}$ is a martingale. Then by
L\'evy's characterization of Brownian motion, $\mathcal B$ is a
Brownian motion.

\section{Preparation: Computing moments of
  $\mathbf{\Psi_\lambda^{12}}$ and
  $\mathbf{\widehat\Psi_\lambda^{12}}$}
The \label{S:prep} key to the proofs of Proposition~\ref{P:smallballs}
and Theorems~\ref{T3},~\ref{T4},~\ref{T5} and~\ref{T6} is a proper
understanding of the behavior of the functions
$\Psi_\lambda^{12}(X_t)$ and $\widehat\Psi_\lambda^{12}(X_t)$ from
Remark~\ref{rem:intPhi} along the paths of the tree-valued
Fleming--Viot dynamics. In this section, we provide useful tools for
the analysis of these functions. In particular, we are going to
compute moments up to fourth order.

Throughout this section we assume that
\begin{align}\label{dg11}
  \mathcal X = (X_t)_{t\geq 0} \text{ is the neutral tree-valued
    Fleming--Viot equilibrium process},
\end{align}
i.e.\ with $X_0 \stackrel d = X_\infty$, and denote by $(\mathcal
F_t)_{t\geq 0}$ the canonical filtration of $\mathcal X$.
Furthermore, for the mutation operator, we will assume throughout this
section that
\begin{align}
  \label{eq:beta-non-atomic}
  \beta(u,\cdot) \quad \text{is non-atomic for all} \quad u\in A.
\end{align}
This assumption implies that every mutation event leads to a new
type. In particular, this will be crucial in Lemma~\ref{l1}.

\subsection{Results on moments of $\Psi^{12}_\lambda(X_t),
  {\widehat\Psi}^{12}_\lambda (X_t)$ and its increments}
\label{ss.behPsi}
The key is to obtain the power at which $\Psi^{12}_\lambda$ vanishes
as $\lambda \to \infty$ respectively at which order the increments of
$t \mapsto ((\lambda +1) \Psi^{12}_\lambda (X_t)-1)$ vanish as $t \to
0$. Proofs of the next two lemmata are given in
Section~\ref{ss.proofLs}.

\begin{lemma}[Convergence for fixed times]\mbox{}\\
  For   \label{l2} all $t\geq 0$,
  \begin{equation}
    \label{eq:l2}
    \begin{aligned}
      (\lambda+1)\Psi^{12}_\lambda(X_t) & \xrightarrow{\lambda\to\infty} 1,\\
      (\lambda+2\vartheta + 1)\widehat \Psi^{12}_\lambda(X_t) &
      \xrightarrow{\lambda\to\infty} 1
    \end{aligned}
  \end{equation}
  in $L^2$ (and therefore also in probability).
\end{lemma}

\begin{lemma}[Second moment of increments of $t\mapsto
  ((\lambda+1)\Psi^{12}_\lambda(X_t)$]\mbox{}\\
  There \label{l3}exists $C>0$, which is
  independent of $\lambda$, such that
  \begin{align}\label{eq:31a}
    & \sup_{\lambda > 0} (\lambda+1)^2 \mathbf
    E[(\Psi^{12}_\lambda(X_t) - \Psi^{12}_\lambda(X_0))^2]\leq
    Ct,\\\label{eq:31b}
    & \sup_{\lambda > 0} (\lambda+2\vartheta + 1)^2 \mathbf
    E[(\widehat \Psi^{12}_\lambda(X_t) - \widehat
    \Psi^{12}_\lambda(X_0))^2]\leq Ct.
  \end{align}
\end{lemma}

\begin{remark}[Tightness]\mbox{}\\
  If the right hand side of \eqref{eq:31a} would have been
  $Ct^{1+\varepsilon}$ for some $\varepsilon>0$, then Lemma~\ref{l3}
  would imply the needed tightness for Theorem~\ref{T3} by the
  Kolmogorov--Chentsov tightness criterion; see e.g.\ Corollary 16.9
  in \cite{Kallenberg2002}. However, since it is only linear in $t$,
  we will have to use fourth moments to get the desired tightness
  property.
\end{remark}

The main goal of this subsection is therefore to compute the fourth
moments of the increments of $t\mapsto (\lambda+1)\Psi^{12}(X_t)$ and
of $t\mapsto (\lambda+2\vartheta+1)\widehat\Psi^{12}(X_t)$.
The proof of the next two results are given in
Sections~\ref{ss.proofsl4l5} and~\ref{ss:56}.

\begin{lemma}[A fourth moment of $X_\infty$]\mbox{}\\
  For all $t\geq 0$, \label{l4b} as $\lambda\to\infty$,
  \begin{align}
    \mathbf E\bigl[\bigl((\lambda+1)\Psi^{12}_\lambda(X_\infty)-1\bigr)^4\bigr]
    & = \frac{3}{4\lambda^2} + \mathcal O\Big(\frac{1}{\lambda^3}\Big),\\
    \mathbf E\bigl[\bigl((\lambda+2\vartheta +
    1)\widehat \Psi^{12}_\lambda(X_\infty)-1\bigr)^4\bigr] & =
    \frac{3}{4\lambda^2} +
    \mathcal O\Big(\frac{1}{\lambda^3}\Big).
  \end{align}
  In particular, the convergences in \eqref{eq:l2} hold in $L^4$ as well.
\end{lemma}

\begin{lemma}[Fourth moment of increment of $t\mapsto
  ((\lambda+1)\Psi^{12}_\lambda(X_t)$]\mbox{}\\
  There \label{l5} exists $C>0$, such that
  \begin{equation}
    \label{eq:310}
    \begin{aligned}
      &\sup_{\lambda > 0} (\lambda+1)^4 \mathbf
      E[(\Psi^{12}_\lambda(X_t) -
      \Psi^{12}_\lambda(X_0))^4]\leq Ct^2,\\
      &\sup_{\lambda > 0} (\lambda+2\vartheta+1)^4 \mathbf
      E[(\widehat\Psi^{12}_\lambda(X_t) -
      \widehat\Psi^{12}_\lambda(X_0))^4]\leq Ct^2.
    \end{aligned}
  \end{equation}
\end{lemma}

We start with second moments and afterwards, we provide an automated
procedure how to compute higher order moments using
\textsc{Mathematica}.

\subsection{Moments up to second order}
Here, we show that $(\lambda+1)\Psi^{12}_\lambda
\xrightarrow{\lambda\to\infty} 1$ and
$(\lambda+2\vartheta+1)\widehat\Psi^{12}_\lambda
\xrightarrow{\lambda\to\infty} 1$ in $L^2$ (Lemma~\ref{l2}) and study
the second moments of the increments of $t\mapsto
((\lambda+1)\Psi^{12}_\lambda(X_t)-1)$ (Lemma~\ref{l3}). We start by
defining some functions which appear in the first and second moment
equations for the polynomials $\Psi_\lambda^{12}$ and $ \widehat
\Psi_\lambda^{12}$ defined in Example~\ref{rem:intPhi}.

\begin{definition}[Some functions on $\mathbb U$]\mbox{}\\
  We \label{def:PsiPhi} define the following functions in $\Pi^1$:
  \begin{align}
    \label{eq:21}
    \begin{split}
      \Psi^{\emptyset}(\overline{(U,r,\mu)}) & \coloneqq 1, \\
      \Psi_\lambda^{12}(\overline{(U,r,\mu)}) & \coloneqq \langle
      \mu^{\otimes \N},
      e^{-\lambda r(u_1, u_2)}\rangle,\\
       \Psi_{\lambda}^{12,12}(\overline{(U,r,\mu)}) &
        \coloneqq \Psi_{2\lambda}^{12} (\overline{(U,r,\mu)}) \\
      \Psi_{\lambda}^{12,23}(\overline{(U,r,\mu)}) & \coloneqq \langle
      \mu^{\otimes \N}, e^{-\lambda (r(u_1, u_2) + r(u_2,u_3))}\rangle,\\
      \Psi_{\lambda}^{12,34}(\overline{(U,r,\mu)}) & \coloneqq \langle
      \mu^{\otimes \N}, e^{-\lambda (r(u_1, u_2) + r(u_3,u_4))}\rangle,\\
      \widehat\Psi^{\emptyset}(\overline{(U,r,\mu)}) & \coloneqq 1,\\
      \widehat \Psi_\lambda^{12}(\overline{(U,r,\mu)}) & \coloneqq
      \langle \mu^{\otimes \N},
      \ind{a_1=a_2}e^{-\lambda r(u_1, u_2)}\rangle,\\
       \widehat \Psi_\lambda^{12,12}(\overline{(U,r,\mu)}) &
       \coloneqq \widehat
        \Psi_{2\lambda}^{12}(\overline{(U,r,\mu)}) ,\\
        \widehat \Psi_{\lambda}^{12,23}(\overline{(U,r,\mu)}) & \coloneqq
        \langle \mu^{\otimes \N}, \ind{a_1=a_2} \ind{a_2=a_3}e^{-\lambda
        (r(u_1, u_2) + r(u_2,u_3))}\rangle,\\
      \Psi_{\lambda}^{12,34}(\overline{(U,r,\mu)}) & \coloneqq \langle
      \mu^{\otimes \N}, \ind{a_1=a_2} \ind{a_3=a_4} e^{-\lambda
        (r(u_1, u_2) + r(u_3,u_4))}\rangle
    \end{split}
  \end{align}
  and centering around the equilibrium expectations
{\allowdisplaybreaks
\begin{eqnarray}
    \label{eq:212}
    \Upsilon^{12}_\lambda & \coloneqq & \Psi_\lambda^{12} - \frac{1}{\lambda+1}, \nonumber\\
    \Upsilon^{12,12}_\lambda & \coloneqq &
    \Upsilon^{12}_{2\lambda},\nonumber\\
    \Upsilon_{\lambda}^{12,23} & \coloneqq &\Psi_{\lambda}^{12,23} -
    \frac{1}{2} \Psi_{2\lambda}^{12} -
    \frac{2}{\lambda+2}\Psi_\lambda^{12} + \frac{\lambda+6}{2
      (\lambda + 2) (2\lambda + 3)},\nonumber\\
    \Upsilon_{\lambda}^{12,34} & \coloneqq &\Psi_{\lambda}^{12,34} -
    \frac 43 \Psi_{\lambda}^{12,23} + \frac{4}{15}
    \Psi_{2\lambda}^{12} + \frac{2}{3(\lambda+5)} \Psi_\lambda^{12}
    - \frac{
      2\lambda+15}{15(\lambda+3)(\lambda+5)},\nonumber\\
    \widehat\Upsilon^{12}_\lambda & \coloneqq &\Psi_\lambda^{12} -
    \frac{1}{\lambda + 2\vartheta + 1},\\
    \widehat\Upsilon^{12,12}_\lambda & \coloneqq & \widehat\Upsilon^{12}_{2\lambda},\nonumber\\
    \widehat\Upsilon^{12,23}_\lambda & \coloneqq&
    \widehat\Psi_{\lambda}^{12,23} - \frac{1}{2+\vartheta}
    \widehat\Psi_{2\lambda}^{12} -
    \frac{2}{\lambda+2\vartheta+2}\widehat\Psi_\lambda^{12} \nonumber\\ &&
    \qquad \qquad \qquad \qquad \qquad \qquad +
    \frac{\lambda+3\vartheta+6}{(\vartheta+2) (\lambda + \vartheta +
      2) (2\lambda + 3\vartheta+3)},\nonumber\\
    \widehat\Upsilon_{\lambda}^{12,34} & \coloneqq&
    \widehat\Psi_{\lambda}^{12,34} - \frac{4}{\vartheta+3}
    \widehat\Psi_{\lambda}^{12,23} +
    \frac{4}{(\vartheta+3)(2\vartheta+5)}
    \widehat\Psi_{2\lambda}^{12} +
    \frac{2(1-\vartheta)}{(3+\vartheta)(\lambda+2\vartheta+5)}
    \widehat\Psi_\lambda^{12} \nonumber\\* & &\qquad \qquad \qquad \qquad
    \qquad \qquad - \frac{2\lambda + \vartheta - 2\vartheta^2 +
      15}{(\vartheta+3) (2\vartheta+5) (\lambda+2\vartheta+3)
      (\lambda + 2\vartheta + 5)}.\nonumber
  \end{eqnarray}}
\end{definition}

\begin{remark}[Connections and a first moment]\mbox{}\\
  Our main object of study are the functions $\Psi_\lambda^{12}$ and
  $\widehat\Psi_\lambda^{12}$. Note that since $X_0$ has the same law
  as $X_\infty$ and since in $X_\infty$ the time it takes that two
  randomly chosen lines coalesce is an exponential random variable $R$
  with parameter $1$, we have
  \begin{align}
    \mathbf E[\Psi_\lambda^{12}(\mathcal U_0)] & = \mathbf E[\langle
    \mu^{\otimes \N}, e^{-\lambda r(u_1,u_2)}\rangle] = \mathbf
    E[e^{-\lambda R}] = \frac{1}{\lambda+1}.
  \end{align}
  Moreover, if mutations arise at constant rate $\vartheta>0$,
  consider again two randomly chosen lines which coalesce at time $R$,
  and the first mutation event at an independent, exponentially
  distributed time $S$ with rate $\vartheta$, using
  (\ref{eq:beta-non-atomic}):
  \begin{equation}
    \begin{aligned}
      \mathbf E[\widehat\Psi_\lambda^{12}(\mathcal U_0)] & = \mathbf
      E[\langle \mu^{\otimes \N}, \ind{a_1=a_2}e^{-\lambda
        r(u_1,u_2)}\rangle] = \mathbf E[e^{-\lambda R} \ind{R\leq S}]
      \\ & = \mathbf E[e^{-\lambda R} \cdot \mathbf P(S\geq R|R)] =
      \mathbf E[e^{-(\lambda + \vartheta)R}] =
      \frac{1}{\lambda+\vartheta+1}.
    \end{aligned}
  \end{equation}
  Moreover, we write
  \begin{equation}
    \begin{aligned}
      \left(\Psi^{12}_\lambda(\overline{(U,r,\mu)})\right)\cdot
      \left(\Psi^{12}_{\lambda}(\overline{(U,r,\mu)})\right) & =
      \langle \mu^{\otimes\N}, e^{-\lambda r(u_1,u_2)}\rangle \cdot
      \langle \mu^{\otimes\N}, e^{-\lambda r(u_1,u_2)}\rangle \\ & =
      \langle \mu^{\otimes \N}, e^{-\lambda ( r(u_1, u_2) + r(u_3,
        u_4))}\rangle = \Psi^{12,34}_{\lambda}(\overline{(U,r,\mu)})
    \end{aligned}
  \end{equation}
  and analogously for $\widehat\Psi^{12,34}_\lambda$.
\end{remark}

Next we compute the action of the generator of ${\mathcal X}$ on the functions from
Definition~\ref{def:PsiPhi}.
\begin{lemma}[$\Omega\Psi_\lambda$ and $\Omega\Upsilon_\lambda$]\mbox{}\\
  1.\, Let \label{l1} $\Omega$ be the generator of the tree-valued
  Fleming--Viot dynamics. Then
  \begin{equation}
    \label{eq:189}
    \begin{aligned}
      \Omega\Psi_\lambda^{12}& = -(\lambda+1)\Psi_\lambda^{12} + 1\cdot
      \Psi_\lambda^\emptyset,\\
      \Omega\Psi_{\lambda}^{12,23} & = -(2\lambda + 3)\Psi_{\lambda}^{12,23} +
      2\Psi_\lambda^{12} + \Psi_{2\lambda}^{12},\\
      \Omega\Psi_{\lambda}^{12,34} & = -(2\lambda + 6)\Psi_{\lambda}^{12,34} +
      4\Psi_{\lambda}^{12,23} +
      2\Psi_\lambda^{12},\\
      \Omega\Upsilon_\lambda^{12} & = -(\lambda+1) \Upsilon_\lambda^{12},\\
      \Omega\Upsilon_{\lambda}^{12,23} & = -(2\lambda +3) \Upsilon_{\lambda}^{12,23},\\
      \Omega\Upsilon_{\lambda}^{12,34} & = -(2\lambda + 6)
      \Upsilon_{\lambda}^{12,34},
    \end{aligned}
  \end{equation}
  and
  \begin{equation}
    \label{eq:189b}
    \begin{aligned}
      \Omega\widehat\Psi_\lambda^{12}& =
      -(\lambda+2\vartheta+1)\widehat
      \Psi_\lambda^{12} + 1\cdot \Psi_\lambda^\emptyset,\\
      \Omega\widehat \Psi_{\lambda}^{12,23} & = -(2\lambda +
      3\vartheta + 3)\widehat\Psi_{\lambda}^{12,23} +
      4\widehat\Psi_\lambda^{12},\\
      \Omega\widehat\Psi_{\lambda}^{12,34} & = -(2\lambda + 4\vartheta
      + 6)\widehat\Psi_{\lambda}^{12,34} +
      4\widehat\Psi_{\lambda}^{12,23} +
      2\widehat\Psi_\lambda^{12},\\
      \Omega\widehat\Upsilon_\lambda^{12} & = -(\lambda + 2\vartheta + 1) \widehat\Upsilon_\lambda^{12},\\
      \Omega\widehat\Upsilon_{\lambda}^{12,23} & = -(2\lambda +
      3\vartheta + 3)
      \widehat \Upsilon_{\lambda}^{12,23},\\
      \Omega\widehat\Upsilon_{\lambda}^{12,34} & = -(2\lambda +
      4\vartheta +6) \widehat\Upsilon_{\lambda}^{12,34}.
    \end{aligned}
  \end{equation}

  \noindent
  2.\ The following functionals of $\mathcal X = (X_t)_{t \ge 0}$ are
  martingales:
  \begin{align}
    \label{eq:98}
    & \bigl(e^{(\lambda+1) t} \Upsilon_\lambda^{12}(X_t)\bigr)_{t\geq
      0}, \; \bigl(e^{(2\lambda+3) t}
    \Upsilon_{\lambda}^{12,23}(X_t)\bigr)_{t\geq 0},
    \;  \bigl(e^{(2\lambda+6) t} \Upsilon_{\lambda}^{12,34}(X_t)\bigr)_{t\geq 0},\\
    \intertext{and} \label{eq:98a} & \bigl(e^{(\lambda + 2\vartheta +
      1) t} \widehat\Upsilon_\lambda^{12}(X_t)\bigr)_{t\geq 0}, \;
    \bigl(e^{(2\lambda+3\vartheta+3) t}
    \widehat\Upsilon_{\lambda}^{12,23}(X_t)\bigr)_{t\geq 0}, \;
    \bigl(e^{(2\lambda+4\vartheta +6) t} \widehat
    \Upsilon_{\lambda}^{12,34}(X_t) \bigr)_{t\geq 0}.
  \end{align}
  Furthermore for all $s\geq 0$ and all $k\in\{12; 12,23;
  12,34\}$
  \begin{align}\label{eq:EPhi0}
    \mathbf E[\Upsilon_\lambda^{k}(X_s)] = \mathbf
    E[\widehat\Upsilon_{\lambda}^{k}(X_s)] = 0.
  \end{align}
\end{lemma}

\begin{remark}[Equilibrium expectations of $\Psi^k_\lambda,
  \widehat\Psi^k_\lambda$]\mbox{}\\
  Since \label{rem:reference}$\mathbf E[\Upsilon^k_\lambda(X_\infty)] = \mathbf
  E[\widehat\Upsilon^k_\lambda(X_\infty)] = 0$ by~\eqref{eq:EPhi0}, it is also
  possible to compute the equilibrium values for $\Psi^k_\lambda$ and
  $\widehat\Psi^k_\lambda$. We write for later reference,
  \begin{equation}\label{eq:sometimes}
    \begin{aligned}
      \Upsilon^{12,23}_{\lambda} & = \Psi^{12,23}_{\lambda} - \frac{5
        \lambda + 3}{(\lambda+1)(2\lambda+1)(2\lambda+3)} \\ & \qquad
      \qquad \qquad \qquad \qquad \qquad - \frac 12 \Big(
      \Psi^{12}_{2\lambda} - \frac{1}{2\lambda+1}\Big) -
      \frac{2}{\lambda+2}\Big(\Psi^{12}_\lambda -
      \frac{1}{\lambda+1}\Big),\\
      \Upsilon^{12,34}_{\lambda} & = \Psi^{12,34}_{\lambda} - \frac{4
        \lambda^2 + 18 \lambda + 9}{(\lambda+1) (\lambda+3)
        (2\lambda+1) (2\lambda+3)} \\ & \qquad - \frac 43\Big(
      \Psi^{12,23}_{\lambda} - \frac{5 \lambda +
        3}{(\lambda+1)(2\lambda+1)(2\lambda+3)}\Big) \\ & \qquad
      \qquad + \frac{4}{15} \Big( \Psi^{12}_{2\lambda} -
      \frac{1}{2\lambda+1}\Big) + \frac{2}{3(\lambda+5)}
      \Big(\Psi^{12}_\lambda - \frac{1}{\lambda+1}\Big)
    \end{aligned}
  \end{equation}
  and
  \begin{equation}\label{eq:sometimes2}
    \begin{aligned}
      \widehat\Upsilon^{12,23}_{\lambda} & =
      \widehat\Psi^{12,23}_{\lambda} - \frac{5\lambda + 6\vartheta +
        3}{(\lambda+2\vartheta +1) (2\lambda + 2\vartheta + 1)
        (2\lambda+3\vartheta + 3)} \\ & \qquad \qquad - \frac
      1{\vartheta +2} \Big( \widehat\Psi^{12}_{2\lambda} -
      \frac{1}{2\lambda+2\vartheta + 1}\Big) -
      \frac{2}{\lambda+\vartheta+2}\Big(\widehat\Psi^{12}_\lambda -
      \frac{1}{\lambda+2\vartheta+1}\Big),\\
      \widehat\Upsilon^{12,34}_{\lambda} & =
      \widehat\Psi^{12,34}_{\lambda} - \frac{4 \lambda^2 + 2
        \lambda(5\vartheta+9) + 21\vartheta + 6\vartheta^2 +
        9}{(\lambda+2\vartheta+1) (\lambda+2\vartheta+3)
        (2\lambda+2\vartheta+1) (2\lambda+3\vartheta+3)} \\ & \qquad -
      \frac 4{3+\vartheta}\Big(\widehat\Psi^{12,23}_{\lambda} -
      \frac{5 \lambda + 6\vartheta +
        3}{(\lambda+2\vartheta+1)(2\lambda+2\vartheta+1)(2\lambda+3\vartheta+3)}\Big)
      \\ & \qquad \qquad \qquad \qquad +
      \frac{4}{(\vartheta+3)(2\vartheta+5)} \Big(
      \widehat\Psi^{12}_{2\lambda} -
      \frac{1}{2\lambda+2\vartheta+1}\Big) \\ & \qquad \qquad \qquad
      \qquad \qquad \qquad +
      \frac{2(1-\vartheta)}{(\vartheta+3)(\lambda+2\vartheta+5)}
      \Big(\widehat\Psi^{12}_\lambda -
      \frac{1}{\lambda+2\vartheta+1}\Big)
    \end{aligned}
  \end{equation}
  where the terms after each $\Psi^k_\lambda$ and
  $\widehat\Psi^k_\lambda$ are the corresponding equilibrium values.
\end{remark}

\subsection{Proofs of Lemmata~\ref{l2}, \ref{l3} and \ref{l1}}
\label{ss.proofLs}

\begin{proof}[Proof of Lemma~\ref{l1}]
  The results in \eqref{eq:189} and~\eqref{eq:189b} follow by simple
  calculations. The stated martingale properties are consequences of general
  theory; see Lemma~4.3.2 in \cite{EthierKurtz86}. In particular, since the
  equilibrium distribution of $\mathcal X$ is ergodic, for $k\in\{12; 12,23;
  12,34\}$ and $\lambda_{12} = \lambda+1, \lambda_{12,23} = 2\lambda + 3,
  \lambda_{12,34} = 2\lambda + 6$,
  \begin{align}
    \mathbf E[\Upsilon^k_{\lambda}(X_s)] = \lim_{t\to \infty} \mathbf
    E[\Upsilon^k_{\lambda}(X_t)] = \lim_{t\to \infty} e^{-\lambda_k
      (t-s)}\mathbf E[\Upsilon^k_{\lambda}(X_s)] = 0\cdot \mathbf
    E[\Upsilon^k_{\lambda}(X_s)],
  \end{align}
  which can only hold (note that $|\mathbf E[\Upsilon^k_{
    \lambda}(X_s)]|<\infty$ by definition) if $\mathbf
  E[\Upsilon^k_{\lambda}(X_s)] = 0$. Replacing $\lambda_k$ by
  $\widehat\lambda_k$, $\widehat\lambda_{12} = \lambda + 2\vartheta +
  1, \widehat\lambda_{12,23} = 2\lambda + 3\vartheta + 3,
  \widehat\lambda_{12,34} = 2\lambda + 4\vartheta + 6$ gives the
  corresponding results for the expectations of
  $\widehat\Upsilon^k_\lambda$.
\end{proof}

\begin{proof}[Proof of Lemma~\ref{l2}]
  We only give the full proof for the first assertion in~\eqref{eq:l2}
  since the second follows exactly along the same lines. The proof
  amounts to computing the variance of $\Psi_\lambda^{12}$. Therefore,
  we need to understand the expectation of $(\Psi_\lambda^{12})^2 =
  \Psi_\lambda^{12,34}$ (for the last equality,
  see~\eqref{eq:31l}). To this end we express the $\Psi_\lambda$'s in
  term of the $\Upsilon_\lambda$'s and obtain
  \begin{equation}
    \label{eq:16}
    \begin{aligned}
      \Psi_\lambda^{12} &= \Upsilon_\lambda^{12} + \frac1{\lambda+1},  \\
      \Psi_\lambda^{12,23} &=\Upsilon_\lambda^{12,23} + \frac{1}2
      \Upsilon_\lambda^{12,12} + \frac{2}{2+\lambda}\Upsilon_\lambda^{12} +
      \frac{5\lambda+3}{(\lambda+1)(2\lambda+1)(2\lambda+3)}, \\
      \Psi_\lambda^{12,34} & = \Upsilon_\lambda^{12,34} + \frac 43
      \Upsilon_\lambda^{12,23} + \frac{2}{5} \Upsilon_\lambda^{12,12} +
      \frac{2(\lambda+6)}{(\lambda+2)(\lambda+5)}\Upsilon_\lambda^{12} \\
      & \qquad \qquad \qquad \qquad \qquad \qquad \qquad +
      \frac{4\lambda^2+18\lambda+9}{(\lambda+1)(\lambda+3)(2\lambda+1)(2\lambda+3)}.
    \end{aligned}
  \end{equation}
  Together with~\eqref{eq:EPhi0}, this implies
  \begin{equation}
    \label{eq:43}
    \begin{aligned}
      \mathbf E[((\lambda+1)\Psi^{12}(X_t) -1)^2] & = (\lambda+1)^2 \cdot
      \mathbf E[ \Psi_\lambda^{12,34}(X_t)] -
      2(\lambda+1)\cdot \mathbf E[ \Psi_\lambda^{12}(X_t)] + 1 \\
      & = \frac{2 \lambda^2} {(\lambda+3) (2 \lambda+1) (2 \lambda +3)}
      \xrightarrow{\lambda\to\infty}0.
    \end{aligned}
  \end{equation}
\end{proof}

\begin{proof}[Proof of Lemma~\ref{l3}]
  Again, we restrict our proof to~\eqref{eq:31a} since~\eqref{eq:31b}
  is proved analogously. We compute
  \begin{equation}
    \begin{aligned}
      \mathbf E[\Psi^{12}_\lambda({X_0})\Upsilon^{12}_\lambda({X_0})]
      & = \mathbf E\Big[\Psi^{12,34}_\lambda(X_0) -
      \frac{1}{\lambda+1}\Psi^{12}_\lambda(X_0)\Big] \\ & = \frac{2
        \lambda^2}{(\lambda+1)^2(\lambda + 3) (2\lambda + 1) (2\lambda
        + 3)}
    \end{aligned}
  \end{equation}
  and (recall that we start in equilibrium)
  \begin{equation}
    \begin{aligned}
      \mathbf E[(\Psi^{12}_\lambda(X_t) & - \Psi^{12}_\lambda(X_0))^2]
      \\& = \mathbf E[(\Psi^{12}_\lambda(X_t))^2 -
      (\Psi^{12}_\lambda(X_0))^2 - 2\Psi^{12}_\lambda(X_0)
      (\Psi^{12}_\lambda(X_t)-\Psi^{12}_\lambda(X_0))] \\ & = \mathbf
      E\Big[2\Psi^{12}_\lambda(X_0) \int_0^t \Omega
      \Upsilon^{12}_\lambda(X_s) ds\Big] \\ & = \int_0^t \mathbf
      E[2\Psi^{12}_\lambda(X_0)(\lambda+1) \mathbf
      E[\Upsilon^{12}_\lambda(X_s)|\mathcal F_0]] ds \\ & =
      2\int_0^t(\lambda+1)e^{-(\lambda+1)s}ds \cdot \mathbf
      E[\Psi^{12}_\lambda({X_0})\Upsilon^{12}_\lambda({X_0})] \\
      & \le \frac{4}{(\lambda+1)^2} t
    \end{aligned}
  \end{equation}
  and the result follows.
\end{proof}

\subsection{Moments up to fourth order}
\label{ss.mom4}

For the proofs of Lemma~\ref{l4b} and Lemma~\ref{l5} we use moment
calculations of the increments of $(\Psi_\lambda^{12}(X_t))_{t\geq 0}$
and $(\widehat\Psi_\lambda^{12}(X_t))_{t\geq 0}$ up to fourth
order. The calculations are presented in an algorithmic way such that
higher moments can be computed along the same lines. The fundamental
idea is that all computations that need to be done are \emph{linear
  maps} on the vector spaces:
\begin{equation}
  \label{eq:V}
  \begin{aligned}
    \mathbb V_\lambda & \coloneqq
    \text{span}\big(\{\Psi_\lambda^{\mathcal I}: \overline{(U,r,\mu)}
    \mapsto \langle \mu^{\otimes \N},
    \textstyle\prod\limits_{(i,j)\in\mathcal I} e^{-\lambda
      r(u_i,u_j)}\rangle: \mathcal
    I \text{ finite subset of }\mathbb N^2\}\big),\\
    \widehat{\mathbb V}_\lambda & \coloneqq \text{span}\big(\{\widehat
    \Psi_\lambda^{\mathcal I}: \overline{(U,r,\mu)} \mapsto \langle
    \mu^{\otimes \N}, \textstyle\prod\limits_{(i,j)\in\mathcal
      I}\ind{a_i=a_j}e^{-\lambda r(u_i,u_j)}\rangle: \mathcal I \text{
      finite subset of }\mathbb N^2\}\big).
  \end{aligned}
\end{equation}
We point out that one advantage of using matrix algebra is that it is possible
to use computer algebra software such as \textsc{Mathematica} for automated
computations.

Apparently, a \emph{basis} of the vector spaces $\mathbb V_\lambda$
and $\widehat{\mathbb V}_\lambda$ are given by (recall from
Definition~\ref{def:PsiPhi})
\begin{align}
  &\underline \Psi_\lambda \coloneqq (\Psi_\lambda^\emptyset,
  \Psi_\lambda^{12}, \Psi_\lambda^{12,12},
  \Psi_\lambda^{12,23}, \Psi_\lambda^{12,34},\dots),\\
  &\underline{\widehat\Psi}_\lambda \coloneqq
  (\widehat\Psi_\lambda^\emptyset, \widehat\Psi_\lambda^{12},
  \widehat\Psi_\lambda^{12,12}, \widehat\Psi_\lambda^{12,23},
  \widehat\Psi_\lambda^{12,34},\dots),
\end{align}
respectively. By inspection of e.g.\ Lemma~\ref{l1} it is clear, that
the generator $\Omega$ gives rise to linear operators $\mathbb
V_\lambda \to \mathbb V_\lambda$ and $\widehat{\mathbb V}_\lambda \to
\widehat{\mathbb V}_\lambda$.  The bases $\underline \Psi_\lambda$ and
$\underline{\widehat\Psi}_\lambda$ can be ordered such that the
matrices representing the linear maps $\Omega$ on $\mathbb V_\lambda$
and $\widehat{\mathbb V}_\lambda$ are \emph{upper triangular}; see
also Lemma~\ref{l1}.

In the next section, we give the action of $\Omega$ on the first~36 basis
vectors for each basis. The last basis vectors in these collections are given by
\begin{equation}\label{ag11}
  \Psi_\lambda^{12,34,56,78}: \overline{(U,r,\mu)} \mapsto \langle
  \mu^{\otimes \N},e^{-\lambda r_{12,34,56,78}} \rangle,
\end{equation}
where we write $r_{12,34,56,78}\coloneqq r(u_1,u_2) + r(u_3,u_4) +
r(u_5,u_6) + r(u_7, u_8) $, and
\begin{equation}\label{ag12}
  \widehat\Psi_\lambda^{12,34,56,78}: \overline{(U,r,\mu)} \mapsto
  \langle \mu^{\otimes \N}, \ind{a_1=a_2, a_3=a_4,a_5=a_6,a_7=a_8}
  e^{-\lambda r_{12,34,56,78} }\rangle,
\end{equation}
respectively. Because of $(\Psi_\lambda^{12})^4 =
\Psi_\lambda^{12,34,56,78}$ and $(\widehat\Psi_\lambda^{12})^4 =
\widehat\Psi_\lambda^{12,34,56,78}$ both 36's basis elements play
important roles.

\subsubsection*{The matrix representation of the generator $\mathbf \Omega$}
Here, we give the matrix representation of the generator $\Omega$ in
terms of the basis $\underline{\Psi}_\lambda$, i.e.\ we find matrices
$A$ and $\widehat A$ such that
\begin{align}\label{ag13}
  \Omega\underline{\Psi}_\lambda^\top =
  A\underline{\Psi}_\lambda^\top, \qquad
  \Omega\underline{\widehat\Psi}_\lambda^\top = \widehat
  A\underline{\widehat\Psi}_\lambda^\top.
\end{align}
The following list contains the action on the first 36 basis vectors
from $\underline\Psi_\lambda$ and
$\underline{\widehat\Psi}_\lambda$.

In 4.\ as an example (see below) we are dealing with the function
\begin{equation}\label{ag15}
  \Psi^{12,34}_\lambda (\mbox{ and } \widehat\Psi^{12,34}_\lambda), \mbox{ for
    which we use \vier as abbreviation.}
\end{equation}
As this symbol indicates, in $\Psi^{12,34}_\lambda$, two
non-overlapping pairs are sampled.

We have already seen in \eqref{eq:189} that
\begin{align}
  \Omega\Psi_\lambda^{12,34} & = -(2\lambda+6)\Psi_\lambda^{12,34} +
  2\Psi_\lambda^{12} + 4\Psi_\lambda^{12,23},\\
  \Omega\widehat\Psi_\lambda^{12,34} & = -(2\lambda+4\vartheta +
  6)\widehat\Psi_\lambda^{12,34} + 2\widehat\Psi_\lambda^{12} + 4
  \widehat \Psi_\lambda^{12,23}.
\end{align}
Our list only contains the last two terms (which are equal in both
equations) which arise from \vier by merging of a pair of leaves,
i.e.\ merging for example the marked leaves in
\parbox[b]{1cm}{
  \beginpicture
  \setcoordinatesystem units <.5cm,.5cm>
  \setplotarea x from 0 to 3, y from 0 to 0.5
  \circulararc -180 degrees from 0 0 center at 0.5 0
  \circulararc -180 degrees from 1.5 0 center at 2 0
  \put{$\bullet$} [cC] at 1 0
  \put{$\bullet$} [cC] at 1.5 0
  \endpicture} to \drei or
merging the marked leaves in
\parbox[b]{1cm}{
  \beginpicture
  \setcoordinatesystem units <.5cm,.5cm>
  \setplotarea x from 0 to 3, y from 0 to 0.5
  \circulararc -180 degrees from 0 0 center at 0.5 0
  \circulararc -180 degrees from 1.5 0 center at 2 0
  \put{$\bullet$} [cC] at 1 0
  \put{$\bullet$} [cC] at 0 0
  \endpicture} to \eins\!\!\!. The coefficient of
$\Psi_\lambda^{12,34}$ in $\Omega\Psi_\lambda^{12,34}$ (and of
$\widehat\Psi_\lambda^{12,34}$ in
$\Omega\widehat\Psi_\lambda^{12,34}$) is given such that the sum of
coefficients in $\Omega\Psi_\lambda^{12,34}$ equals $\lambda$ times
the number of pairs (i.e.\ 2) in $\Psi_\lambda^{12,34}$ (and
$\vartheta$ times the number of different indices (i.e.\ 4)). Since
the table would be the same for the $\widehat\Psi_\lambda^k$'s, we
restrict ourselves to the $\Psi_\lambda^k$'s.
\begin{enumerate}\setcounter{enumi}{-1}
\item $\Psi_\lambda^\emptyset\equiv$\nul$\to$ ---
\item $\Psi_\lambda^{12}\equiv$ \eins$\to$ 1x\nul
\item $\Psi_\lambda^{12,12}\equiv$ \zwei$\to$ 1x\nul
\item $\Psi_\lambda^{12,23}\equiv$ \drei$\to$ 2x\eins 1x\zwei
\item $\Psi_\lambda^{12,34}\equiv$ \vier$\to$ 2x\eins 4x\drei
\item $\Psi_\lambda^{12, 12, 12}\equiv$ \fuenf$\to$ 1x\nul
\item $\Psi_\lambda^{12,12,23}\equiv$ \sechs$\to$ 1x\eins 1x\zwei 1x\fuenf
\item $\Psi_\lambda^{12,13,23}\equiv$ \sieben$\to$ 3x\zwei
\item $\Psi_\lambda^{12,12,34}\equiv$\acht$\to$ 1x\eins 1x\zwei 4x\sechs
\item $\Psi_\lambda^{12,23,24}\equiv$\neun$\to$ 3x\drei 3x\sechs
\item $\Psi_\lambda^{12,23,34}\equiv$\einsnull$\to$ 3x\drei 2x\sechs 1x\sieben
\item $\Psi_\lambda^{12,23,45}\equiv$\einseins$\to$ 1x\drei 2x\vier
  1x\acht 2x\neun \\[1ex] \hspace*{2.5cm} 4x\einsnull
\item $\Psi_\lambda^{12,34,56}\equiv$\einszwei$\to$ 3x\vier 12x\einseins
\item $\Psi_\lambda^{12,12,12,12}\equiv$\einsdrei$\to$ 1x\nul
\item $\Psi_\lambda^{12,12,12,23}\equiv$\einsvier$\to$ 1x\eins 1x\fuenf 1x\einsdrei
\item $\Psi_\lambda^{12,12,23,23}\equiv$\einsfuenf$\to$ 2x\zwei 1x\einsdrei
\item $\Psi_\lambda^{12,12,13,23}\equiv$\einssechs$\to$ 1x\zwei 2x\fuenf
\item $\Psi_\lambda^{12,12,23,34}\equiv$\einssieben$\to$ 1x\drei 2x\sechs
  1x\einsvier 1x\einsfuenf 1x\einssechs
\item $\Psi_\lambda^{12,23,23,34}\equiv$\einsacht$\to$ 1x\drei 2x\sechs
  2x\einsvier 1x\einssechs
\item $\Psi_\lambda^{12,13,23,34}\equiv$\einsneun$\to$ 3x\sechs 1x\sieben
  2x\einssechs
\item $\Psi_\lambda^{12,23,34,14}\equiv$\zweinull$\to$ 4x\sieben
  2x\einssechs
\item $\Psi_\lambda^{12,12,23,24}\equiv$\zweieins$\to$ 1x\drei 2x\sechs
  2x\einsvier 1x\einsfuenf
\item $\Psi_\lambda^{12,12,34,34}\equiv$\zweizwei$\to$ 2x\zwei 4x\einsfuenf
\item $\Psi_\lambda^{12,12,12,34}\equiv$\zweidrei$\to$ 1x\eins 1x\fuenf
  4x\einsvier
\item $\Psi_\lambda^{12,23,34,45}\equiv$\zweivier$\to$ 4x\einsnull
  2x\einssieben 1x\einsacht  \\ \hspace*{2.5cm} 2x\einsneun
  1x\zweinull
\item $\Psi_\lambda^{12,23,34,25}\equiv$\zweifuenf$\to$ 2x\neun
  2x\einsnull 1x\einssieben \\ \hspace*{2.5cm} 2x\einsacht 2x\einsneun
  1x\zweieins
\item $\Psi_\lambda^{12,23,24,25}\equiv$\zweisechs$\to$ 4x\neun 6x\zweieins
\item $\Psi_\lambda^{12,12,34,45}\equiv$\zweisieben$\to$ 1x\drei 2x\acht
  4x\einssieben 2x\zweieins \\ \hspace*{2.5cm} 1x\zweizwei
\item $\Psi_\lambda^{12,12,23,45}\equiv$\zweiacht$\to$ 1x\vier 1x\sechs
  1x\acht 2x\einssieben \\ \hspace*{2.5cm} 2x\einsacht 2x\zweieins
  1x\zweidrei
\item $\Psi_\lambda^{12,23,13,45}\equiv$\zweineun$\to$ 1x\sieben 3x\acht
  6x\einsneun
\item $\Psi_\lambda^{12,23,45,56}\equiv$\dreinull$\to$ 4x\einseins
  4x\zweivier 4x\zweifuenf \\ \hspace*{2.5cm} 1x\zweisechs
  2x\zweisieben
\item $\Psi_\lambda^{12,12,34,56}\equiv$\dreieins$\to$ 1x\vier 2x\acht
  4x\zweisieben \\ \hspace*{2.5cm} 8x\zweiacht
\item $\Psi_\lambda^{12,23,24,56}\equiv$\dreizwei$\to$ 1x\neun 3x\einseins
  6x\zweifuenf \\ \hspace*{2.5cm} 2x\zweisechs 3x\zweiacht
\item $\Psi_\lambda^{12,23,34,56}\equiv$\dreidrei$\to$ 1x\einsnull
  3x\einseins 4x\zweivier \\ \hspace*{2.5cm} 4x\zweifuenf
  2x\zweiacht 1x\zweineun
\item $\Psi_\lambda^{12,23,45,67}\equiv$\dreivier$\to$ 2x\einseins
  2x\einszwei \\ \hspace*{2.5cm} 4x\dreinull 1x\dreieins 4x\dreizwei
  \\ \hspace*{2.5cm} 8x\dreidrei
\item $\Psi_\lambda^{12,34,56,78}\equiv$\dreifuenf$\to$ 4x\einszwei
  24x\dreivier
\end{enumerate}
For the matrices $A$ (and $\widehat A$) representing $\Omega$ in terms
of the basis $\underline \Psi_\lambda$ (and
$\underline{\widehat\Psi}_\lambda$), we give here only the first 13
(5) rows and columns for $A$ ($\widehat A$), which are dealing with samples of
at most three (two) pairs (i.e.\ $|\mathcal I|\leq 3$ ($\leq 2$)
in~\eqref{eq:V}). For $A$, we get
{\tiny
  \begin{align*}
    -A \coloneqq \left(
      \begin{array}{ccccccccccccc}
        0&  0&  0&  0&  0&  0&  0&  0&  0&  0&  0&  0&  0\\
        -1&  \lambda + 1&  0&  0&  0&  0&  0&  0&  0&  0&  0&  0&  0\\
        -1&  0& 2\lambda + 1&  0&  0&  0&  0&  0&  0&  0&  0&  0&  0\\
        0&  2&  -1&  2 \lambda + 3&  0&  0&  0&  0&  0&  0&  0&  0&  0\\
        0&  2&  0&  -4& 2 \lambda + 6&  0&  0&  0&  0&  0&  0&  0&  0\\
        -1&  0&  0&  0&  0&  3 \lambda + 1&  0&  0&  0&  0&  0&  0&  0\\
        0&  -1&  -1&  0&  0&  -1&  3 \lambda + 3&  0&  0&  0&  0&  0&  0\\
        0&  0&  -3&  0&  0&  0&  0&  3 \lambda + 3&  0&  0&  0&  0&  0\\
        0&  -1&  -1&  0&  0&  0&  -4&  0&  3 \lambda + 6&  0&  0&  0&  0\\
        0&  0&  0&  -3&  0&  0&  -3&  0&  0&  3 \lambda + 6&  0&  0&  0\\
        0&  0&  0&  -3&  0&  0&  -2&  -1&  0&  0&  3 \lambda + 6&  0&  0\\
        0&  0&  0&  -1&  -2&  0&  0&  0&  -1&  -2&  -4&  3 \lambda + 10&  0\\
        0&  0&  0&  0&  -3&  0&  0&  0&  0&  0&  0&  -12& 3 \lambda + 15\\
      \end{array}\right)
  \end{align*}
}
while for $\widehat A$ we have
\begin{align*}
  -\widehat A \coloneqq \left(
    \begin{array}{ccccccccccccc}
      0&  0&  0&  0&  0 \\
      -1&  \lambda + 2\vartheta + 1&  0&  0&  0 \\
      -1&  0& 2\lambda + 2\vartheta + 1&  0&  0 \\
      0&  2&  -1&  2 \lambda + 3\vartheta + 3& 0\\
      0&  2&  0&  -4& 2 \lambda + 4\vartheta + 6
    \end{array}\right)
\end{align*}

\noindent
Note that both matrices are diagonalisable. The reason is that the
submatrix containing the same eigenvalue (take $3\lambda+6$ in $A$
above as an example) is diagonal. Hence, there are three independent
eigenvectors for this eigenvalue and so, we find a basis of
eigenvectors.

In the rest of this section, the calculations for the
$\Psi^k_\lambda$'s and $\widehat\Psi^k_\lambda$'s are completely
analogous using the $\widehat\Upsilon_\lambda^k$'s instead of the
$\Upsilon_\lambda^k$'s. Hence, we only present the calculations
concerning $\Psi^k_\lambda$'s.

\subsubsection*{Conditioning as a linear map}
During our calculations, we need to be able to compute terms like
$\mathbf E\Big[\sum_{k=1}^n a_k \Psi_\lambda^k(X_t)|\mathcal F_s\Big]$
for $s\leq t$. We do this using the following procedure. We have
given the following objects:
\begin{enumerate}
\item A basis
  \begin{align}\label{eq:Phibasis}
    \underline \Upsilon_\lambda \coloneqq (\Upsilon_\lambda^\emptyset,
    \Upsilon_\lambda^{12}, \Upsilon_\lambda^{12,12},\dots)
  \end{align}
  such that $\Omega\Upsilon_\lambda^k = \lambda_k \Upsilon_\lambda^k$,
  i.e.\ $\Upsilon_\lambda^k$ is an eigenvector of the generator
  $\Omega$ for the eigenvalue $\lambda_k$ (where $k\in\{\emptyset; 12;
  12,12; 12,23; 12,34;\dots\}$. As argued below, these eigenvectors
  exist, since $A$ is diagonalisable.
\item A diagonal matrix $D_{t-s}$ with diagonal entries
  $e^{-\lambda_k(t-s)}$ in the $k$th line. Since the $\lambda_k$'s are
  given by 1., this matrix is readily obtained.
\item An invertible matrix
  $M_{\underline\Psi_\lambda}^{\underline\Upsilon_\lambda} =
  (M_{\underline\Upsilon_\lambda}^{\underline\Psi_\lambda} )^{-1}$,
  such that $\underline \Psi_\lambda
  M_{\underline\Psi_\lambda}^{\underline\Upsilon_\lambda}=
  \underline\Upsilon_\lambda$ and $\underline \Upsilon_\lambda
  M_{\underline\Upsilon_\lambda}^{\underline\Psi_\lambda} =
  \underline\Psi_\lambda$. These two matrices accomplish a change of
  basis from $\underline{\Psi}_\lambda$ to
  $\underline{\Upsilon}_\lambda$ and back.

  Since $\underline\Upsilon_\lambda$ is the basis of eigenvectors of the
  matrix $A$, the matrices
  $M_{\underline\Psi_\lambda}^{\underline\Upsilon_\lambda}$ and
  $M_{\underline\Upsilon_\lambda}^{\underline\Psi_\lambda}$ are obtained
  by standard linear algebra. We stress that both matrices are lower
  triangular since $A$ is lower triangular.
\end{enumerate}
Then, because $(e^{\lambda_k t}\Upsilon_\lambda^k(t))_{t\geq 0}$ are
martingales, (compare with Lemma~\ref{l1}),
\begin{equation}
  \label{eq:918}
  \begin{aligned}
    \mathbf E\big[\underline\Psi_\lambda^\top(X_t)|\mathcal F_s\big] &
    = M_{\underline\Upsilon_\lambda}^{\underline\Psi_\lambda}\cdot \mathbf
    E[\underline \Upsilon_\lambda(X_t)|\mathcal F_s]^\top =
    M_{\underline\Upsilon_\lambda}^{\underline\Psi_\lambda}\cdot D_{t-s}
    \cdot \underline \Upsilon_\lambda(X_s)^\top \\ & =
    M_{\underline\Upsilon_\lambda}^{\underline\Psi_\lambda}\cdot D_{t-s}
    \cdot M_{\underline\Psi_\lambda}^{\underline\Upsilon_\lambda}\cdot
    \underline \Psi_\lambda(X_s)^\top.
  \end{aligned}
\end{equation}
This means that conditioning linear combinations of elements in
$\underline{\Psi}_\lambda(X_t)$ on $\mathcal F_s$ can be represented
by a composition of linear maps.

\subsubsection*{Multiplication with \boldmath $\Psi_\lambda^{12}$ as a linear
  map}
We need to compute the action of multiplying an element of $\mathbb
V_\lambda$ with $\Psi_\lambda^{12}$. Since $\mathbb V_\lambda$ is
actually an algebra of functions (see Remark~\ref{rem:intPhi} and
Remark~2.8 in \cite{GPWmp}), this can be done. We have that
\begin{align}
  \Psi_\lambda^{12} \underline\Psi_\lambda^\top =
  \Big(\Psi_\lambda^{12}, \Psi_\lambda^{12,34},
  \Psi_\lambda^{12,12,34}, \Psi^{12,23,45},\dots).
\end{align}
This means that there is a linear matrix $Q$ such that
\begin{align}
  \label{eq:Q}
  \Psi_\lambda^{12}\underline\Psi_\lambda^\top = Q\underline
  \Psi_\lambda^\top.
\end{align}
Apparently, $Q$ is the matrix for a linear map on $\mathbb V_\lambda$,
with respect to the basis $\underline \Psi_\lambda$.

\subsection{Proof of Lemma~\ref{l4b}}
\label{ss.proofsl4l5}

Again, we only give the calculations for the $\Psi_\lambda^k$'s in detail. We
use
\begin{equation}\label{eq:l4b1}
  \begin{aligned}
    ((\lambda+1)\Psi^{12}_\lambda-1)^4 & = (\lambda+1)^4
    \Psi^{12,34,56,78}_\lambda - 4(\lambda+1)^3\Psi^{12,34,56} \\ &
    \qquad\qquad\qquad + 6(\lambda+1)^2\Psi^{12,34}_\lambda + 4
    (\lambda+1)\Psi^{12}_\lambda + 1.
  \end{aligned}
\end{equation}
By construction, we obtain for all $k\in\mathcal I$
\begin{align}
  \mathbf E[\Psi^{k}_\lambda(X_\infty)] = \mathbf
  E[(\underline\Upsilon_\lambda(X_\infty) M_{{\underline
      \Upsilon_{\lambda}}}^{{\underline \Psi_{\lambda}}})_{k}] =
  (M_{{\underline \Upsilon_{\lambda}}}^{{\underline
      \Psi_{\lambda}}})_{\emptyset; k}
\end{align}
since $\mathbf E[\Upsilon^k_\lambda(X_\infty)] = 0$ for
$k\neq\emptyset$ as in~\eqref{eq:EPhi0}. Hence,
\begin{equation}
  \label{eq:961}
  \begin{aligned}
    (\lambda+1)\Psi^{12}_\lambda(X_\infty) & = 1,\\
    (\lambda+1)^2\Psi^{12,34}_\lambda(X_\infty) & = \frac{4
      \lambda^4+ 26 \lambda^3 + 49 \lambda^2 + 36 \lambda + 9 }{4
      \lambda^4+ 24 \lambda^3 + 47 \lambda^2 + 36 \lambda + 9} = 1 +
    \frac{1}{2\lambda}
    - \frac{5}{2\lambda^2} + \mathcal O\Big(\frac{1}{\lambda^3}\Big),\\
    (\lambda+1)^3\Psi^{12,34,56}_\lambda(X_\infty) & = \frac{36
      \lambda^7+ 618 \lambda^6 + 4143 \lambda^5 + \mathcal
      O(\lambda^4)}{36 \lambda^7+ 564 \lambda^6 + 3487 \lambda^5 +
      \mathcal O(\lambda^4)} \\ & = 1 + \frac{3}{2\lambda} -
    \frac{95}{18\lambda^2} +
    \mathcal O\Big(\frac{1}{\lambda^3}\Big),\\
    (\lambda+1)^4\Psi^{12,34,56,78}_\lambda(X_\infty) & =
    \frac{36864 \lambda^{16}+ 1536000 \lambda^{15} + 28807680
      \lambda^{14} + \mathcal O(\lambda^{13}) }{ 36864 \lambda^{16}+
      1425408 \lambda^{15} + 24729088 \lambda^{14} + \mathcal
      O(\lambda^{13}) } \\ & = 1 + \frac{3}{\lambda} -
    \frac{193}{36\lambda^2} + \mathcal
    O\Big(\frac{1}{\lambda^3}\Big).
  \end{aligned}
\end{equation}
Plugging the last expressions in~\eqref{eq:l4b1} gives the result.

\subsection{Proof of Lemma~\ref{l5}}
\label{ss:56}
Now we can compute the fourth moments of the increments of $t \mapsto
(\lambda+1)\Psi^{12}(X_t)$.  Again, we only give the proof of the
first assertion. The second follows by replacing the
$\Psi_\lambda^k$'s by the $\widehat\Psi_\lambda^k$'s. Using
Minkowski's inequality and the fact that we start in equilibrium we
have
\begin{equation}
  \label{eq:5365}
  \begin{aligned}
    \bigl(\mathbf E[(\Psi^{12}_\lambda(X_t) -
    \Psi^{12}_\lambda(X_0))^4] \bigr)^{1/4} & \le \bigl( \mathbf
    E[(\Psi^{12}_\lambda(X_t))^4] \bigr)^{1/4}
    +   \bigl( \mathbf E[(\Psi^{12}_\lambda(X_0))^4]      \bigr)^{1/4} \\
    & = 2 \bigl( \mathbf E[\Psi^{12,34,56,78}_\lambda(X_0)]
    \bigr)^{1/4}
  \end{aligned}
\end{equation}
and therefore, since $\mathbf E[\Psi^{12,34,56,78}_\lambda(X_0)] =
\mathcal O (\lambda^{-4})$ as $\lambda \to \infty$ by~\eqref{eq:961},
we have
\begin{align}
  \label{eq:5366}
  (\lambda +1)^4 \mathbf E[(\Psi^{12}_\lambda(X_t) -
  \Psi^{12}_\lambda(X_0))^4] & \le 2^4 (\lambda+1)^4 \mathbf
  E[\Psi^{12,34,56,78}_\lambda(X_0)] = \mathcal O(1).
\end{align}
Thus, the expression on the left hand side of \eqref{eq:310} is
\emph{bounded}.

Now, using
\begin{align}
  (a-b)^4 = a^4-b^4 - 4b(a^3-b^3) + 6b^2(a^2-b^2)-4b^3(a-b)
\end{align}
we obtain,
\begin{equation}
  \label{eq:976}
  \begin{aligned}
    \mathbf E[& (\Psi^{12}_{\lambda}(X_t) - \Psi^{12}_{\lambda}(X_0))^4] \\ &
    = \mathbf E[\Psi^{12,34,56,78}_\lambda(X_t) -
    \Psi^{12,34,56,78}_\lambda(X_0)] \\ & \qquad \qquad - 4\mathbf
    E\big[\Psi^{12}_\lambda(X_0)\big(\Psi^{12,34,56}_\lambda(X_t) -
    \Psi^{12,34,56}_\lambda(X_0)\big)\big] \\ & \qquad \qquad \qquad \qquad +
    6 \mathbf E\big[\Psi^{12,34}_\lambda(X_0)\big(\Psi^{12,34}_\lambda(X_t) -
    \Psi^{12,34}_\lambda(X_0)\big)\big] \\ & \qquad \qquad \qquad \qquad
    \qquad \qquad - 4 \mathbf
    E[\Psi^{12,34,56}_\lambda(X_0)\big(\Psi^{12}_\lambda(X_t) -
    \Psi^{12}_\lambda(X_0)\big)\big] \\ & = -4 \int_0^t \mathbf
    E[\Psi^{12}_\lambda(X_0) \Omega \Psi^{12,34,56}_\lambda(X_s)]ds +
    6\int_0^t \mathbf
    E[\Psi^{12,34}_\lambda(X_0)\Omega\Psi^{12,34}_\lambda(X_s)]ds \\
    & \qquad \qquad \qquad \qquad \qquad \qquad \qquad \qquad \quad - 4
    \int_0^t \mathbf
    E[\Psi^{12,34,56}_\lambda(X_0)\Omega\Psi^{12}_\lambda(X_s)] ds,
  \end{aligned}
\end{equation}
because $\Psi^{12,34,56,78}_{\lambda}(X_t) \stackrel d =
\Psi^{12,34,56,78}_{\lambda}(X_0)$, since we start in
equilibrium.

Let us consider the expectation in the second term on the right hand side. We
have
\begin{equation}
  \label{eq:977}
  \begin{aligned}
    \mathbf E[\Psi^{12,34}_\lambda(X_0)\Omega\Psi^{12,34}_\lambda(X_s)] & =
    \mathbf E[\Psi^{12,34}_\lambda(X_0)\mathbf
    E[\Omega\Psi^{12,34}_\lambda(X_s)|\mathcal F_0]] \\ & = \mathbf
    E[\Psi^{12,34}_\lambda(X_0)\big(A \cdot \mathbf
    E[\underline{\Psi}_\lambda(X_s)|\mathcal F_0]\big)_{12,34}] \\ & = \mathbf
    E[\Psi^{12,34}_\lambda(X_0)\big(A \cdot M_{\underline
      \Upsilon_\lambda}^{\underline\Psi_\lambda} \mathbf
    E[\underline{\Upsilon}_\lambda(X_s)|\mathcal F_0]\big)_{12,34}] \\ & = \mathbf
    E[\Psi^{12,34}_\lambda(X_0)\big(A \cdot M_{\underline
      \Upsilon_\lambda}^{\underline\Psi_\lambda} D_s M_{\underline
      \Psi_\lambda}^{\underline\Upsilon_\lambda}
    \underline{\Psi}_\lambda(X_0)\big)_{12,34}] \\ & = \mathbf
    E[\Psi^{12}_\lambda(X_0)\big(A \cdot M_{\underline
      \Upsilon_\lambda}^{\underline\Psi_\lambda} D_s M_{\underline
      \Psi_\lambda}^{\underline\Upsilon_\lambda}
    Q\underline{\Psi}_\lambda(X_0)\big)_{12,34}] \\ & = \mathbf E[\big(A \cdot
    M_{\underline \Upsilon_\lambda}^{\underline\Psi_\lambda} D_s M_{\underline
      \Psi_\lambda}^{\underline\Upsilon_\lambda} Q Q
    \underline{\Psi}_\lambda(X_0)\big)_{12,34}].
  \end{aligned}
\end{equation}
The expression in the last line can easily be integrated in $s$ over $[0,t]$,
because only $D_s$ depends on $s$. In addition, $X_0$ is in equilibrium and
hence, the expectation can be evaluated using the equilibrium distribution of
$\mathcal X$. All three terms in the right hand side of \eqref{eq:976} can be
computed analogously. Using \textsc{Mathematica}, we see that \eqref{eq:310}
holds.

\section{Proof of Lemma~\ref{L.reform},
  Propositions~\ref{P:smallballs} and~\ref{P.fluc} and
  Theorems~\ref{T3} and~\ref{T4}}
\label{S:proofT3}

\subsection{Proof of Proposition~\ref{P:smallballs}}
By Lemma~\ref{L.reform} we have to show that~\eqref{eq:convPsi12}
holds.  Recall $\Psi_\lambda^{12}$ from~\eqref{eq:int}. We abbreviate
$\Psi_\lambda^{12} \coloneqq \Psi_\lambda^{12}(X_\infty)$ and compute,
using Lemma~\ref{l4b}
\begin{equation}\label{eq:fourth}
  \begin{aligned}
    \lambda^2\mathbf E\bigl[\bigl( (\lambda+1)\Psi_\lambda^{12} -
    1\bigr)^4\bigr] \xrightarrow{\lambda\to\infty} \frac{3}{4}.
  \end{aligned}
\end{equation}
Therefore, there is $C>0$, such that
\begin{align}
  \mathbf P(|(\lambda+1)\Psi_\lambda^{12} - 1|>\varepsilon) & \leq
  \frac{\mathbf E[((\lambda+1)\Psi_\lambda^{12} -
    1)^4]}{\varepsilon^4} \leq \frac{C}{\varepsilon^4 \lambda^2}.
\end{align}
It follows that almost surely $|(n+1)\Psi^{12}_n - 1|>\varepsilon$ for
at most finitely many $n$ by the Borel--Cantelli lemma. Thus,
\begin{align}
  1 = \lim_{\lambda\to\infty} \lfloor \lambda \rfloor
  \Psi_{\lceil\lambda\rceil}^{12} \leq \liminf_{\lambda\to\infty}
  (\lambda+1)\Psi_\lambda^{12} \leq \limsup_{\lambda\to\infty}
  (\lambda+1)\Psi_\lambda^{12} \leq
  \lim_{\lambda\to\infty}\lceil\lambda+1\rceil
  \Psi_{\lfloor\lambda\rfloor}^{12}=1
\end{align}
holds almost surely and the result is shown.

\subsection{Proof of Lemma~\ref{L.reform}}
Recall that~\eqref{eq:convPsi12} is the short form of
\begin{align}
  (\lambda+1) \int \mu_\infty^{\otimes 2}(d(u_1,u_2)) e^{-\lambda
    r(u_1, u_2)} \xrightarrow{\lambda\to\infty} 1.
\end{align}
Furthermore, since $ F_i(\varepsilon)$ is the
probability of sampling a leaf from $B_i(\varepsilon)$
under the sampling measure (and hence,
$F_i(\varepsilon)^2$ is the probability of picking two
leaves from $B_i(\varepsilon)$), for every realisation
of $N_\ve, F_i$ and $\mu_\infty$ we have
\begin{align}
  \frac1\ve \sum_{i=1}^{N_\ve}F_i(\ve)^2 & =
  \frac{1}{\ve} \int \mu_\infty^{\otimes 2}(d(u_1,u_2))
  \ind{r(u_1,u_2)\leq \ve} = \frac
  1{\ve} \nu[0,\varepsilon].
\end{align}
Here we set $\nu \coloneqq r( \cdot,\cdot)_\ast \mu_\infty^{\otimes
  2}$, i.e.\ the distribution of the distance of two randomly chosen
leaves. Thus, the assertion of Lemma~\ref{L.reform} is the equivalence
of $\lambda\int e^{-\lambda x} \nu(dx)
\xrightarrow{\lambda\to\infty}1$ and $\frac 1\varepsilon
\nu[0,\varepsilon] \xrightarrow{\varepsilon \to 0} 1$. This
equivalence is implied by classical Tauberian Theorems, as e.g.\ given
in Theorem~3, Chapter XIII.5 of \cite{Fel66}.

\subsection{Proof of Proposition~\ref{P.fluc}}
We proceed in three steps.
\begin{itemize}
\item \textbf{Step 1:} Tightness of one-dimensional distributions
  of $\mathcal Z^\lambda$.
\item \textbf{Step 2:} Moments and covariance structure.
\item \textbf{Step 3:} Tightness of $\mathcal Z^\lambda$ in path-space
  (in the space of continuous functions).
\end{itemize}

\medskip
\noindent
\emph{Step 1: Tightness of one-dimensional distributions of $\mathcal
  Z^\lambda$.}  We obtain from Remark~\ref{rem:reference},
\begin{equation}
  \label{eq:5367}
  \begin{aligned}
    \mathbf E[(Z_t^\lambda)^2] & = \lambda \mathbf E[(\lambda t+1)^2
    \Psi_{\lambda t}^{12,34} - 2(\lambda t+1)\Psi^{12}_{\lambda t}+1]
    \xrightarrow{ \lambda\to\infty} \frac1{2t},
  \end{aligned}
\end{equation}
which shows the claimed tightness.

\medskip
\noindent
\emph{Step 2: Moments and covariance structure.}  For the covariance
structure, we have to compute moments of
$\Psi^{12}_{s\lambda}\Psi^{12}_{t\lambda}$ which are not included in
the manuscript up to here. We set
\begin{equation}
  \label{eq:21c}
  \begin{aligned}
    \Psi_{\lambda, \lambda'}^{12,23}(\overline{(U,r,\mu)}) & \coloneqq
    \langle \mu^{\otimes \N}, e^{-\lambda r(u_1, u_2)
      - \lambda' r(u_2,u_3)}\rangle,\\
    \Psi_{\lambda, \lambda'}^{12,34}(\overline{(U,r,\mu)}) & \coloneqq
    \langle \mu^{\otimes \N}, e^{-\lambda r(u_1, u_2) -
      \lambda'r(u_3,u_4)}\rangle,\\
    \Upsilon^{12}_\lambda & \coloneqq  \Psi_\lambda^{12} - \frac{1}{\lambda+1},\\
    \Upsilon_{\lambda,\lambda'}^{12,23} & \coloneqq  \Psi_{\lambda,
      \lambda'}^{12,23} - \frac{1}{2} \Psi_{\lambda+\lambda'}^{12} -
    \frac{1}{\lambda+2}\Psi_\lambda^{12} -
    \frac{1}{\lambda'+2}\Psi_{\lambda'}^{12} \\ & \qquad \qquad
    \qquad \qquad \qquad \qquad \qquad \qquad + \frac{\lambda
      \lambda' + 4 \lambda + 4 \lambda' + 12}{2 (\lambda + 2)
      (\lambda' + 2) (\lambda + \lambda' + 3)},\\
    \Upsilon_{\lambda,\lambda'}^{12,34} & \coloneqq \Psi_{\lambda, \lambda'}^{12,34} -
    \frac 43 \Psi_{\lambda, \lambda'}^{12,23} + \frac{4}{15}
    \Psi_{\lambda + \lambda'}^{12} + \frac{1}{3(\lambda+5)}
    \Psi_\lambda^{12}+ \frac{1}{3(\lambda'+5)} \Psi_{\lambda'}^{12}
    \\ & \qquad \qquad \qquad \qquad \qquad \qquad \qquad \qquad -
    \frac{ 4 \lambda \lambda' + 25 \lambda + 25 \lambda' +150}{15
      (\lambda + 5) (\lambda' + 5) (\lambda + \lambda' + 6)}.
  \end{aligned}
\end{equation}
Then
\begin{equation}
  \label{eq:189c}
  \begin{aligned}
    \Omega\Psi_{\lambda, \lambda'}^{12,23} & = -(\lambda + \lambda'
    + 3)\Psi_{\lambda, \lambda'}^{12,23} +
    \Psi_\lambda^{12} + \Psi_{\lambda'}^{12} + \Psi_{\lambda+\lambda'}^{12},\\
    \Omega\Psi_{\lambda, \lambda'}^{12,34} & = -(\lambda + \lambda'
    + 6)\Psi_{\lambda, \lambda'}^{12,34} + 4\Psi_{\lambda,
      \lambda'}^{12,23} +
    \Psi_\lambda^{12} + \Psi_{\lambda'}^{12},\\
    \Omega\Upsilon_\lambda^{12} & = -(\lambda+1) \Upsilon_\lambda^{12},\\
    \Omega\Upsilon_{\lambda, \lambda'}^{12,23} & = -(\lambda + \lambda' +3)
    \Upsilon_{\lambda, \lambda'}^{12,23},\\
    \Omega\Upsilon_{\lambda, \lambda'}^{12,34} & = -(\lambda + \lambda'
    +6) \Upsilon_{\lambda, \lambda'}^{12,34}.
  \end{aligned}
\end{equation}
Now, in order to compute the covariance structure and higher moments,
we write, again using \textsc{Mathematica},
\begin{equation}
  \label{eq:31d}
  \begin{aligned}
    \mathbf E[Z^\lambda_s Z^\lambda_t] & = \lambda \mathbf
    E[(s\lambda + 1)(t\lambda + 1)\Psi^{12,34}_{s\lambda,
      t\lambda}-1] \\ & = \frac{4 s t \lambda^3 }{((s+t)\lambda+1)
      ((s+t)\lambda+3) ((s+t)\lambda+6)} \\ &
    \xrightarrow{\lambda\to\infty}\frac{4st}{(s+t)^3}.
  \end{aligned}
\end{equation}
For the third moment,
\begin{align}
  \label{eq:31e}
  \mathbf E[(Z^\lambda_t)^3] = \lambda^{3/2}\frac{16 t^3\lambda^3
    (5t^2\lambda^2 + 9 t\lambda - 10)}{(t\lambda + 2)(t\lambda + 3)
    (t\lambda + 5) (2t\lambda+1) (2t\lambda+3) (3t\lambda+1) (3
    t\lambda+10)} \xrightarrow{\lambda\to\infty} 0.
\end{align}
The fourth moment was already given in~\eqref{eq:fourth}.

\medskip
\noindent
\emph{Step 3: Tightness of $\mathcal Z^\lambda$ in path-space.}  Here,
we show that there exists $C>0$, which is independent of $\lambda$,
such that
\begin{align}
  \label{eq:31}
  \sup_{\lambda > 0} \lambda \cdot \mathbf E[((t\lambda+1)
  \Psi^{12}_{t\lambda}(X_\infty) - (s\lambda+1)\Psi^{12}_{s\lambda}(X_\infty))^2]\leq C(t-s)^2
\end{align}
for $0<s\leq t$. Tightness then follows from Step~1 and the
Kolmogorov--Chentsov criterion. In order to show~\eqref{eq:31}, we
simplify the notation and suppress the dependency on $X_\infty$.  From
\eqref{eq:21c}, we read off $\mathbf E[\Psi^{12,34}_{s\lambda,
  t\lambda}]$. Then, the result follows from
\begin{equation}
  \label{eq:31c}
  \begin{aligned}
    ((t\lambda+1)\Psi^{12}_{t\lambda} & -
    (s\lambda+1)\Psi^{12}_{s\lambda})^2 \\ & =
    (t\lambda+1)^2\Psi^{12,34}_{t\lambda, t\lambda} - 2
    (s\lambda+1)(t\lambda+1)\Psi^{12,34}_{s\lambda, t\lambda} +
    (s\lambda+1)^2\Psi^{12,34}_{s\lambda, s\lambda},
  \end{aligned}
\end{equation}
with an application of \textsc{Mathematica} and \eqref{eq:31} is
shown.

\subsection{Proof of Theorem~\ref{T3}}
\label{ss:proofT3}
As in the proof of Theorem~\ref{T1}, we only need to show the assertion in the
case $\alpha=0$ due to absolute continuity recalled in Proposition~\ref{P:main}.
In this case, the proof of the theorem requires the following three steps:
\begin{itemize}
\item \textbf{Step 1:} Instead of starting in $\mathcal X_0$, it suffices
  to start in equilibrium and then show~\eqref{eq:small1}.
\item \textbf{Step 2:} Assume that $\mathcal X_0$ is in equilibrium, i.e.\
  $X_0\stackrel d = X_\infty$.
  Then, the set of processes
  \begin{align}
    \label{eq:67}
    \{((\lambda+1)\Psi^{12}_\lambda(X_t)-1)_{t\geq 0}:\lambda>0\}
  \end{align}
  is tight in $\mathcal C_{\R}([0,\infty))$.
\item \textbf{Step 3:} The finite-dimensional distributions of $
  \{((\lambda+1)\Psi^{12}_\lambda(X_t)-1)_{t\geq 0}:\lambda>0\} $
  converge to~0 as $\lambda\to\infty$.
\end{itemize}
These steps imply that the object in \eqref{eq:small1}
converges to $0$ in distribution, which is (in the case of convergence to a
constant) equivalent to convergence in probability.

\medskip
\noindent
\emph{Step 1: Start in equilibrium.} Let $\widetilde{\mathcal X} =
(\widetilde X_t)_{t\geq 0}$ with $\widetilde X_t =
\overline{(\widetilde U_t, \widetilde r_t, \widetilde \mu_t)}$ be the
tree-valued Fleming--Viot process started in equilibrium. Then,
$\mathcal X$ and $\widetilde{\mathcal X}$ can be coupled such that for
$\varepsilon>0$
\begin{align}\label{dg12}
  \langle \widetilde \mu_t^{\otimes\mathbb N}, \phi\rangle = \langle
  \mu_t^{\otimes\mathbb N}, \phi\rangle
\end{align}
for all $\phi: \mathbb R^{\binom{\mathbb N}{2}}\to\mathbb R$ which
depends only on elements in $\underline{\underline r}$ with values at
most $t$.

Recall the Moran model approximation. Observe that distances evolve
with time of rate $t$ and may change by resampling
events by being reset to zero.  The coupling arises in this model for
every $N$ by taking the same resampling events for $\mathcal X$ and
$\widetilde{\mathcal X}$ which defines the two processes on a common
probability space. For this coupling,
\begin{equation}
  \label{eq:54}
  \begin{aligned}
    \sup_{\varepsilon\leq t \leq \tau}(\lambda+1)
    |\Psi^{12}_\lambda(X_t) - \Psi^{12}_\lambda(\widetilde X_t)| & =
    \sup_{\varepsilon\leq t \leq \tau}
    (\lambda+1)|\langle\mu_t^{\otimes\mathbb
      N}-\widetilde\mu_t^{\otimes\mathbb N},
    \ind{r(u_1,u_2)>t}e^{-\lambda r(u_1,u_2)}\rangle | \\
    & \leq \langle\mu_t^{\otimes\mathbb N}, \lambda e^{-\lambda
      \varepsilon}\rangle + \langle\widetilde \mu_t^{\otimes\mathbb
      N}, \lambda e^{-\lambda \varepsilon}\rangle = 2\lambda
    e^{-\lambda\varepsilon} \xrightarrow{\lambda\to\infty} 0.
  \end{aligned}
\end{equation}
Taking now on this common probability space the limit $N \to \infty$
results in the coupled laws.

Hence, it suffices to prove the assertion when started in equilibrium,
i.e.\ 1.\ holds for all assertions which concern properties which depend only
on distances below a threshold and in particular all limiting properties close to the leaves.

\medskip
\noindent
\emph{Step 2: Tightness of
  $\{((\lambda+1)\Psi^{12}_\lambda(X_t)-1)_{t\geq 0}:\lambda>0\}$.}
This is clearly implied by Lemma~\ref{l5} and the Kolmogorov--Chentsov
criterion for tightness in $\mathcal C_{\R}([0,\infty))$.

\medskip
\noindent
\emph{Step 3: Convergence of finite-dimensional distributions to~0.}
This follows from Lemma~\ref{l2}.

\subsection{Proof of Theorem~\ref{T4}}
\label{ss.proofT4}
We proceed in several steps.
\begin{itemize}
\item {\bf Step 0}: Warm up; computation of first two moments of
  $W_\lambda(t)- W_\lambda(s)$.
\item {\bf Step 1}: The family $(W_\lambda)_{\lambda>0}$ is tight in
  $\mathcal C_{\R}([0,\infty))$.
\item {\bf Step 2}: If $(W(t))_{t\geq 0}$ is a limit point, then
  $(W(t))_{t\geq 0}$ and $(W(t)^2-t)_{t\geq 0}$ are both martingales.
\end{itemize}

\noindent Throughout we let $(\mathcal F_t)_{t\geq 0}$ be the
canonical filtration of the process $(X_t)_{t\geq 0}$.

\medskip

\noindent
\emph{Step 0: Computation of first two moments of $W_\lambda(t) -
  W_\lambda(s)$.}  We start with some basic computations which we will
need in the sequel. First, recall that $(\lambda + 1)\Upsilon^{12}_\lambda
= (\lambda+1)\Psi^{12}_\lambda -1$ and by Lemma~\ref{l1}
\begin{equation}
  \label{eq:338a}
  \begin{aligned}
    \mathbf E[(\lambda+1)\Upsilon^{12}_\lambda(X_t)|\mathcal F_s] =
    e^{-\lambda (t-s)}\cdot (\lambda+1)\Upsilon^{12}_\lambda(X_s).
  \end{aligned}
\end{equation}
Then, by Fubini's theorem
\begin{equation}
  \label{eq:339}
  \begin{aligned}
    \mathbf E[W_\lambda(t) - W_\lambda(s)|\mathcal F_s] & = \lambda
    \int_s^t \mathbf E[(\lambda+1)\Upsilon^{12}_\lambda(X_r)| \mathcal
    F_s]\, dr \\ & = \lambda \int_s^t e^{-\lambda (r-s)} \cdot
    (\lambda+1)\Upsilon^{12}_\lambda(X_s) \, dr \\ & = (1-e^{-\lambda
      (t-s)})\cdot (\lambda+1)\Upsilon^{12}_\lambda(X_s).
  \end{aligned}
\end{equation}
\sloppy We compute
\begin{equation}
  \label{eq:339a}
  \begin{aligned}
    ((\lambda+1)\Upsilon^{12}_\lambda)^2 & = (\lambda+1)^2 \Psi^{12,34} -
    2(\lambda+1)\Psi^{12}_\lambda + 1 \\ & = (\lambda+1)^2
    \Big(\Upsilon_\lambda^{12,34} + \frac 43 \Upsilon_\lambda^{12,23} +
    \frac{2}{5} \Upsilon_\lambda^{12,12} +
    \frac{2(\lambda+6)}{(\lambda+2)(\lambda+5)}\Upsilon_\lambda^{12} \\
    & \qquad \qquad \qquad \qquad +
    \frac{4\lambda^2+18\lambda+9}{(\lambda+1)(\lambda+3)(2\lambda+1)(2\lambda+3)}\Big)
    - 2(\lambda+1)\Upsilon^{12}_\lambda - 1 \\ & = (\lambda+1)^2\Big(
    \Upsilon^{12,34}_\lambda + \frac 43 \Upsilon_\lambda^{12,23} + \frac{2}{5}
    \Upsilon_\lambda^{12,12}\Big) - \frac{8 (\lambda +
      1)}{(\lambda+2)(\lambda+5)} \Upsilon^{12}_\lambda \\ & \qquad \qquad
    \qquad \qquad \qquad \qquad \qquad \qquad + \frac{2
      \lambda^2}{(\lambda + 3) (2\lambda + 1) (2\lambda + 3)},
  \end{aligned}
\end{equation}
which already implies that
\begin{align}\label{eq:339b}
  \mathbf E[W_\lambda(t) - W_\lambda(s)|\mathcal F_s]
  \xrightarrow{\lambda\to\infty} 0 \quad \text{ in }L^2,
\end{align}
since we started in equilibrium. Hence $\mathbf E[\Upsilon^k_\lambda(X_0)]
= 0$ as in Lemma~\eqref{eq:EPhi0}.

Next, we come to the second moment
\begin{equation}
  \label{eq:483}
  \begin{aligned}
    \mathbf E[(W_\lambda(t) & - W_\lambda(s))^2|\mathcal F_s] =
    \lambda^2 \cdot \mathbf E\Big[\Big(\int_s^t (\lambda+1)^2
    \Upsilon^{12}_\lambda(r)\, dr\Big)^2\Big|\mathcal F_s\Big] \\ & =
    2\lambda^2 \cdot \int_s^t \int_{r_1}^t\mathbf E\big[
    (\lambda+1)\Upsilon^{12}_\lambda(X_{r_1}) \mathbf
    E[(\lambda+1)\Upsilon^{12}_\lambda(X_{r_2})|\mathcal F_{r_1}] \big|
    \mathcal F_{s}\big]\, dr_2\, dr_1 \\ & = 2\lambda^2 \cdot \int_s^t
    \int_{r_1}^t\mathbf E\big[ (\lambda+1)^2\Upsilon^{12}_\lambda(X_{r_1})
    e^{-(\lambda+1)(r_2-r_1)} \Upsilon^{12}_\lambda(X_{r_1}) \big|
    \mathcal F_{s}\big] \, dr_2\, dr_1 \\ & = 2\frac{\lambda^2}{\lambda+1}
    \int_s^t (1-e^{-(\lambda+1)(t-r_1)}) \cdot \mathbf
    E[((\lambda+1)\Upsilon^{12}_\lambda(X_{r_1}))^2|\mathcal F_s] \, dr_1 \\
    & = t-s + A_{\lambda},
  \end{aligned}
\end{equation}
with
{\allowdisplaybreaks
  \begin{align}
    \label{eq:484}
   \begin{split}
    A_\lambda & =  A_{\lambda,1} + A_{\lambda,2} + A_{\lambda,3},\\
    A_{\lambda,1} & =   \int_s^t
    \Big(\frac{4\lambda^4}{(\lambda+1)(\lambda + 3) (2\lambda + 1)
      (2\lambda + 3)} -1\Big) \, dr_1,\\
    A_{\lambda,2} & =   \frac{2\lambda^2(\lambda+1)}{2\lambda+6}\cdot
    \Big( (1-e^{-(2\lambda+6)(t-s)}) \Upsilon^{12,34}_\lambda(X_s) \\
    &
    \qquad \qquad + \frac{4(2\lambda+6)}{3(2\lambda+3)}
    (1-e^{-(2\lambda+3)(t-s)}) \Upsilon^{12,23}_\lambda(X_s)\\
    & \qquad \qquad \qquad \qquad +
    \frac{2(2\lambda+6)}{5(2\lambda+2)}(1-e^{-(2\lambda+2)(t-s)})
    \Upsilon^{12,12}_\lambda(X_s) \\
    & \qquad \qquad \qquad \qquad \qquad
    \qquad -
    \frac{8(2\lambda+6)}{\lambda^2(\lambda+2)(\lambda+5)}(1-e^{-(\lambda+1)(t-s)})\Upsilon^{12}_\lambda(X_s)\Big)
  \\ & =   A_{\lambda,21} + A_{\lambda,22} + A_{\lambda,23},\\*
    A_{\lambda,3} & = {2\lambda^2}{\lambda+1} \int_s^t
    e^{-(\lambda+1)(t-r_1)} \cdot \mathbf
    E[((\lambda+1)\Upsilon^{12}_\lambda(X_{r_1}))^2|\mathcal F_s] \, dr_1,
\end{split}
\end{align}}
where
{\allowdisplaybreaks
  \begin{align*}
    A_{\lambda,21} & = \frac{\lambda^2(\lambda+1)}{\lambda+3}
    \Big(\Psi^{12,34}_\lambda(X_s) -
    \frac{4\lambda^2+18\lambda+9}{(\lambda+1)(\lambda+3)(2\lambda+1)(2\lambda+3)}\Big),\\*
    A_{\lambda,22} & = - \frac{2 (\lambda+1)(\lambda^3 + 6 \lambda^2 +
      8 \lambda + 24)}{(\lambda+3)(2 + \lambda) (5 + \lambda)}
    \Upsilon^{12}_\lambda(X_s)
    ,\\
    A_{\lambda,23} & = -
    e^{-(2\lambda+6)(t-s)}\frac{2\lambda^2(\lambda+1)}{2\lambda+6}
    \Upsilon^{12,34}_\lambda(X_s) -
    \frac{8\lambda^2(\lambda+1)}{3(2\lambda+3)} e^{-(2\lambda+3)(t-s)}
    \Upsilon^{12,23}_\lambda(X_s) \\*
    & \qquad -
    e^{-(2\lambda+2)(t-s)}\frac{4\lambda^2(\lambda+1)}{5(2\lambda+2)}
    \Upsilon^{12,12}_\lambda(X_s) +
    \frac{2(\lambda+1)}{(\lambda+2)(\lambda+5)}
    e^{-(\lambda+1)(t-s)}\Upsilon^{12}_\lambda(X_s).
  \end{align*}} Since
\begin{equation}
  \label{eq:842}
  \begin{aligned}
    A_{\lambda,1} & \xrightarrow{\lambda\to\infty} 0,\\
    A_{\lambda,21}, A_{\lambda,22} & \xrightarrow{\lambda\to\infty} 0 \quad\text{ in }L^2,\\
    A_{\lambda,23}, A_{\lambda,3} & \xrightarrow{\lambda\to\infty} 0
    \quad\text{ in }L^1,
  \end{aligned}
\end{equation}
we have
\begin{align}\label{eq:340b}
  \mathbf E[(W_\lambda(t) - W_\lambda(s))^2|\mathcal F_s] &
  \xrightarrow{\lambda\to\infty} t-s \quad\text{ in }L^1.
\end{align}

\medskip

\noindent\emph{Step 1: The family $(W_\lambda)_{\lambda>0}$ is tight
  in $\mathcal C_{\R}([0,\infty) )$.} Again, we use the Kolmogorov--Chentsov
criterion. To this end we bound the fourth moments of the increments
in $W_\lambda$ by
\begin{equation}
  \label{eq:31f}
  \begin{aligned}
    \mathbf E&[ (W_\lambda(t) - W_\lambda(s))^4] \\ & = \lambda^4
    \cdot \mathbf E\Big[ \Big( \Upsilon^{12}_\lambda(X_t) -
    \Upsilon^{12}_\lambda(X_0) - \int_0^t (\lambda+1)
    \Upsilon^{12}_\lambda(X_s)ds - \Psi^{12}_\lambda(X_t) +
    \Psi^{12}_\lambda(X_0)\Big)^4\Big] \\ & \lesssim \lambda^4 \Big(
    \mathbf E[\Big(\Upsilon^{12}_\lambda(X_t) -
    \Upsilon^{12}_\lambda(X_0) -
    \int_0^t (\lambda+1) \Upsilon^{12}_\lambda(X_s)ds \Big)^4\Big] \\
    & + \mathbf E\big[(\Psi^{12}_\lambda(X_t) -
    \Psi^{12}_\lambda(X_0))^4\big]\Big).
  \end{aligned}
\end{equation}
The second term is bounded by $Ct^2$ for some $C>0$ which is
independent of $\lambda$ by Lemma~\ref{l5}. For the first term, we use
the Burkholder--Davis--Gundy inequality and write (recall
Lemma~\ref{l1} and the quadratic variation of the semimartingale
$\Psi^{12}_\lambda(\mathcal X) = (\Psi^{12}_\lambda(X_t))_{t\geq 0}$
from Remark~\ref{rem:qv}) by
\begin{equation}\label{eq:int1}
  \begin{aligned}
    \lambda^4 & \cdot \mathbf E \Big[\Big(\Upsilon^{12}_\lambda(X_t) -
    \Upsilon^{12}_\lambda(X_0) + \int_0^t (\lambda+1)
    \Upsilon^{12}_\lambda(X_s)ds \Big)^4\Big] \\ & \lesssim \lambda^4
    \mathbf E\big[ [\Psi^{12}_\lambda(\mathcal X)]_t^2\big] \\ & =
    \lambda^4 \int_0^t \int_0^s \mathbf E\big[
    (\Psi^{12,23}_\lambda(X_r) - \Psi^{12,34}_\lambda(X_r)) \cdot
    \mathbf E[ \Psi^{12,23}_\lambda(X_s) - \Psi^{12,34}_\lambda(X_s)
    |\mathcal F_r]\big] \, dr \, ds.
  \end{aligned}
\end{equation}
We have to show that the integrand is of order $1/\lambda^4$, if $X_0$
is in equilibrium.

First, we compute the conditional expectation using
Lemma~\ref{l1}.2 using~\eqref{eq:16} and obtain
\begin{equation}
  \label{eq:31g}
  \begin{aligned}
    \mathbf E[ & \Psi^{12,23}_\lambda(X_s) - \Psi^{12,34}_\lambda(X_s)
    |\mathcal F_r] \\ & = \mathbf E\Big[\frac{1}{10}
    \Upsilon^{12,12}_\lambda(X_s) -\frac 13 \Upsilon^{12,23}_\lambda(X_s) -
    \Upsilon^{12,34}_\lambda(X_s) - \frac{2}{(\lambda+2)(\lambda+5)}
    \Upsilon^{12}_\lambda(X_s)\Big|\mathcal F_r\Big] \\ & \qquad \qquad
    \qquad \qquad \qquad \qquad \qquad \qquad \qquad \qquad +
    \frac{\lambda^2}{(\lambda+1)(\lambda+3)(2\lambda+1)(2\lambda+3)}
    \\ & = e^{-(2\lambda+1)(s-r)} \frac{1}{10}
    \Upsilon^{12,12}_\lambda(X_r) -e^{-(2\lambda+3)(s-r)}\frac 13
    \Upsilon^{12,23}_\lambda(X_r) -
    e^{-(2\lambda+6)(s-r)}\Upsilon^{12,34}_\lambda(X_r) \\ & \qquad \qquad
    - e^{-(\lambda+1)(s-r)}\frac{2}{(\lambda+2)(\lambda+5)}
    \Upsilon^{12}_\lambda(X_r) +
    \frac{\lambda^2}{(\lambda+1)(\lambda+3)(2\lambda+1)(2\lambda+3)}.
  \end{aligned}
\end{equation}
Abbreviating $\Upsilon^k_\lambda := \Upsilon^k_\lambda(X_\infty)$,
note that by~\eqref{eq:sometimes} and the last display, the integrand
in~\eqref{eq:int1} is a linear combination of terms of the form
$\mathbf E[(\Psi^{12,23}_\lambda -
\Psi^{12,34}_\lambda)(\Psi^k_\lambda - \mathbf E[\Psi^k_\lambda])]$
for $k\in\{12; 12,12; 12,23; 12,34\}$. We compute all these terms:
\begin{equation}
  \label{eq:31h}
  \begin{aligned}
    \mathbf E\Big[& (\Psi^{12,23}_\lambda -
    \Psi^{12,34}_\lambda)\Big(\Psi^{12}_\lambda -
    \frac{1}{\lambda+1}\Big)\Big] \\ & = \mathbf E[\Psi^{12,23,45} -
    \Psi^{12,34,56}] - \frac{1}{\lambda+1}\mathbf
    E[\Psi^{12,23}_\lambda - \Psi^{12,34}_\lambda]\\ & = \frac{4
      \lambda^3 (5 \lambda^2 + 9 \lambda -10)}{(\lambda+1)^2
      (\lambda+3) (\lambda+5) (4 \lambda^2 + 8 \lambda + 3 ) (9
      \lambda^3 + 51 \lambda^2 + 76 \lambda + 20)} \\ & = \mathcal O\Big(\frac{1}{\lambda^4}\Big),
  \end{aligned}
\end{equation}
\begin{equation}
  \label{eq:31ha}
  \begin{aligned}
    \mathbf E\Big[& (\Psi^{12,23}_\lambda -
    \Psi^{12,34}_\lambda)\Big(\Psi^{12,12}_\lambda -
    \frac{1}{2\lambda+1}\Big)\Big] \\ & = \mathbf
    E[\Psi^{12,12,34,45}_\lambda - \Psi^{12,12,34,56}] -
    \frac{1}{2\lambda+1} \mathbf E[\Psi^{12,23}_\lambda -
    \Psi^{12,34}_\lambda] \\ & = \frac{2 \lambda^3 (1576 \lambda^6 + \mathcal O(\lambda^5))}
    {3 (\lambda + 3) (2 \lambda + 1)^2 (4 \lambda + 15) (2 \lambda^2 +
      5 \lambda + 3)^2 (96 \lambda^5 + \mathcal O(\lambda^4))}
    \\ & = \mathcal O\Big(\frac{1}{\lambda^4}\Big),
  \end{aligned}
\end{equation}
\begin{equation}
  \label{eq:31hb}
  \begin{aligned}
    \mathbf E\Big[& (\Psi^{12,23}_\lambda -
    \Psi^{12,34}_\lambda)\Big(\Psi^{12,23}_\lambda -
    \frac{5\lambda+3}{(\lambda+1) (2\lambda+1) (2\lambda+3)}\Big)\Big]
    \\ & = \mathbf E\Big[\Psi^{12,23,45,56}_\lambda -
    \Psi^{12,23,45,67}] - \frac{5\lambda+3}{(\lambda+1) (\lambda+3)
      (2\lambda+1) (2\lambda+3)}\mathbf E[\Psi^{12,23}_\lambda -
    \Psi^{12,34}_\lambda] \\ & = \frac{2 \lambda^2 (683976 \lambda^9 +
      \mathcal O(\lambda^8))}
    {9 (\lambda+3) (2\lambda+1)^2 (2\lambda^2+5\lambda+3)^2 (4
      \lambda^2 + 41 \lambda + 105) (1152 \lambda^7 + \mathcal O(\lambda^6))}
    \\ & = \mathcal
    O\Big(\frac{1}{\lambda^5}\Big),
  \end{aligned}
\end{equation}
\begin{equation}
  \label{eq:31hc}
  \begin{aligned}
    \mathbf E\Big[& (\Psi^{12,23}_\lambda -
    \Psi^{12,34}_\lambda)\Big(\Psi^{12,34}_\lambda - \frac{4 \lambda^2
      + 18 \lambda + 9}{(\lambda+1) (\lambda+3) (2\lambda+1)
      (2\lambda+3)}\Big)\Big] \\ & = \mathbf
    E[\Psi^{12,23,45,67}_\lambda - \Psi^{12,34,56,78}_\lambda] -
    \frac{4 \lambda^2 + 18 \lambda + 9}{(\lambda+1) (\lambda+3)
      (2\lambda+1) (2\lambda+3)}\mathbf E[\Psi^{12,23}_\lambda -
    \Psi^{12,34}_\lambda] \\ & = \frac{4 \lambda^2 (20480 \lambda^{11}
      + \mathcal O(\lambda^{10}))}
    {(\lambda+3)^2 (\lambda+7) (2\lambda+1)^2 (2 \lambda^2+ 5 \lambda
      + 3 )^2 ( 4608 \lambda^9 + \mathcal O(\lambda^8))}
    \\ & = \mathcal O\Big(\frac{1}{\lambda^5}\Big),
  \end{aligned}
\end{equation}
which shows that the integrand on the right hand side
of~\eqref{eq:int1} is $\mathcal O(1/\lambda^4)$.

Hence, we have shown that there is a constant $C>0$ such that
\begin{align}
  \label{eq:31i}
  \mathbf E[ (W_\lambda(t) - W_\lambda(s))^4] \leq C(t-s)^2
\end{align}
and we have shown tightness of $\{(W_\lambda(t))_{t\geq 0}:
\lambda>0\}$.

\medskip

\noindent\emph{Step 2: If $(W_t)_{t\geq 0}$ is a limit point,
  $(W_t)_{t\geq 0}$ as well as $(W_t^2-t)_{t\geq 0}$ are
  martingales.}
Let $\mathcal W = (W_t)_{t\geq 0}$ be a weak limit point of
$\{(W_\lambda(t))_{t\geq 0}: \lambda>0\}$. The claimed martingale
properties follow by the same arguments as in Step~3 of the proof of
Theorem~\ref{T2}.

\section{Proof of Theorems~\ref{T5} and~\ref{T6}}
\label{S:proofT5}

\subsection{Proof of Theorem~\ref{T5}}
To get a first idea, let us do a little computation. By Fubini's
theorem, dominated convergence theorem and Lemma~\ref{l1}, we get
\begin{equation}
  \label{pp68}
  \begin{aligned}
    \mathbf E\Big[ \int_0^T \lim_{\lambda\to\infty}
    \Psi^{12}_\lambda(X_t) dt\Big] & = \lim_{\lambda\to\infty}
    \int_0^T \mathbf E\big[ \Psi^{12}_\lambda(X_t)\big] dt \\ & =
    \lim_{\lambda\to\infty} \frac{1}{\lambda+1}\int_0^T (1 - \mathbf
    E\big[ \Omega\Psi^{12}_\lambda(X_t)\big])\, dt \\ & =
    \lim_{\lambda\to\infty} \frac{1}{\lambda+1}\Big( T - \mathbf
    E[\Psi^{12}_\lambda(X_T) - \Psi^{12}_\lambda(0)]\Big) = 0,
  \end{aligned}
\end{equation}
thus, since $\mu_t$ has no atom $\iff$ $\lim_{\lambda\to\infty}
\Psi^{12}_\lambda(X_t)=0$,
\begin{align}\label{pp7}
  \mathbf P(\mu_t & \text{ has no atom for \emph{almost} all }0\leq
  t\leq T) = \mathbf P\Big(\int_0^T \ind{\lim_{\lambda\to\infty}
    \Psi^{12}_\lambda(X_t)>0}\, dt=0\Big) = 1.
\end{align}
Our task is to remove the \emph{almost} in the $\mathbf P(.)$ on the
left hand side.

\medskip

As in the proof of Theorem~\ref{T3}, it suffices to show the assertion
if $X_0\stackrel d = X_\infty$. Recall that $\Omega \Psi^{12}_\lambda = 1 -
(\lambda+1)\Psi^{12}_\lambda$ from Lemma~\ref{l1}. Hence, for all
$\lambda>0$, the process $(M_\lambda(t))_{t\geq 0}$ defined as
\begin{align}\label{pp8}
  M_\lambda(t) \coloneqq \Psi^{12}_\lambda (X_t) -
  \Psi^{12}_\lambda(X_0) - \int_0^t (1 - (\lambda +
  1)\Psi^{12}_\lambda(X_s))\, ds,
\end{align}
is a continuous martingale with quadratic variation (recall
Remark~\ref{rem:qv})
\begin{align}\label{pp9}
  [M_\lambda]_t = \int_0^t \bigl(\Psi^{12,23}_\lambda(X_s) -
  \Psi^{12,34}_\lambda(X_s)\bigr) ds.
\end{align}
Using~\eqref{eq:16} and recalling from Lemma~\ref{l1} that the $\Upsilon$'s are
mean-zero martingales
\begin{align}\label{pp10}
  \lim_{\lambda\to\infty} \mathbf E[M_\lambda(t)^2] = \lim_{\lambda \to \infty}
  \int_0^t \mathbf E[\Psi^{12,23}_\lambda(X_s) - \Psi^{12,34}_\lambda(X_s)]\, ds
  = 0.
\end{align}
This implies for all $\varepsilon>0$ by Doob's maximal inequality:
\begin{align}\label{eq:imp1}
  \lim_{\lambda\to\infty}\mathbf P\Big(\sup_{0\leq t\leq T}|M_\lambda(t)|>\varepsilon\Big) =0.
\end{align}

Then we can start calculating as follows. For all $\varepsilon>0$,
using again Fubini's theorem and then~\eqref{eq:43},
\begin{equation}
  \label{eq:imp2}
  \begin{aligned}
    \lim_{\lambda\to\infty} \mathbf P \Big( & \sup_{0\leq t\leq
      T}\Big|\int_0^t (1-(\lambda+1)\Psi^{12}_\lambda(X_s)) \,ds\Big|>
    \varepsilon\Big) \\
    & = \lim_{\lambda\to\infty} \mathbf P \bigg( \sup_{0\leq t\leq T}\Bigr(
    \int_0^t \big(1-(\lambda+1)\Psi^{12}_\lambda(X_s)\big) \, ds \Bigr)^2 >
    \varepsilon^2\bigg) \\ & \leq \lim_{\lambda\to\infty} \mathbf P
    \Big( \sup_{0\leq t\leq T}\int_0^t \big(1 -
    (\lambda+1)\Psi^{12}_\lambda(X_s)\big)^2 \,ds> \varepsilon^2\Big) \\
    & = \lim_{\lambda\to\infty} \mathbf P \Big( \int_0^T \big( 1 -
    (\lambda+1)\Psi^{12}_\lambda(X_s)\big)^2 \, ds> \varepsilon^2\Big) \\
    & \leq \lim_{\lambda\to\infty} \tfrac 1{\varepsilon^2} \int_0^T
    \mathbf E\big[\big( 1 -
    (\lambda+1)\Psi^{12}_\lambda(X_s)\big)^2\big] \,ds\\ & = 0.
  \end{aligned}
\end{equation}

Hence, by \eqref{eq:imp1} and  \eqref{eq:imp2},
there is a subsequence $\lambda_n\uparrow \infty$ with
\begin{align}\label{pp10b}
  \sup_{0 < t\leq T} |M_{\lambda_n}(t)| &\xrightarrow{n\to\infty} 0,\\
  \sup_{0 < t\leq T}\Big|\int_0^t 1 -
  (\lambda_n+1)\Psi^{12}_{\lambda_n}(X_s) \,ds\Big|
  &\xrightarrow{n\to\infty} 0 \nonumber
\end{align}
almost surely. Recall that we can characterize the existence of atoms
by the property whether we can draw two points at distance zero or not
which we can in turn characterize by the $\lambda \to \infty$ limit of
the sample Laplace transform. Combining \eqref{pp8} with \eqref{pp10b}
and with \eqref{pp7} (which allows us to discard the
$\Psi^{12}_{\lambda_n}(0)$) gives
\begin{equation}
  \label{eq:pp11}
  \begin{aligned}
    \mathbf P(\mu_t & \text{ has no atoms for all }0<t\leq T) =
    \mathbf P(\lim_{n\to\infty} \sup_{0 < t\leq T}
    \Psi^{12}_{\lambda_n}(X_t) = 0) \\ & = \mathbf
    P\bigg(\lim_{n\to\infty} \sup_{0 < t\leq T} \Big(M_{\lambda_n}(t) +
    \int_0^t (1 - (\lambda_n+1) \Psi^{12}_{\lambda_n}(X_s))\, ds\Big) = 0\bigg)
    = 1.
  \end{aligned}
\end{equation}

\subsection{Proof of Theorem~\ref{T6}}
\label{S:proofT6}
It turns out that the following simple criterion for existence of a
mark function is useful.

\begin{lemma}[Criterion for mark function]\mbox{}\\
  An   \label{l:char} mmm-space $\overline{(U,r,\mu)} \in \mathbb U_A$ admits a mark
  function if there is a sequence $\varepsilon_n\downarrow 0$ with
  \begin{align}\label{eq:81}
    \lim_{n\to\infty}\mathbf P\big(\mathfrak A_1=\mathfrak
    A_2|r_U(\mathfrak U_1, \mathfrak U_2)<\varepsilon_n\big)=1.
  \end{align}
  where $(\mathfrak U_1, \mathfrak A_1), (\mathfrak U_2, \mathfrak
  A_2)$ are two independent pairs distributed according to
  $\mu$. Equivalently,
  \begin{align}
    \label{eq:equmark}
    \frac{\langle \mu^{\otimes 2}, \ind{a_1=a_2} \ind{r(u_1, u_2)<\varepsilon_n}
      \rangle}{\langle \mu^{\otimes 2}, \ind{r(u_1, u_2)<\varepsilon_n} \rangle}
    \xrightarrow{n\to\infty} 1.
  \end{align}
\end{lemma}

We proceed in three steps to prove Theorem~\ref{T6}.
\begin{itemize}
\item {\bf Step 1}: Proof of Lemma~\ref{l:char}.
\item {\bf Step 2}: An extension of Theorem~\ref{T3}.
\item {\bf Step 3}: Combination of Steps 1 and 2 gives Theorem~\ref{T6}.
\end{itemize}

\medskip

\noindent
\emph{Step 1: Proof of Lemma~\ref{l:char}}.  Since $\mu$ is a
probability measure on $U\times A$, we can write $\mu(du, da) =
(\pi_U)_\ast\mu(du) \otimes K(u, da)$ for some probability kernel $K$
from $U$ to $A$. We have to show that $K(u,da)$ only has a single atom
for $(\pi_U)_\ast \mu$-almost every $u$.

We proceed by contradiction and assume that $K(u,\cdot)$ is not
concentrated on a single atom for a set $U'\subseteq U$ of positive
$(\pi_U)_\ast \mu$-probability, i.e.\ $\mu(U'\times A)=\delta>0$.

Then, for all sequences $\varepsilon_n\downarrow 0$
\begin{align}
  \label{eq:equmark3}
  \limsup_{n\to\infty} \frac{\langle \mu^{\otimes 2}, \ind{a_1=a_2} \ind{u_1,u_2\in
      U', r(u_1, u_2)<\varepsilon_n} \rangle}{\langle \mu^{\otimes 2},
    \ind{u_1,u_2\in U', r(u_1 u_2)<\varepsilon_n} \rangle} = 1 -
  \delta' < 1.
\end{align}
Applying this to the pairs $(\mathfrak A_1, \mathfrak U_1)$ and
$(\mathfrak A_2, \mathfrak U_2)$ gives
\begin{equation}
  \label{eq:equmark2}
  \begin{aligned}
    \limsup_{n\to\infty} & \; [\mathbf P\big(\mathfrak
    A_1=\mathfrak A_2 |r_U(\mathfrak U_1, \mathfrak
    U_2)<\varepsilon_n\big) \\ & = \limsup_{\varepsilon_n\to 0}\mathbf
    P\big(\mathfrak A_1=\mathfrak A_2|r_U(\mathfrak U_1, \mathfrak
    U_2)<\varepsilon_n, \mathfrak U_1, \mathfrak U_2\in U'\big) \cdot
    \mathbf P(\mathfrak U_1, \mathfrak U_2\in U') \\ & \qquad +
    \mathbf P\big(\mathfrak A_1=\mathfrak A_2|r_U(\mathfrak U_1,
    \mathfrak U_2)<\varepsilon_n, \mathfrak U_1\notin U_2 \vee \mathfrak
    U_2\notin U'\big) \cdot \mathbf P(\mathfrak U_1\in U' \vee
    \mathfrak U_2\in U')] \\
    & \leq (1-\delta') \cdot \mathbf
    P(\mathfrak U_1, \mathfrak U_2\in U') + \mathbf P(\mathfrak U_1\notin
    U' \vee \mathfrak U_2\notin U') < 1.
  \end{aligned}
\end{equation}
Hence, \eqref{eq:81} cannot hold if $\mathbf{P} (\mathfrak
U_1,\mathfrak U_2 \in U^\prime)>0$.  This quantity however is by
assumption bounded below by $\delta^2 > 0$ and we have shown
Lemma~\ref{l:char}.

\medskip

\noindent
\emph{Step 2: Extension of Theorem~\ref{T3}}. We will show the
following for $\widehat\Psi_\lambda^{12}$ from~\eqref{eq:int2}: For
all $T<\infty$ and $\varepsilon>0$,
\begin{align}\label{eq:small1b}
  \lim_{\lambda\to\infty} \mathbf P\left(\sup_{\varepsilon\leq t\leq
      T} |(\lambda +1)\widehat \Psi^{12}_\lambda(X_t) - 1| >
    \varepsilon\right) = 0.
\end{align}

\medskip

\noindent
In order to see this, it suffices -- analogously to the proof of
Theorem~\ref{T3} -- to consider $\alpha=0$ and to start in equilibrium
$X_0 \stackrel d = X_\infty$. First, we assume that $\beta(\cdot ,
\cdot)$ is non-atomic, i.e.\ \eqref{eq:beta-non-atomic} from
Section~\ref{S:prep} holds. In this case, we proceed as in Steps~2
and~3 from Section~\ref{ss:proofT3}. These steps rely on
Lemmata~\ref{l2} and~\ref{l5}, and the second assertions of these
lemmata imply \eqref{eq:small1b}. Second, consider the general case,
i.e.\ $\beta(\cdot , \cdot)$ is not necessarily non-atomic. It is
straight-forward to construct a coupling $(X_t, \check X_t)_{t\geq 0}$
such that $(X_t)_{t\geq 0}$ and $(\check X_t)_{t\geq 0}$ use the same
resampling and mutation events, where mutant types in $(X_t)_{t\geq
  0}$ are chosen according to $\beta(\cdot ,\cdot )$, but mutant types
in $(\check X_t)_{t\geq 0}$ are chosen (at the same rate) according to
some non-atomic $\check \beta(\cdot , \cdot )$. If $X_0 = \check X_0$,
it is clear that mutant types in $(\check X_t)_{t\geq 0}$ lead to new
types in any case and thus, the inequality
\begin{align}\label{eq:ineq}
  \widehat\Psi^{12}(\check X_t) \leq \widehat \Psi^{12}(X_t)\leq
  \Psi^{12}(X_t) = \Psi^{12}(\check X_t)
\end{align}
holds for all $t\geq 0$, almost surely. Recall that we have already
shown in Theorem~\ref{T3} that~\eqref{eq:small1b} holds if
$\widehat\Psi^{12}$ is replaced by $\Psi^{12}$, and since
$\check\beta(\cdot ,\cdot )$ is non-atomic, it also holds for $(\check
X_t)_{t\geq 0}$ by our arguments above. Hence, by~\eqref{eq:ineq}, it
also holds for $\widehat\Psi^{12}(X_t)$, i.e.\ we have
shown~\eqref{eq:small1b}.

\medskip

\noindent
\emph{Step 3: Combination of Steps 1 and 2 gives Theorem~\ref{T6}}. Fix
$0<\delta<T<\infty$.  It suffices to show that $\mathcal X_t\in\mathbb
U_A^{\text{mark}}$ for all $\delta\leq t\leq T$. Using
Theorem~\ref{T3} and Step~2 above, take a sequence $\varepsilon_n\downarrow
0$ and set $\lambda_n := 1/\varepsilon_n$ such that
\begin{equation}
  \label{eq:801}
  \begin{aligned}
    &\mathbf P(\lim_{n\to\infty} \lambda_n \langle \mu_t^{\otimes
      2},e^{-\lambda_n r_{12}}\rangle = 1 \text{ for all
    }\delta\leq t\leq T) = 1,\\
    &\mathbf P(\lim_{n\to\infty} \lambda_n \langle \mu_t^{\otimes 2},
    \ind{a_1=a_2}e^{-\lambda_n r_{12}}\rangle = 1 \text{ for all
    }\delta\leq t\leq T) = 1.
  \end{aligned}
\end{equation}
We use Lemma~\ref{l:char} and the Tauberian result from
Lemma~\ref{L.reform} with its obvious extension to
$\widehat\Psi_\lambda^{12}$ to write:
\begin{equation}
  \label{eq:31j}
  \begin{aligned}
    \mathbf P(\mathcal X_t= & \overline{(U_t,r_t,\mu_t)} \in \mathbb
    U_A^{\text{mark}} \text{ for all $\delta \leq t \leq T$}) \\ &
    \geq \mathbf P\Big( \lim_{n\to\infty}\frac{\langle\mu_t^{\otimes
        2}, \ind{a_1=a_2}\ind{r(u_1,u_2)<\varepsilon_n}\rangle}{
      \langle\mu_t^{\otimes 2},
      \ind{r(u_1,u_2)<\varepsilon_n}\rangle}=1 \text{ for all
      $\delta\leq t\leq T$}\Big) \\ & \geq \mathbf P\Big(
    \lim_{n\to\infty} \lambda_n \langle\mu_t^{\otimes 2},
    \ind{a_1=a_2}\ind{r(u_1,u_2)<\varepsilon_n}\rangle = 1 \\ & \qquad
    \qquad \qquad \qquad \text{ and } \lim_{n\to\infty} \lambda_n
    \langle\mu_t^{\otimes 2}, \ind{r(u_1,u_2)<\varepsilon_n}\rangle=1
    \text{ for all $\delta\leq t\leq T$}\Big) \\ & = \mathbf P\Big(
    \lim_{n\to\infty} \lambda_n \langle \mu_t^{\otimes 2},
    \ind{a_1=a_2}e^{-\lambda_n r(u_1, u_2)}\rangle = 1 \\ & \qquad
    \qquad \qquad \qquad \text{ and } \lim_{n\to\infty} \lambda_n
    \langle \mu_t^{\otimes 2}, e^{-\lambda_n r(u_1, u_2)}\rangle =
    1\text{ for all $\delta\leq t\leq T$}\Big) \\ & = 1
  \end{aligned}
\end{equation}
by~\eqref{eq:801}. This concludes the proof of Theorem~\ref{T6}.

\subsubsection*{Acknowledgments}
This research was supported by the DFG through grants Pf-672/6-1 to
AD and PP and through grant GR 876/16-1 to AG.


\end{document}